\documentclass[preprint,12pt]{elsarticle}

\usepackage{amssymb}
\usepackage{amsmath,amsfonts,color,verbatim}
\usepackage{amscd}
\usepackage{bm} 
\usepackage[config,labelfont={rm,rm},textfont=rm]{caption,subfig}
\DeclareMathOperator{\ddiv}{div}
\newcommand{\triplenorm}[1]{\ensuremath{|\!|\!| #1 |\!|\!|}}
\newcommand{\bmz}[1]{\ensuremath{#1}}

\usepackage{amsthm}

\newtheorem{theorem}{Theorem}
\newtheorem{lemma}[theorem]{Lemma}
\newtheorem{corollary}{Corollary}

\newtheorem{remark}[theorem]{Remark}

\numberwithin{figure}{section}
\numberwithin{table}{section}

\newcommand{\pairv}[2]{\ensuremath{\langle#1, #2 \rangle_{V'\times V}}}
\newcommand{\pairq}[2]{\ensuremath{\langle#1, #2  \rangle_{Q'\times  Q}}}

\renewcommand{\pairv}[2]{\ensuremath{\langle#1, #2 \rangle}}
\renewcommand{\pairq}[2]{\ensuremath{\langle#1, #2  \rangle}}

\journal{Computer Methods in Applied Mechanics and Engineering}

\begin{document}

\begin{frontmatter}



\title{Stability and monotonicity for some discretizations of the Biot's consolidation model}


\author[UniZar]{C. Rodrigo\corref{cor1}}\ead{carmenr@unizar.es}
\address[UniZar]{Departamento de Matem\'{a}tica Aplicada,
Universidad de Zaragoza, Zaragoza, Spain}
\author[UniZar]{F.J. Gaspar}\ead{fjgaspar@unizar.es}
\author[Tufts]{X. Hu}\ead{xiaozhe.hu@tufts.edu}
\address[Tufts]{Department of Mathematics, Tufts University, 
Medford, Massachusetts  02155, USA}
\author[PennState]{L.T. Zikatanov}\ead{ludmil@psu.edu}
\address[PennState]{Department of Mathematics, Penn State, 
University Park, Pennsylvania, 16802, USA}
\cortext[cor1]{Corresponding author. Tel.: +34 976762148; E-mail address: carmenr@unizar.es (C. Rodrigo)}

\begin{abstract}
  We consider finite element discretizations of the Biot's consolidation model in
  poroelasticity with MINI and stabilized P1-P1 elements. We analyze
  the convergence of the fully discrete model based on spatial
  discretization with these types of finite elements and implicit
  Euler method in time. We also address the issue related to the
  presence of non-physical oscillations in the pressure approximation
  for low permeabilities and/or small time steps. We show that even in
  1D a Stokes-stable finite element pair fails to provide a monotone
  discretization for the pressure in such regimes. We then introduce a
  stabilization term which removes the oscillations. We present
  numerical results confirming the monotone behavior of the stabilized
  schemes.
\end{abstract}


\end{frontmatter}


\section{Introduction}\label{sec:Intro}
The theory of poroelasticity models the interaction between the
deformation and the fluid flow in a fluid-saturated porous medium.
Such coupling was already modelled in the early one-dimensional work of
Terzaghi,
see~\cite{terzaghi}, whereas the general three-dimensional mathematical model
was established by Maurice Biot in several pioneering publications
(see~\cite{biot1} and~\cite{biot2}).

We assume here that the porous medium is linearly elastic,
homogeneous, isotropic and saturated by an incompressible Newtonian
fluid.  Under these assumptions, the quasi-static Biot's model can be
written as a time-dependent system of partial differential equations
in the variables of displacements of the solid, $u$, and pressure of the
fluid, $p$,
\begin{eqnarray}
&& -\ddiv\bmz{\sigma} + \nabla p = \bmz{f},\qquad
\bmz{\sigma} = 2\mu \varepsilon(\bmz{u})  + \lambda\ddiv(\bmz{u}) I \label{eq:u_poro}\\
&& -\ddiv \bmz{\dot{u}} + \ddiv K \nabla p = g,\label{eq:p_poro}
\end{eqnarray}
where $\sigma$ and $\varepsilon$ are the effective stress and strain
tensors, $\lambda$ and $\mu$ are the Lam\'e coefficients, $K$ is the
hydraulic conductivity tensor, the right-hand term $f$ is the density
of applied body forces and the source term $g$ represents a forced
fluid extraction or injection process. The time derivative of the
displacement vector is denoted by $\dot{u}$. 
Results on the existence and uniqueness of the solution for these
models have been investigated by Showalter in~\cite{showalter}
and by Zenisek in~\cite{zenisek}, and the well-posedness for
nonlinear poroelastic models is considered, for
example, in~\cite{DDL1997}.

Biot's models are still used today in a great variety of fields,
ranging from geomechanics and petroleum engineering, where these models
have been applied ever since their discovery, to biomechanics or even
food processing more recently. Some examples of applications in
geosciences include petroleum production, solid waste disposal, carbon
sequestration, soil consolidation, glaciers dynamics, subsidence,
liquefaction and hydraulic fracturing, for instance. In biomechanics
the poroelastic theory can be used to describe tumor-induced stresses
in the brain (see~\cite{Roose2003204}), which can cause
deformation of the surrounding tissue, and bone deformation under a
mechanical load (see~\cite{Swan2003}), for example.  More recently, a
promising and innovative application studies the
food processes as a multiphase deformable porous media, in order to
improve the quality and safety of the food, see~\cite{Datta2010}.

Although some analytical solutions have been derived for some linear
poroelasticity problems, see~\cite{Coussy2004}, and even some
of them are obtained artificially as in~\cite{BarryMercer},
numerical simulations seem to be the only way to obtain quantitative
results for real applications. The numerical solution of these
problems is usually based on finite element methods, see for example
the monograph of Lewis and Schrefler in~\cite{LewisSchrefler}
and the papers in~\cite{LewisSchrefler78, LewisTran89, LewisSchreflerSimoni, MastersPaoLewis}.
Finite difference methods have been also applied to solve this
problem, see for example the convergence analysis in~\cite{CMAM_2002} and the extension to the discontinuous
coefficients case in~\cite{Lazarov2007, PhDNaumovich}.

It is well-known that approximations by standard finite difference
and finite element methods of the poroelasticity equations often
exhibit strong nonphysical oscillations in the fluid pressure, see for
instance~\cite{Gaspar2003, Ferronato2010, Langtangen2012, Favino2013, WheelerPhillip}.
For example, this is the case when linear finite elements are used to
approximate both displacement and pressure unknowns, or when a central
finite difference scheme on collocated grids is considered.
To eliminate such instabilities, approximation
spaces for the vector and scalar fields, satisfying an appropriate
inf-sup condition (see~\cite{1974BrezziF-aa}) are commonly used. Such
discretizations have been theoretically investigated by Murad et
al. in~\cite{MuradLoula92, MuradLoula94, MuradLoulaThome}.
As we show later, however, an inf-sup stable pair of spaces does not
necessarily provide oscillation-free solutions. On the other hand,
the oscillations disappear on very fine
grids, but evidently, this is not always practical.

Our work here is on investigating mechanisms for avoiding the
nonphysical oscillations in the discrete solution, for example, by
adding stabilization terms to the Galerkin formulation, while still
maintaining the accuracy of approximations. Such strategy
has been applied in~\cite{Aguilar2008} to provide a stable
scheme by using linear finite element approximations for both
unknowns. This was accomplished by adding an artificial term, namely,
the time derivative of a diffusion operator multiplied by a
stabilization parameter, to the flow equation. The stabilization
parameter, which depends on the elastic properties of the solid and on
the characteristic mesh size, was given a priori, and its optimality
was shown in the one-dimensional case. This scheme provided solutions
without oscillations independently of the chosen discretization
parameters.

In this work, we present convergence analysis of fully discrete implicit schemes
for the numerical solution of Biot's consolidation model. We derive
appropriate stabilization terms for both MINI element and P1-P1
discretizations, and numerically show that such choices of
stabilization parameters and operators remove the non-physical
oscillations in the approximations of the pressure. In this regard,
our work fills in a gap in the literature, since to our knowledge the
results presented here are the first theoretical results for fully discrete
schemes involving stabilized spatial discretizations aimed to improve
the monotonicity properties of the finite element schemes.


The rest of the paper is organized as follows. In
Section~\ref{sec:oscillations}, we provide one dimensional example
elements illustrating the undesirable oscillatory pressure
behavior. We show both numerically and theoretically, that adding
appropriate stabilization terms provide monotone discrete schemes and
we calculate the exact values of the optimal stabilization
parameters for both MINI and P1-P1 schemes. In Section~\ref{sec:Schur} we show several abstract
results on stabilized discretizations which we use in
Section~\ref{sec:fully-discrete} to analyze the convergence of the
fully discrete model. The abstract results in Section~\ref{sec:Schur}
apply to more general saddle-point problems with stabilization
terms. In this section, we have also computed the exact Schur
complement corresponding to the bubble functions in the MINI element.
Next, in Section~\ref{sec:fully-discrete} we use the abstract results
and show first order convergence in time and space for the fully
discrete Biot's consolidation model. The section~\ref{sec:numerics} is
devoted to the numerical study of the convergence and monotonicity
properties of the resulting discretizations. We use several benchmark
tests in poromechanics and show that appropriate choice of
stabilization parameters result in approximations which respect the
underlying physical behavior and are oscillation-free. Conclusions are
drawn in~Section~\ref{sec:conclusions}.

\section{Pressure oscillatory behaviour: one dimensional example}\label{sec:oscillations}
We consider an example modeling a column of height $H$ of a porous medium
saturated by an incompressible fluid, bounded by impermeable and rigid
lateral walls and bottom, and supporting a load $\sigma_0$ on the top
which is free to drain. We have the following PDEs describing this model:
\begin{equation} \label{example-int}
\begin{array}{l}
-\displaystyle \frac{\partial}{\partial x} \left( E \, \frac{\partial u}{\partial x}\right) + \displaystyle \frac{\partial p}{\partial x}=0, \\
\displaystyle \frac{\partial }{\partial t}\left (\displaystyle \frac{\partial u}{\partial
x}\right)-\displaystyle \frac{\partial}{\partial x} \left( K \, \frac{\partial p}{\partial x}\right) = 0,
\end{array}
(x,t)\in (0,H)\times(0,T],
\end{equation}
with boundary and initial conditions
$$
\begin{array}{l}
\displaystyle E \, \frac{\partial u}{\partial x}(0,t)=\sigma_0, \quad p(0,t)=0, \ t\in (0,T], \\
u(H,t)=0,  \ \displaystyle K \frac{\partial p}{\partial x}(H,t)=0, \ t\in (0,T], \\
\displaystyle \frac{\partial u}{\partial x}(x,0)=0, \ x\in [0,H],
\end{array}
$$
where $E$ is the Young's modulus and $K$ is the hydraulic
conductivity. It can be easily seen that problem \eqref{example-int}
is decoupled, giving rise to the following heat-type equation for the
pressure
\begin{equation} \label{eq:pressure}
\displaystyle \frac{\partial }{\partial t}\left (\displaystyle \frac{1}{E} \, p \right)-\displaystyle \frac{\partial}{\partial x} \left( K \, \frac{\partial p}{\partial x}\right) = 0.
\end{equation}
In order to discretize problem \eqref{example-int}, we consider a
non-uniform partition of spatial domain $\Omega=(0,H)$,
\[
0=x_0 < x_1 < \ldots < x_{n-1} < x_n = H.
\]
In this way, the domain $\Omega$ is given by the disjoint union of
elements $T_i=[x_i,x_{i+1}], \; 0 \leq i \leq n-1$, of size
$h_i = x_{i+1}-x_i$.  We assume that the Young modulus $E(x)$ and the
hydraulic conductivity $K(x)$ are constants $E_i$ and $K_i$ on each
element $T_i$.  Next, we are going to analyze two discretizations by
two different pairs of finite elements with a backward Euler method in
time.

\subsection{Discretization with linear finite elements}
First, we discretize using linear finite elements for both displacement and
pressure. In this case, the following linear system of equations has
to be solved on each time step
\begin{equation} \label{linear_system}
\left[
\begin{array}{cc}
A_l & G_l \\
G_l^T & \tau A_p
\end{array}
\right] \left[
\begin{array}{c}
U_l^m \\ P^m
\end{array}
\right] = \left[
\begin{array}{cc}
0 & 0 \\
G_l^T & 0
\end{array}
\right]
\left[
\begin{array}{c}
U_l^{m-1} \\ P^{m-1}
\end{array}
\right] +
\left[
\begin{array}{c}
f_l^m \\ 0
\end{array}
\right],
\end{equation}
where $m \geq 1$, and $\tau$ is the time discretization parameter. It is clear that the pressure at time level $m$ must
satisfy the following equation
\begin{equation} \label{Schur}
(C_l + \tau A_p) P^{m} = C_l P^{m-1}  - G_l^T A_l^{-1} (f_l^m -f_l^{m-1}),
\end{equation}
where $C_l=-G_l^T A_l^{-1} G_l$ is a tridiagonal matrix such that for an interior node~$x_i$ it is given by
\begin{equation} \label{C_linear}
(C_l P^m)_i = \frac{1}{4} \left(\displaystyle \frac{h_{i-1}}{E_{i-1}} P_{i-1}^m + \left(\frac{h_{i-1}}{E_{i-1}} + \frac{h_{i}}{E_{i}} \right) P_{i}^m +
\frac{h_{i}}{E_{i}} P_{i+1}^m \right).
\end{equation}
Notice that the scheme associated with the above equation should be an appropriate discretization for problem \eqref{eq:pressure}. Depending on the relation between the space and time discretization parameters, the off-diagonal elements of matrix $C_l + \tau A_p$ could be positive
and therefore the cause of possible non-physical oscillations in the approximation of the pressure. To avoid these instabilities, the following restriction holds,
\begin{equation}\label{Restr1}
\max_{0 \leq i \leq n-1} \displaystyle \frac{h_i^2}{4 K_i E_i} < \tau.
\end{equation}
For example, in the case of an uniform-grid of size $h$ and constant values of the parameters $E$ and $K$ in the whole domain, such restriction becomes
$h^2 < 4 E K \tau$. To confirm these unstable behavior, we solve system \eqref{example-int} in the computational domain $(0,1)$ by using linear finite elements considering $K \, E \, \tau = 10^{-6}$. In this case, it is necessary a mesh of at least $500$ nodes to fulfill the restriction.
\begin{figure}
\begin{tabular}{cc}
\includegraphics[scale = 0.4]{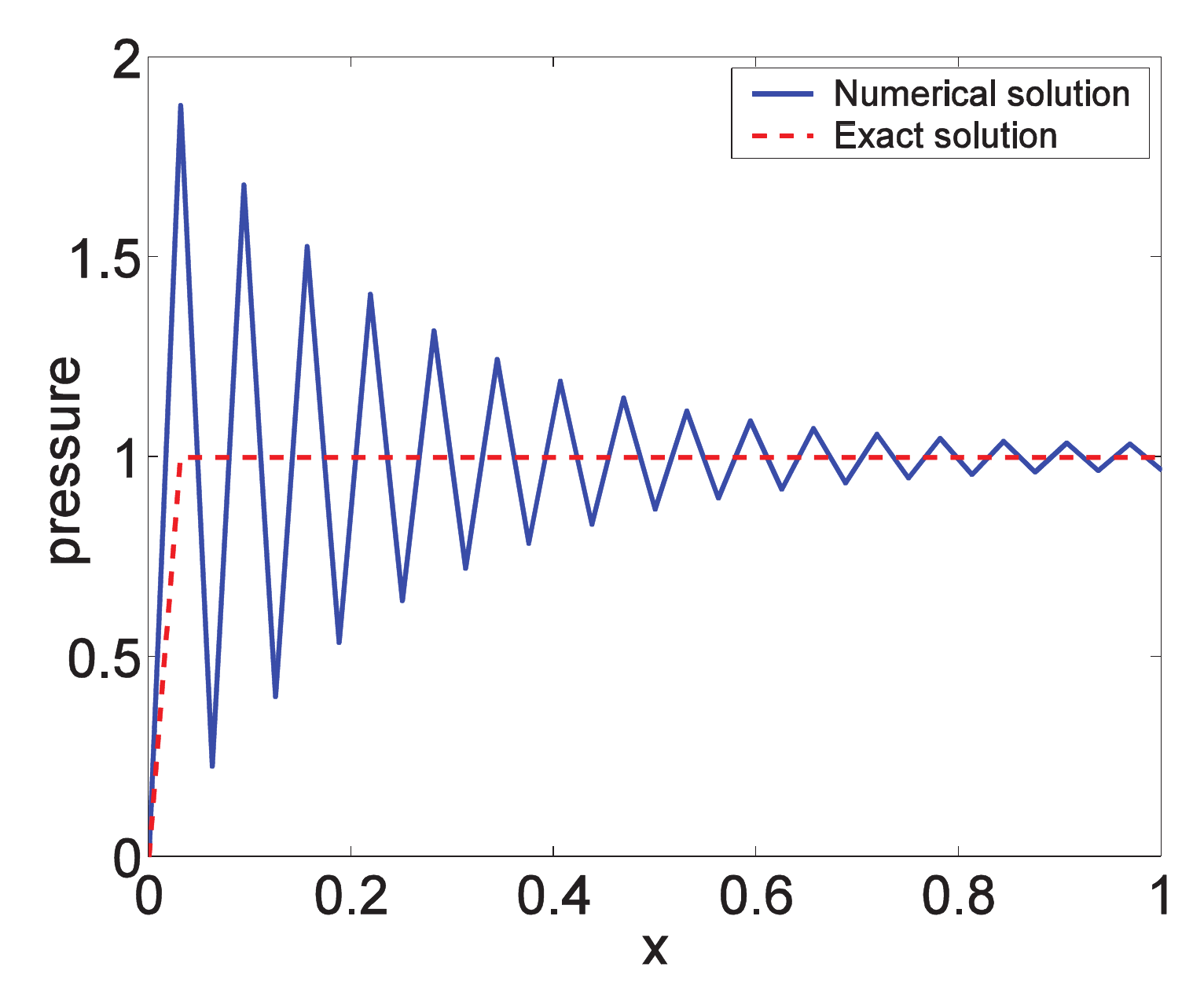}
&
\includegraphics[scale = 0.4]{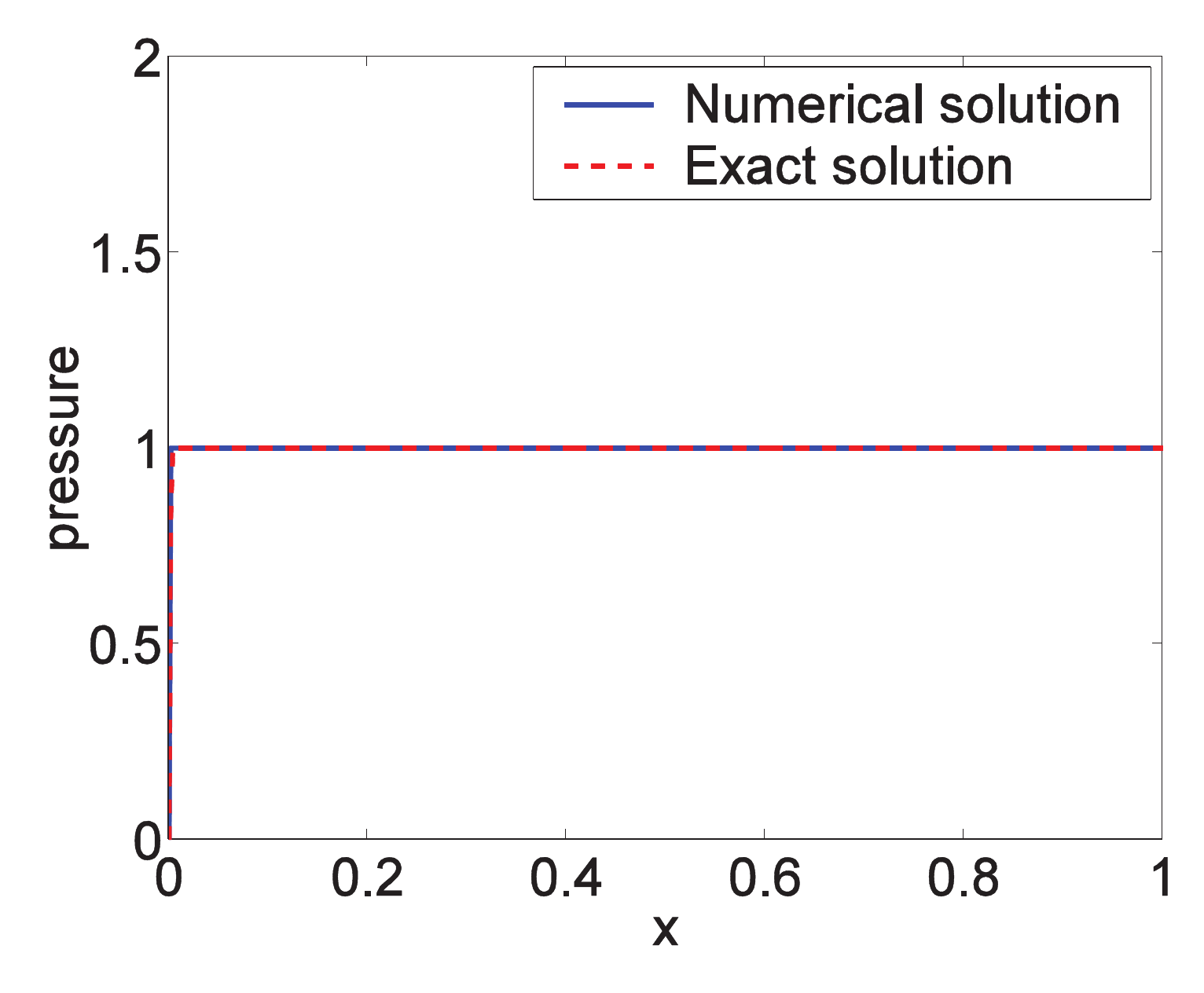}
\\
(a)
&
(b)
\end{tabular}
\caption{Numerical solution for the pressure field obtained with finite elements P1-P1 and corresponding exact solution for (a) $h=1/32$ and (b) $h=1/500$. }\label{figuras_p1_p1}
\end{figure}
In Figure~\ref{figuras_p1_p1} we show the corresponding
approximation of the pressure at the first time step, for two
different values of $h$, that is, (a) $h=1/32$ and (b) $h=1/500$. Besides, we
have plotted the analytical solution of the problem (see~\cite{Aguilar2008}). We can observe that strong non-physical
oscillations appear for this type of finite element approximations,
when the space discretization parameter is not small enough. It is
clear that this is due to a lack of monotonicity of the scheme.
At a first glance, it appears that these oscillations might be related
to the locking effect and/or the fact that
the pair of finite element does not satisfy an inf-sup
condition. However, since our test is an one-dimensional problem,
elastic locking can not appear, and therefore, in general, this can
not be the only cause of this oscillatory behavior.

\subsection{Discretization with Taylor-Hood elements}
We consider the Taylor-Hood finite element method
proposed in~\cite{1973TaylorC_HoodP-aa} approximating the
displacement by continuous piecewise quadratic functions and the
pressure by continuous piecewise linear functions. It is well-known
that this pair of finite elements provides a stable discretization for
the Stokes equation and satisfies inf-sup condition.
Following similar computations as for the P1-P1 case, and
we obtain the following linear system of equations on each time step
\begin{equation} \label{linear_system_b}
\left[
\begin{array}{ccc}
A_b & 0 & G_b \\
0 & A_l & G_l \\
G_b^T & G_l^T & \tau A_p
\end{array}
\right] \left[
\begin{array}{c}
U_b^m \\ U_l^m \\ P^m
\end{array}
\right] = \left[
\begin{array}{ccc}
0 & 0 & 0 \\
0 & 0 & 0 \\
G_b^T & G_l^T & 0
\end{array}
\right]
\left[
\begin{array}{c}
U_b^{m-1} \\ U_l^{m-1} \\ P^{m-1}
\end{array}
\right] +
\left[
\begin{array}{c}
f_b^m \\ f_l^m \\ 0
\end{array}
\right],
\end{equation}
where $A_l, G_l$ correspond again to the linear basis functions
whereas $A_b, G_b$ are associated with the bubble basis functions.  In
this case, the pressure at time level $m$ satisfies the equation
\begin{equation} \label{Schur2}
(C_l + C_b + \tau A_p) P^{m} = (C_l +
  C_b) P^{m-1} - G_l^T A_l^{-1} (f_l^m -f_l^{m-1}) - G_b^T A_b^{-1}
  (f_b^m -f_b^{m-1}),
\end{equation}
where $C_l$ is as in \eqref{C_linear} and $C_b=-G_b^T A_b^{-1} G_b$ is given by
\[
(C_b P^m)_i = \frac{1}{12} \left(\displaystyle -\frac{h_{i-1}}{E_{i-1}} P_{i-1}^m + \left(\frac{h_{i-1}}{E_{i-1}} + \frac{h_{i}}{E_{i}} \right) P_{i}^m -
\frac{h_{i}}{E_{i}} P_{i+1}^m \right).
\]
Note that the off-diagonal entries of matrix $C_b$ are non-positive,
but again depending on the values of the
parameters, the whole matrix $C_l + C_b + \tau A_p$ can still have
positive off-diagonal terms. To avoid this, on each element the
restriction
\begin{equation}\label{Restr2}
\max_{0 \leq i \leq n-1} \displaystyle \frac{h_i^2}{6 K_i E_i} < \tau.
\end{equation}
must be fulfilled.

In summary, the use of quadratic finite elements for displacement
does contributes towards the reduction of the non-physical oscillations,
but is still not enough to eliminate them.

To illustrate this behavior, we consider again system
\eqref{linear_system_b} on an uniform grid of size $h$ and constant
coefficients $E$ and $K$. In this particular case, the restriction
\eqref{Restr2} is simplified to $h^2 < 6 E K \tau$, and when
$E K \tau = 10^{-6}$ it is deduced that $409$ nodes are needed to
ensure a non-oscillatory behavior. In Figure \eqref{figuras_p2_p1} we
show the corresponding approximation of the pressure at the first time
step, for two different values of $h$, that is, $h = 1/32$ and
$h = 1/409$.  Notice again that in the first case the pressure is not
monotone (oscillations show up), which shows that
the inf-sup condition is not enough for the monotonicity of the discretization.
\begin{figure}
\begin{tabular}{cc}
\includegraphics[scale = 0.4]{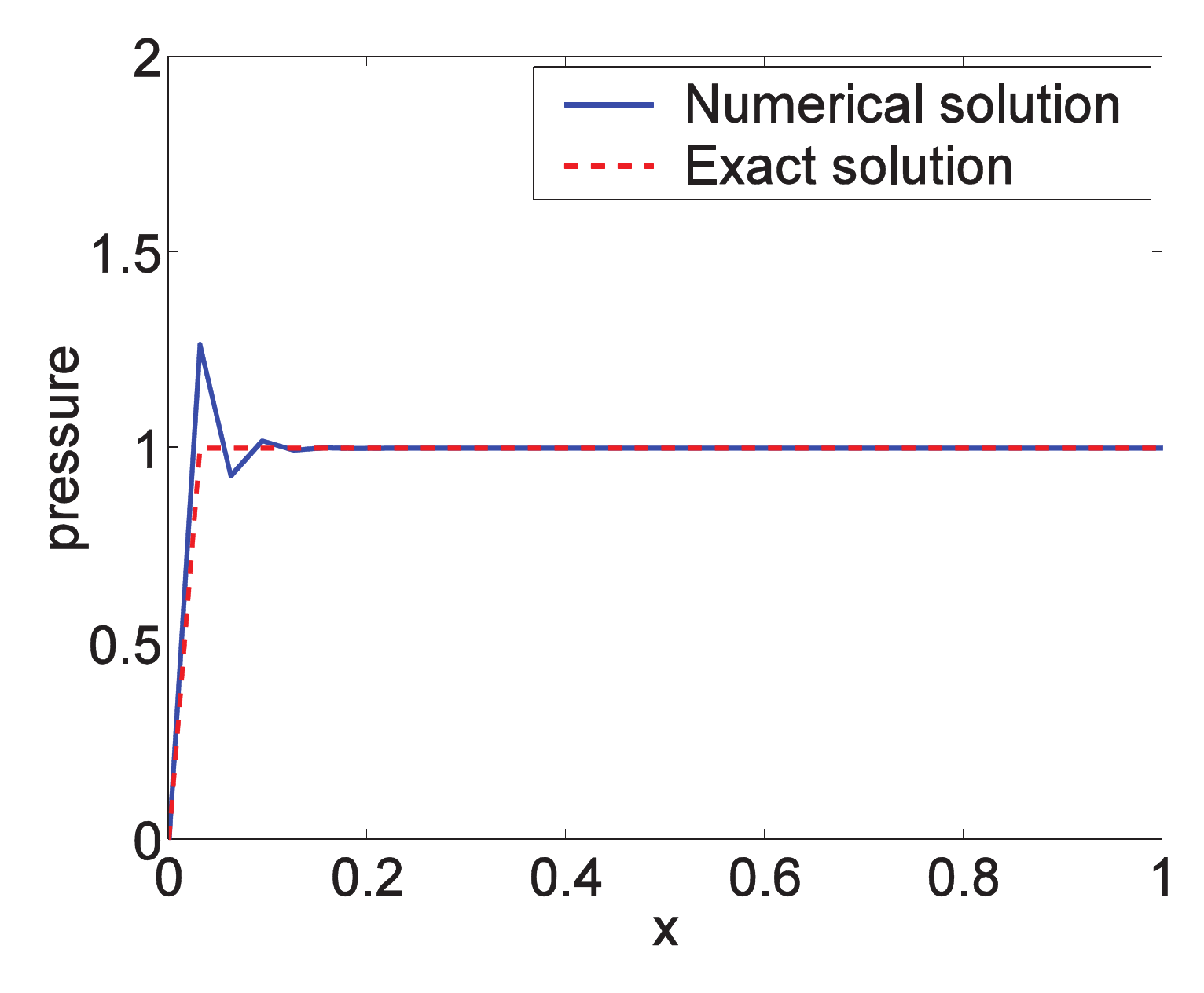}
&
\includegraphics[scale = 0.4]{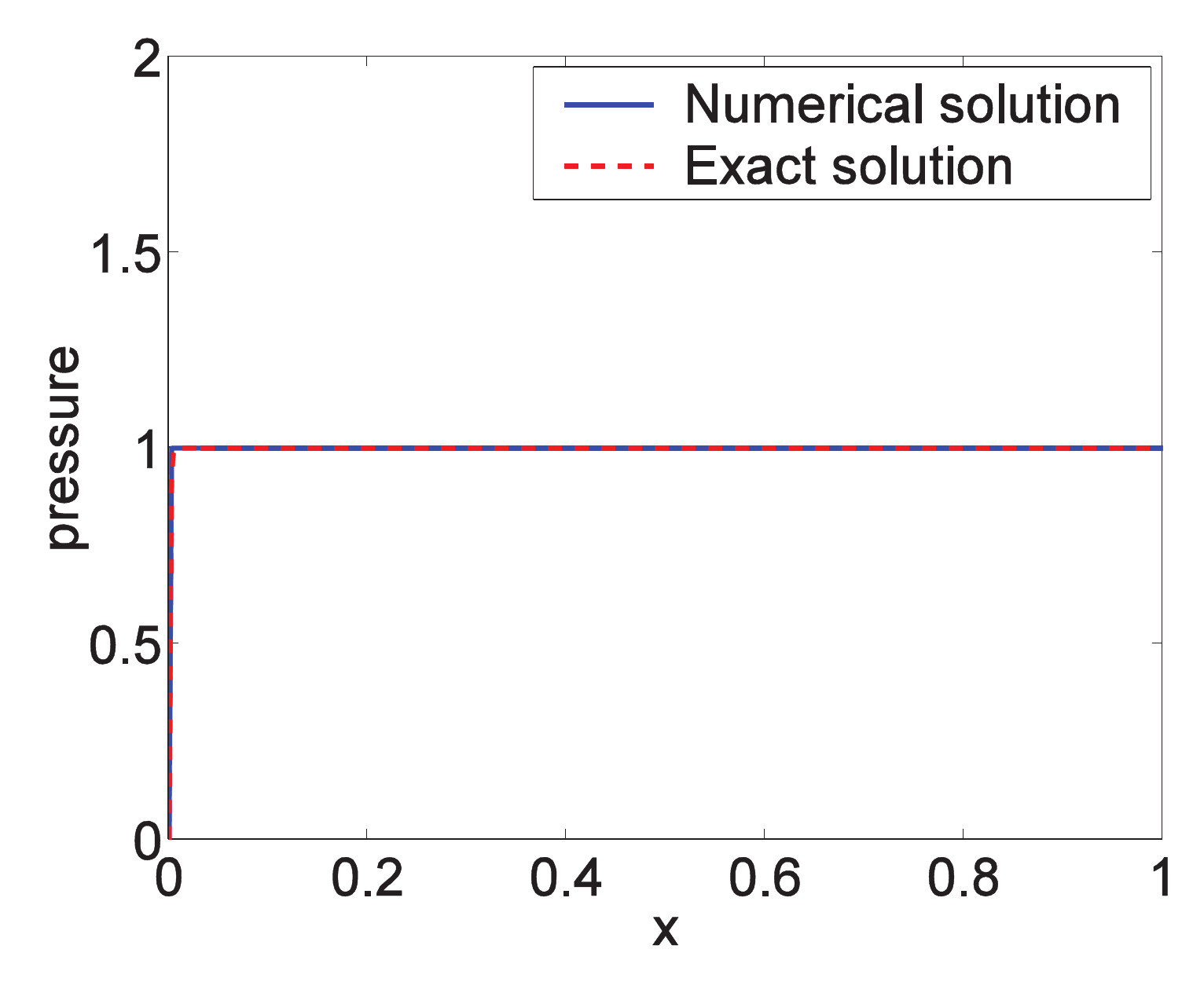}
\\
(a)
&
(b)
\end{tabular}
\caption{Numerical solution for the pressure field obtained with
  finite elements P2-P1 and corresponding exact solution for (a)
  $h=1/32$ and (b) $h=1/409$. }\label{figuras_p2_p1}
\end{figure}

\subsection{Monotone discretizations using perturbations}
To avoid the restrictions \eqref{Restr1} for P1-P1 and \eqref{Restr2}
for P2-P1 which result in the requirement for using very small mesh
size, we are going to introduce a perturbation
which will lead to monotone (and accurate) discretization
independently of the chosen parameters.

One way to achieve this is to add stabilization terms so that the
discretizations \eqref{Schur} and \eqref{Schur2} correspond to the
standard monotone linear finite element discretization of the
parabolic (heat) equation \eqref{eq:pressure}.
We define the following tridiagonal matrix
\begin{equation} \label{add_artificial}
(A_{\varepsilon} P^m)_i = \varepsilon \left(\displaystyle -\frac{h_{i-1}}{E_{i-1}} P_{i-1}^m + \left(\frac{h_{i-1}}{E_{i-1}} + \frac{h_{i}}{E_{i}} \right) P_{i}^m -
\frac{h_{i}}{E_{i}} P_{i+1}^m \right),
\end{equation}
where $\varepsilon=1/4$ for the linear finite element pair and
$\varepsilon=1/6$ for the Taylor--Hood method. Then, it is clear that
the perturbation of scheme \eqref{Schur}
\begin{equation} \label{pertur}
(C_l + A_{\varepsilon} + \tau A_p) P^{m} = (C_l + A_{\varepsilon}) P^{m-1}  - G_l^T A_l^{-1} (f_l^m -f_l^{m-1}),
\end{equation}
or the perturbation of \eqref{Schur2}
\begin{equation} \label{pertur2}
(C_l + C_b + A_{\varepsilon} + \tau A_p) P^{m} = (C_l + C_b + A_{\varepsilon}) P^{m-1}  - G_l^T A_l^{-1} (f_l^m -f_l^{m-1}) - G_b^T A_b^{-1} (f_b^m -f_b^{m-1}),
\end{equation}
gives the standard discretization of ~\eqref{eq:pressure} by linear
finite element method with mass-lumping. We also note that this
perturbation corresponds to adding the following term to the second
equation in~\eqref{example-int}
\begin{equation}\label{fvar}
\varepsilon \sum_{i=0}^{n-1} \frac{h_i^2}{E_i} \int_{T_i} \left(\frac{\nabla p_h^{m+1}-\nabla p_h^m}{\tau} \right) \cdot \nabla q_h \, {\rm d}x.
\end{equation}
Finally, in Figure~\ref{figuras_stabilization} we show the
approximation for the pressure obtained using the stabilized scheme
for both the linear finite element pair and the Taylor--Hood method
with $h = 1/32$ and we obtain monotone approximation for the
pressure.



\begin{figure}
\begin{tabular}{cc}
\includegraphics[scale = 0.4]{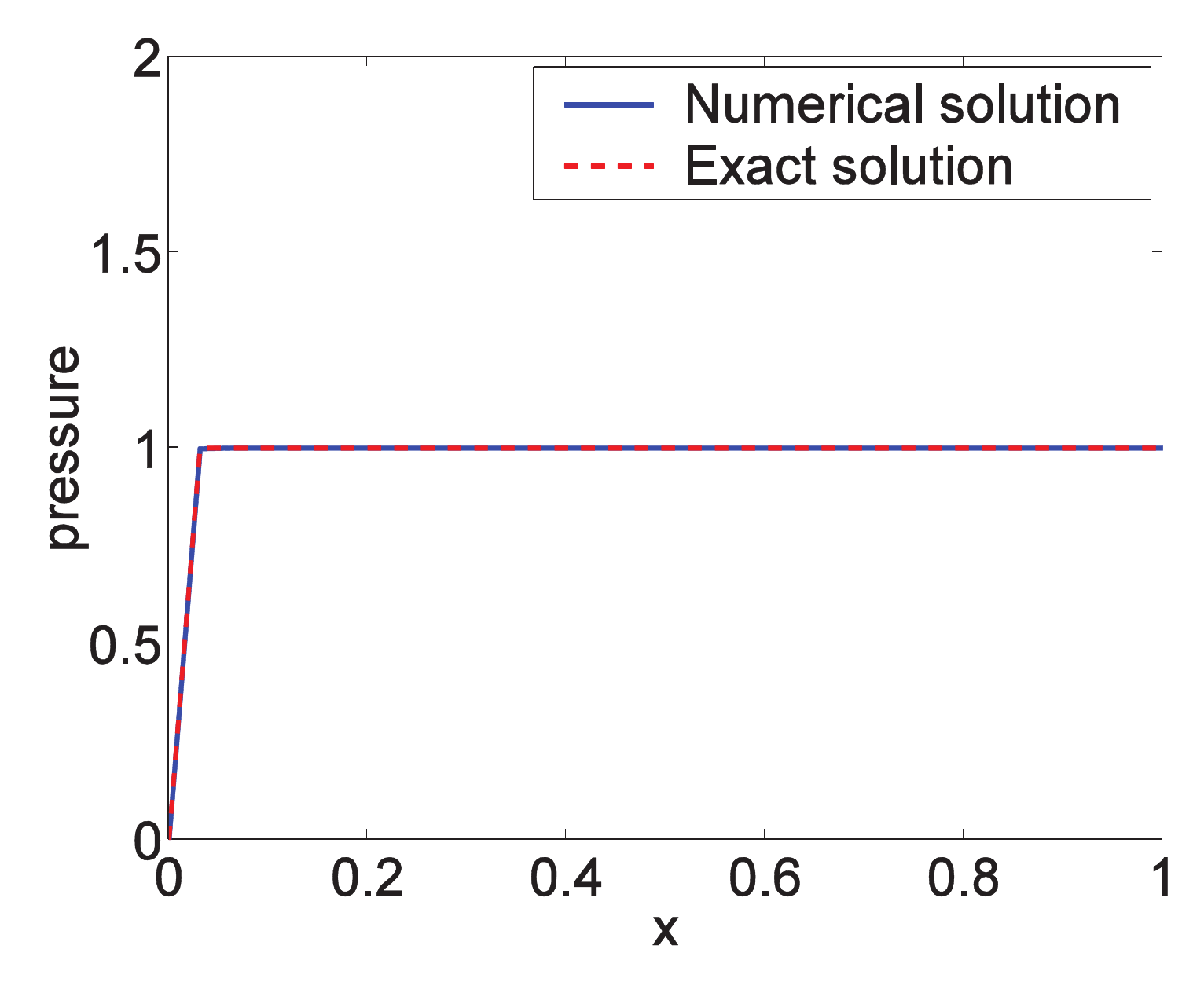}
&
\includegraphics[scale = 0.4]{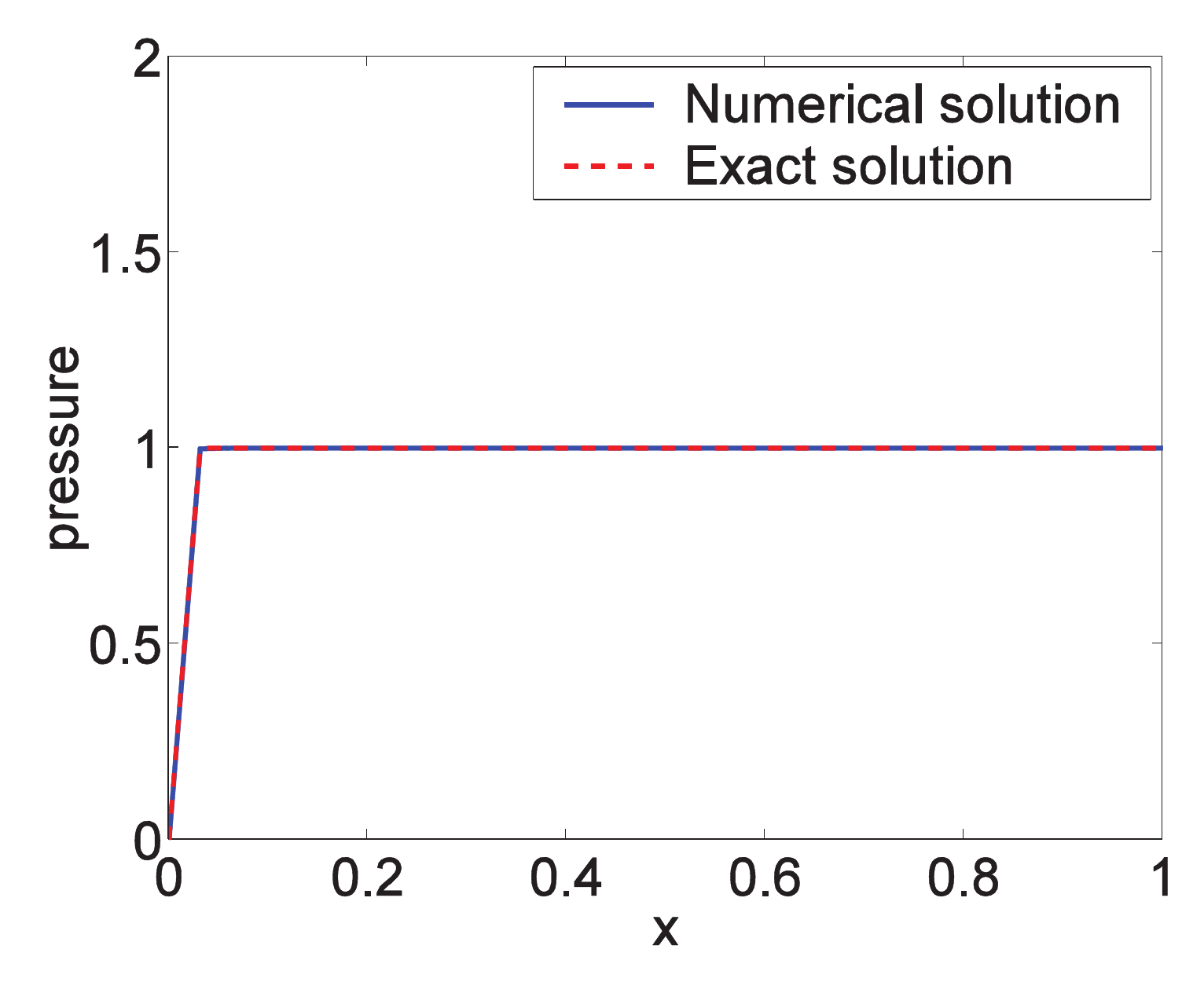}
\\
(a)
&
(b)
\end{tabular}
\caption{Numerical solution for the pressure field obtained with the stabilized finite elements (a) P1-P1 and (b) P2-P1  and corresponding exact solution. }\label{figuras_stabilization}
\end{figure}


\section{Stability of discretizations and perturbations of Biot's model}
\label{sec:Schur}
In this section we provide results on the stability of discretizations
of saddle point problems that can be viewed as perturbations of the
Stokes equations. By stability, here, we mean bounds on the inverse of the
discrete operator (for a fixed time step).  We prove inf-sup condition for different
discretizations for the poroelasticity problem, more precisely for
MINI element and stabilized P1-P1 schemes.  Such results are
well-known for Stokes equations (see,
e.g.~\cite{1984ArnoldD_BrezziF_FortinM-aa, 1991BrezziF_FortinM-aa, 2013BoffiD_BrezziF_FortinM-aa}).

We hope that the results given below in Section~\ref{section:abstract} will
be useful in other situations.
We note that the generality of the abstract results allows us to use
an unweighted $L^2$ norm for the pressure (not only an energy norm),
which gives new estimates in the analysis of the fully discretized time dependent
Biot's model.

\subsection{Stability of a class of saddle point problems with
  perturbation}\label{section:abstract}
In this section, we consider operators of the form
\begin{equation}\label{AC}
\mathcal{A}_C=\begin{pmatrix}
A & B'\\
B & -C
\end{pmatrix}: V\times Q \mapsto V'\times Q',
\end{equation}
where $V$ and $Q$ are Hilbert spaces and $V'$ and $Q'$ are their dual
spaces. Here, $\langle\cdot,\cdot\rangle$ is the standard duality
pairing and $B': Q\mapsto V'$ is the adjoint of $B$.
We make the following assumptions on $A$ and $C$.
\begin{enumerate}

\item[(A1)] The operator $A:V\mapsto V'$ is bounded, selfadjoint and  positive definite.
Thus, $A$ provides
  a scalar product $(\cdot,\cdot)_A=\langle A \cdot,\cdot\rangle$ and a norm on
  $V$ denoted by $\|\cdot\|_A$. The Hilbert Space $V$ is then equipped with this inner
 product and norm, and we have that
\begin{eqnarray*}
&&\|v\|_A^2:=\langle  Av,v\rangle,\quad
\|f\|_{V'}^2:=\langle  f,A^{-1}f \rangle,
\quad \mbox{for all}\quad v\in V, \quad f\in V'\\
&&   \|A\|_{V\mapsto V'}=\|A^{-1}\|_{V'\mapsto V}=1.
\end{eqnarray*}
\item[(A2)] The operator $B:V\mapsto Q'$ is bounded.

\item[(A3)] Similarly to $A$, the operator $C:Q\mapsto Q'$ is bounded, selfadjoint  and positive
  (semi)definite.  Thus on $Q$ we have a norm (or a semi-norm) denoted by
  $\|\cdot\|_C$
\end{enumerate}
We introduce a norm on $V\times Q$:
\begin{equation}\label{norm-definition}
\triplenorm{(u,p)}^2= \|u\|_A^2 + \|p\|_C^2 +  \|p\|^2.
\end{equation}
We note that if $C$ is only semidefinite, then $\|\cdot\|_C$ is only a
seminorm on $Q$.
Here $\|\cdot\|$ denotes the norm on $Q$ and $\triplenorm{\cdot}$ is
the norm on $V\times Q$ in which we will prove stability estimates for the operator
$\mathcal{A}_C$.

Clearly, $\mathcal{A}_C$ can be viewed as a perturbation of
$\mathcal{A}_0$, i.e. the operator with $C=0$.
For detailed discussion on perturbations of such saddle point problems,
we refer the reader to the recent monograph by Boffi, Brezzi and Fortin~\cite{2013BoffiD_BrezziF_FortinM-aa}.

We now state and prove a necessary and sufficient condition for
$\mathcal{A}_C$ to be isomorphism under the assumptions (A1)-(A3).
More general results also hold (with $A$ only invertible on a
subspace, etc), but to prove them would require more elaborate
arguments and such generality is beyond the scope of our
considerations here. We have the following theorem.

\begin{theorem}\label{theorem:stokes_poro}
  Assume that (A1)-(A3) hold. Then $\mathcal{A}_C$ defined
  in~\eqref{AC} is an isomorphism if and only if the operator $B$
  satisfies the following inf-sup condition: For any $q\in Q$ we have
\begin{equation} \label{ine:inf-sup}
\sup_{v\in V} \frac{\langle Bv, q\rangle}{\|v\|_A} \ge \gamma_B \|q\|
- \| q \|_{C}
\end{equation}
\end{theorem}
\begin{proof}
We first assume that~\eqref{ine:inf-sup} holds and we introduce the
bilinear form
\begin{equation*}
\langle \mathcal{A}_{C} (u,p); (v,q)\rangle = \langle Au, v \rangle + \langle Bv, p \rangle + \langle Bu, q \rangle - \langle Cp,q \rangle
\end{equation*}
It is easy to verify that the operator $\mathcal{A}_{C}$ is bounded in
$\triplenorm{\cdot}$ since both $A$ and $B$ are continuous.  From the
inf-sup condition~\eqref{ine:inf-sup}, for any $p$, there exist
$w \in V$, such that
$\langle Bw, p \rangle \geq ( \gamma_B \| p \| - \| p \|_C ) \|w\|_A$.
Since this inequality does not change when we multiply w by a positive
scalar, without loss of generality, we may assume that
$\|w\|_A=\|p\|$. We then have,
\begin{equation*}
\langle Bw, p \rangle \geq ( \gamma_B \| p \| - \| p \|_C ) \| p \|.
\end{equation*}
For a given pair $(u,p)\in V\times Q$ and with $w$ defined as above,
we choose $v = u + \theta w$, and, $q = -p$, with some $\theta>0$ to be
determined later.  Using the inf-sup condition, the fact that
$\|w\|_A=\|p\|$ and applying some obvious inequalities, such as,  $ab \geq
-\frac{1}{2\theta}a^2-\frac{\theta}{2}b^2$,  we have
\begin{eqnarray*}
\langle \mathcal{A}_{C} (u,p); (v,q)\rangle  & = &
\langle Au, u+ \theta w \rangle +
\langle B (u + \theta w), p \rangle - \langle Bu, p \rangle + \langle C p, p \rangle \\
 	& = & \|u\|_A^2 + \theta \langle Au, w \rangle + \theta\langle B w, p \rangle
 + \| p \|^2_{C} \\
 		& \geq& \frac{1}{2} \|u\|_A^2 -
 \frac{\theta^2}{2}\|p\|^2 + \theta\gamma_B\| p \|^2 - \theta\|p\|_C \| p \| + \| p \|_{C}^2\\
		& \geq& \frac{1}{2} \|u\|_A^2 +
\left(\theta\gamma_B-\frac{\theta^2}{2}\right)\|p\|^2 -
\theta\left(\frac{1}{2\theta}\|p\|_C^2 + \frac{\theta}{2} \| p \|^2\right) + \| p \|_{C}^2.
\end{eqnarray*}
Since the inequality above holds for any $\theta>0$, we choose
$\theta=\frac{\gamma_B}{2}$ to obtain that
\begin{eqnarray*}
\langle \mathcal{A}_{C} (u,p); (v,q)\rangle
& \ge  & \frac12\|u\|_A^2+\frac{\gamma_B^2}{4}\| p \|^2 +\frac12\|p\|_C^2
\geq \widetilde{\gamma}\triplenorm{(u,p)}^2
\end{eqnarray*}
where
$\widetilde{\gamma} = \frac14\min \{ 2, \gamma_B^2\} $.  On the other
hand, the triangle inequality implies that
\[
\triplenorm{(v,q)} = \triplenorm{(u+\theta w,p)}\leq
\widetilde{\gamma}_1\triplenorm{(u,p)},
\]
with $\widetilde{\gamma}_1$ depending only on $\gamma_B$. Hence,
\begin{equation*}
\sup_{v, q} \frac{\langle \mathcal{A}_{C}(u,p);(v,q)
  \rangle}{\triplenorm{(v,q)}} \geq
\gamma\triplenorm{(u,p)}, \quad \gamma = \frac{\widetilde{\gamma}}{\widetilde{\gamma}_1}
\end{equation*}
which shows that $\mathcal{A}_{C}$ is an isomorphism.

To prove the other direction, that the invertibility of
$\mathcal{A}_C$ implies condition~\eqref{ine:inf-sup}, for any $q\in Q$, we define
$v_q = -A^{-1}B' q\in V$. Since
$\mathcal{A}_{C}\begin{pmatrix}v_q\\q\end{pmatrix}
= \begin{pmatrix}0\\ Bv_q-Cq\end{pmatrix}$
the invertibility of $\mathcal{A}_{C}$ implies that
\[
\|q\|\le \triplenorm{(v_q,q)}\le
\|\mathcal{A}_{C}^{-1}\|
\;\|Bv_q-Cq\|_{Q'}\le
\|\mathcal{A}_{C}^{-1}\|\;(\| B v_q \|_{Q'} +\| C q\|_{Q'}).
\]
Since $C$ is symmetric and positive (semi)-definite, we have
$\pairq{C q}{s} \le\sqrt{\pairq{Cq}{q}}\sqrt{\pairq{Cs}{s}}$. Hence,
\[
\|C q\|_{Q'} = \sup_{s\in Q}\frac{\pairq{C q}{s}}{\|s\|} \le
\sqrt{\|C\|\pairq{C q}{q}}.
\]
To estimate $\|Bv_q\|_{Q'}$ we observe that
$\| B v_q\|_{Q'} = \sup_{s\in Q}\frac{\pairq{B v_q}{s}}{\|s\|}$ and
we also have for all $s\in Q$,
\begin{eqnarray*}
\frac{|\pairq{Bv_q}{s}|}{\|s\|} & = &
\frac{|\pairv{B's}{A^{-1}B'q}|}{\|s\|} \le \|B'\|
\frac{|\pairv{B's}{A^{-1}B'q}|}{\|B's\|_{V'}} \\
& \le &
\|B'\|\sup_{f\in V'}\frac{\pairv{f}{A^{-1}B'q}}{\|f\|_{V'}}=
\|B'\|\sup_{w\in V}\frac{\pairv{Aw}{A^{-1}B'q}}{\|Aw\|_{V'}}\\
& \le &
\|B'\|\|A^{-1}\|\sup_{w\in V}\frac{\pairq{Bw}{q}}{\|w\|_{A}}
=\|B'\|\sup_{w\in V}\frac{\pairq{Bw}{q}}{\|w\|_{A}}.
\end{eqnarray*}
The inf-sup condition~\eqref{ine:inf-sup} easily follows by
combining the last two estimates.
\end{proof}

We have the following immediate corollaries.
\begin{corollary}\label{corollary1}
  Suppose that (A1)-(A3) hold. If $\mathcal{A}_0$ is an isomorphism,
  then $\mathcal{A}_C$ is an isomorphism for all continuous and positive
  (semi-)definite $C$.
\end{corollary}
\begin{proof} From the fact that $\mathcal{A}_0$ is isomorphism it
  follows that~\eqref{ine:inf-sup} holds with $C=0$, and hence, also
  with any symmetric positive, (semi-)definite and bounded $C$. This in
  turn (by Theorem~\ref{theorem:stokes_poro}) implies that
  $\mathcal{A}_C$ is an isomorphism.
\end{proof}

The next corollary allows us to add consistent perturbations to
already stable discretizations in order to improve the monotonicity properties
of the underlying discretizations.
\begin{corollary}\label{corollary2}
  Suppose that $\mathcal{A}_C$ is an isomorphism, that (A1)-(A3) hold,
  and that $D$ is spectrally equivalent to $C$, namely
  $\alpha_0\|q\|_C\le \|q\|_D \le \alpha_1\|q\|_C$ for some
  positive constants $\alpha_0$ and $\alpha_1$.  Then
  $\mathcal{A}_{D}$ is an isomorphism.
\end{corollary}
\begin{proof}
For all $q\in Q$, we have
\[
\|q\|_D + \sup_{v\in V}\frac{\langle Bv ,  q\rangle}{\|v\|_A}
\geq \min\{1,\alpha_0\}\left(\|q\|_C + \sup_{v\in V}\frac{\langle Bv,
    q\rangle}{\|v\|_A}\right)\ge \min\{1,\alpha_0\}\gamma_B\|q\|,
\]
which shows~\eqref{ine:inf-sup} for $\mathcal{A}_D$. Applying
Theorem~\ref{theorem:stokes_poro} gives the desired result.
\end{proof}

\subsection{Application to discretizations of Biot's model}
After a time discretization (backward Euler scheme in time) of the Biot's model, the following system of
differential equations is solved on every time step on a domain
$\Omega\subset \mathbb{R}^d$:
\begin{eqnarray}
&& -\ddiv\bmz{\sigma} + \nabla p = \bmz{f},\qquad
\bmz{\sigma} = 2\mu \varepsilon(\bmz{u})  + \lambda\ddiv(\bmz{u}) I \label{eq:u0}\\
&& -\ddiv \bmz{u} + \tau \ddiv K \nabla p = g.\label{eq:p0}
\end{eqnarray}
A typical set of boundary conditions is
\begin{eqnarray*}
&&\bmz{u} = \bmz{0}, \quad \mbox{and}\quad (K\nabla p\cdot\bmz{n}) =0,
\quad\mbox{on}\quad \Gamma_c,\\
&&\bmz{\sigma}\cdot\bmz{n} = \bmz{\beta}, \quad \mbox{and}\quad p=0\quad\mbox{on}\quad \Gamma_t.
\end{eqnarray*}

To introduce the spatial discretization of the Biot's model, we consider
finite dimensional spaces $V_h\subset [H_{\Gamma_c}^1(\Omega)]^d$ and
$Q_h\subset H_{\Gamma_t}^1(\Omega)$ where $H_{\Gamma_c}^1(\Omega)$ and
$H_{\Gamma_t}^1(\Omega)$ are the standard Sobolev spaces with
functions whose traces vanish on $\Gamma_c$ and $\Gamma_t$
respectively.

We have the following discrete formulation (on each time step)
corresponding to~\eqref{eq:u0}--\eqref{eq:p0}. Find
$(\bmz u,p)\in \bmz V_h\times Q_h$ such that
\begin{eqnarray}
&& a(\bmz{u},\bmz{v}) - (\ddiv \bmz{v},p) = (f,\bmz{v}),
\quad\mbox{for all}\quad
\bmz{v}\in \bmz V_h, \label{eq:u}\\
&& -(\ddiv \bmz{u},q) -
\tau a_p( p, q) = (g,q),
\quad\mbox{for all}\quad
 q\in Q_h. \label{eq:p}
\end{eqnarray}
The bilinear form $a(\cdot,\cdot)$ is as follows:
\begin{eqnarray*}
&&
a(\bmz{u},\bmz{v}) =
2\mu\int_{\Omega}\varepsilon(\bmz{u}):\varepsilon(\bmz{v}) +
\lambda\int_{\Omega} \ddiv\bmz{u}\ddiv\bmz{v},\quad
a_p(p,q) = \int_{\Omega} K\nabla p\cdot \nabla q.
\end{eqnarray*}
The corresponding operators $A: V_h\mapsto V_h'$,
$B: Q_h\mapsto V_h'$, and the norm on $Q$, $\|\cdot\|$, are defined as
follows:
\begin{eqnarray*}
&& \langle A u,v \rangle := a(u,v), \quad
\langle B u,q \rangle := -(\ddiv u ,q),\quad
\langle A_p p,q \rangle := a_p(p ,q),\\
&&\|q\|^2  := \tau \langle A_pq,q\rangle  + \|q\|_{L^2(\Omega)}^2.
\end{eqnarray*}
Since $C$ may take different form for different discretizations, we do not
specify its definition here.

\subsubsection{Discretization with MINI element}
We consider a discretization with MINI element, introduced in~\cite{1984ArnoldD_BrezziF_FortinM-aa} where the finite element
spaces that we use are as follows:
\[
\bmz V_h\times Q_h, \quad \mbox{where}\quad  \bmz V_h = \bmz{V}_l\oplus \bmz{V}_b,
\]
where $\bmz V_l$ is the space of piece-wise (with respect to a
triangulation $\mathcal{T}_h$) linear continuous vector
valued functions on $\Omega$ and $\bmz V_b$ is the space of bubble
functions, defined as
\[
\bmz V_b = \operatorname{span}\{\varphi_{b,T}\bmz e_1,
\ldots, \varphi_{b,T}\bmz e_d\}_{T\in \mathcal{T}_h}, \quad
\varphi_{b,T} = \alpha_T\lambda_{1,T}\ldots\lambda_{d+1,T},
\]
where $\lambda_{m,T}$ are the barycentric coordinates on $T$, $\bmz
e_j$ are the canonical Euclidean basis vectors in $\mathbb{R}^d$ and
$\alpha_T$ is a normalizing constant for $\varphi_{b,T}$.  The
function $\varphi_{b,T}$ is scalar valued and is called a bubble
function. The space $Q_h$ consists of piece-wise linear continuous
scalar valued functions.

Note that if we write $\bmz v = \bmz v_l + \bmz v_b$ we have that
\[
a(\bmz{u},\bmz{v}) = a(\bmz{u}_l,\bmz{v}_l) +  a(\bmz{u}_b,\bmz{v}_b).
\]
This is so because $\bmz v_b$ is zero on $\partial T$ for
$T\in T_h$ and integration by parts shows that
$a(\bmz v_l,\bmz v_b) = 0$. We then have the following block
form of the discrete problem~\eqref{eq:u}-\eqref{eq:p}:
\begin{equation}\label{eq:op}
\mathcal{A}\;
\begin{pmatrix}
\bmz{u}_{b}\\
\bmz{u}_{l}\\
 p
\end{pmatrix} =
\begin{pmatrix}
\bmz f_b\\
\bmz f_l\\
g
\end{pmatrix}, \quad\mbox{where}\quad
\mathcal{A}=\begin{pmatrix}
A_b & 0 & G_b\\
0 &  A_l & G_l     \\
G_b^T & G_l^T   & -\tau A_p
\end{pmatrix}\;
\end{equation}
The operators $A_b$, $A_l$, $G_b$, $G_l$ and $A_p$ correspond to the
following bilinear forms:
\begin{eqnarray*}
&&
a(\bmz u_b,\bmz v_b)\rightarrow A_b, \quad a(\bmz u_l,\bmz v_l)\rightarrow
A_l, \quad (K\nabla p,\nabla q) \rightarrow A_p\\
&&-(\ddiv \bmz v_b, p) = (\bmz v_b,\nabla p)  \rightarrow  G_b, \quad
-(\ddiv \bmz v_l,  p)  \rightarrow  G_l,\\
&&\bmz u_b, \;
\bmz v_b
\in \bmz{V}_b,\quad
\bmz u_l, \;
\bmz v_l
\in \bmz{V}_l,\quad
p,\;q\in Q_h.
\end{eqnarray*}

It is well known that inf-sup condition holds for the MINI element for
the Stokes problem, and therefore, by Corollary~\ref{corollary1}, we
obtain the following inf-sup condition for MINI element
discretization of poro-elasticity operator: There exists $\gamma_0$
independent of $h$, $\tau$ and $K$, such that for any
$(\bmz v, q)\in \bmz V_h\times Q_h$ we have
\begin{equation}\label{inf-sup-big}
  \sup_{(\bmz w,s)\in \bmz V_{h}\times Q_h} \frac{(\mathcal{A}(\bmz
    v,q),(\bmz w,s))}{\triplenorm{(\bmz w,s)}} \ge
\gamma_0\triplenorm{(\bmz v,q)}.
 \end{equation}
 As it is well-known (see~\cite{1974BrezziF-aa}),
 equation~\eqref{inf-sup-big} is equivalent to the estimate
\begin{equation}\label{estimate_sol}
\triplenorm{(\bmz u,p)} \leq \gamma_0^{-1} \|(\bmz f,g)\|.
\end{equation}

\subsection{Stabilization via elimination of bubbles}
All P1-P1 stabilized discretizations which we consider here, are
derived from the MINI element by eliminating locally the bubble
functions. For details on such stabilizations we refer to the
classical paper by Brezzi and
Pitk\"{a}ranta~\cite{1984BrezziF_PitkarantaJ-aa} (see
also~\cite{1993BaiocchiC_BrezziF_FrancaL-aa}).

We now consider the following operator on $\bmz V_l \times Q_h$:
\[
\mathcal{A}_l=\begin{pmatrix}
A_l & \;\;G_l     \\
G_l^T   &\;\; -(\tau A_p +S_b)
\end{pmatrix},
\quad \mbox{where}\quad
S_b = G_b^TA_{b}^{-1} G_b,
\]
which is obtained after eliminating the equation corresponding to
bubble functions from~\eqref{eq:op}.  This is also an operator of the
form given in~\eqref{AC} with $C=\tau A_p + S_b$.
We have the following theorem:
\begin{theorem}
Suppose that the triple $(\bmz u_b, \bmz u_l, p)$ solves
\begin{equation}\label{eq:big}
\mathcal{A}\;
\begin{pmatrix}
\bmz{u}_{b}\\
\bmz{u}_{l}\\
 p
\end{pmatrix} =
\begin{pmatrix}
\bmz 0\\
\bmz f_l\\
g
\end{pmatrix}.
\end{equation}
Then the pair $(\bmz u_l, p)$ solves
\begin{equation}\label{eq:small}
\mathcal{A}_l\;
\begin{pmatrix}
\bmz{u}_{l}\\
 p
\end{pmatrix} =
\begin{pmatrix}
\bmz f_l\\
g
\end{pmatrix}.
\end{equation}
Moreover, a uniform inf-sup condition such as~\eqref{inf-sup-big}
holds: For any $(\bmz v_l,q) \in \bmz V_l\times~Q_h$,
\begin{equation}\label{inf-sup-small}
\sup_{(\bmz w_l,s)\in \bmz V_{l}\times Q_h} \frac{(\mathcal{A}_l
(\bmz v_l,q),(\bmz w_l,s))}{\triplenorm{(\bmz w_l,s)}} \ge \gamma_1\triplenorm{(\bmz v_l,q)}.
\end{equation}
\end{theorem}
\begin{proof}
Since $(\bmz u_b,\bmz u_l, p)$ solves the system~\eqref{eq:big} we have
that
\begin{eqnarray*}
&& \bmz u_b = -A_b^{-1}G_b p\\
&& A_l \bmz u_l + G_l p = \bmz f_l. \\
&& G_l^T\bmz u_l + G_b^T \bmz u_b - \tau A_p p = g \quad \Longrightarrow
G_l^T\bmz u_l -(G_b^T A_b^{-1} G_b +  \tau A_p )p = g.
\end{eqnarray*}
From this we conclude that $(\bmz u_l, p)$ solves~\eqref{eq:small}.
Now, since $(\bmz u_b,\bmz u_l, p)$ solves~\eqref{eq:big}, from~\eqref{estimate_sol},
\[
\triplenorm{(\bmz u_b, \bmz u_l,p)} \le \gamma_0^{-1}\|(\bmz 0, \bmz f_l,g)\|,
\]
and therefore we have
\[
\triplenorm{(\bmz u_l,p)} \le \triplenorm{(\bmz u_b, \bmz u_l,p)} \le
\gamma_0^{-1}\|(\bmz 0, \bmz f_l,g)\|=\gamma_0^{-1}\|(\bmz f_l,g)\|.
\]
This estimate shows that $\mathcal{A}_l$ is a bounded
isomorphism, which is equivalent to the inf-sup condition~\eqref{inf-sup-small}.
This completes the proof.
\end{proof}

Applying Corollary~\ref{corollary2} to $\mathcal{A}_l$ then shows that
any operator $C: Q_h\mapsto Q_h'$, spectrally equivalent to $\tau A_p
+ S_b$ will
result in a stable discretization of the Biot's model. As we show in
the next section (Theorem~\ref{thm:schur}), the perturbations
spectrally equivalent to $S_b$ are of the form
\[
\langle C p,q\rangle=\sum_{T\in \mathcal{T}_h} C_Th_T^{2}\int_T(\nabla
p\cdot \nabla q),
\]
where $C_T$, $T\in \mathcal{T}_h$ are constants independent of the
mesh size $h$ or $\tau$.

\subsection{Perturbations, spectrally equivalent to the Schur
  complement}
In this section we compute the Schur complement (the perturbation or the
stabilization) given by $S_b=G_b^TA_b^{-1}G_b$.  We denote
$\bmz V_{b,T}=\operatorname{span}{\varphi_{b,T}}$ and we have that
$\bmz V_b=\oplus_{T\in \mathcal{T}_h} \bmz V_{b,T}$. Let $n_V$ be the
number of vertices in the triangulation, $n_T$ be the number of
elements, and $n_b=d\,n_T$. Note that $n_b$ equals the dimension of
$\bmz V_b$.  With every element $T\in \mathcal{T}_h$ we associate the
incidence matrices $I_{T}\in \mathbb{R}^{n_V\times(d+1)}$ and
$J_T\in \mathbb{R}^{n_b\times d}$ mapping the local degrees of freedom
on $T$  to the degrees of freedom corresponding to $\bmz Q$
and $\bmz V_b$.

Let us now give a more precise definition of the incidence matrices
$I_T$ and $J_T$ for an element $T\in \mathcal{T}_h$, with vertices
$(j_1,\ldots, j_{d+1})$, $j_k\in \{1,\ldots, n_V\}$, and
$j_\ell\neq j_m$, for $j\neq m$.  Let
$\{\bmz{\delta}_1,\ldots,\bmz{\delta}_{d+1}\}$,
$\{\bmz e_1,\ldots,\bmz e_{n_V}\}$,
$\{\bmz f_1,\ldots,\bmz f_{n_b}\}$ and $\{\bmz \eta_1,\ldots,\bmz
\eta_d\}$  be the canonical Euclidean bases in
$\mathbb{R}^{d+1}$, $\mathbb{R}^{n_V}$, $\mathbb{R}^{n_b}$ and
$\mathbb{R}^d$, respectively.
We also denote by $(k_1,\ldots,k_d)$
the degrees of freedom corresponding to the bubble functions
associated with $T\in \mathcal{T}_h$.
We   then define
\begin{equation}
\mathbb{R}^{n_V\times (d+1)}\ni I_T = \sum_{m=1}^{d+1}\bmz e_{j_m}\bmz \delta_m^T, \quad
\mathbb{R}^{n_b\times d}\ni J_T = \sum_{m=1}^{d}\bmz f_{k_m}\bmz \eta_m^T.
\end{equation}
Since the sets of degrees of freedom corresponding to the bubble
functions in different elements do not intersect, we have
$J_{T}^T J_T = I_{d\times d}$, and, $J_{T^\prime}^T J_T = 0$ when
$T^\prime\neq T$. Here $I_{d\times d}\in \mathbb{R}^{d\times d}$ is the
identity matrix.  Using these definitions, we easily find
that
\begin{eqnarray*}
&& A_b = \sum_{T\in \mathcal{T}_h} J_T A_{b,T} J_T^T, \quad
A_b^{-1} = \sum_{T\in \mathcal{T}_h} J_T A^{-1}_{b,T} J_T^T,\\
&&G_b = \sum_{T\in \mathcal{T}_h} J_T G_{b,T} I_T^T.
\end{eqnarray*}
These identities then give,
\begin{equation}\label{eq:local-global}
S_b = G_b^T A_b^{-1} G_b,
\quad\mbox{and hence}\quad
S_b = \sum_{T\in \mathcal{T}_h} I_T G_{b,T}^T A_{b,T}^{-1} G_{b,T} I_T^T.
\end{equation}
We next state a spectral equivalence result which shows that $S_b$
introduces a stabilization term of certain order in $h$ for P1-P1
discretization.
Such stabilization techniques have
been discussed by~Verf\"{u}rth in~\cite{1984VerfurthR-aa} (see
also \S~8.5.2 and \S~8.13.2
in~\cite{2013BoffiD_BrezziF_FortinM-aa}).
\begin{theorem} \label{thm:schur} Let $L$ be the stiffness
  matrix corresponding to the Laplace operator discretized with
  piece-wise linear continuous finite elements. Then the following
  spectral equivalence result holds
\begin{equation}
S_{b}  \eqsim  h^2 L,
\end{equation}
where the constants hidden in ``$\eqsim$" are independent  of the mesh size.
\end{theorem}
\begin{proof}The spectral equivalence is a direct consequence from
Lemma~\ref{lemma:schur-local} and the relations given in~\eqref{eq:local-global}.
\end{proof}

\begin{remark}
  The spectral equivalence in Theorem~\ref{thm:schur} and the analysis
  that follows justifies the addition of stabilization terms to both
  the MINI element and the stabilized P1-P1 discretizations. The
  results in~\ref{sect:local-Schur} also hold for one, two and three
  spatial dimensions and also give the exact perturbation
  (stabilization) to P1-P1 elements that provides inf-sup condition
  with the same constant as the MINI element.

  Related results (in 2D) are found in a paper on Stokes equations by
  Bank and Welfert~\cite{1990BankR_WelfertB-aa} where it was shown
  that in 2D the elimination of the bubbles in the MINI element gives
  the Petrov-Galerkin discretization by Hughes, Franka and
  Balestra~\cite{1986HughesT_FrancaL_BalestraM-aa} and Brezzi and
  Douglas~\cite{1988BrezziF_DouglasJ-aa}. Here we not only compute the
  exact Schur complement in any spatial dimension, but we also show
  that the perturbation is spectrally equivalent to a scaling of the
  discretization of the Laplacian with piece-wise linear finite
  elements. The details are in the appendix.

  Such results, however, do not say anything about the monotonicity of
  the corresponding discretization (except in 1D, where a further
  stabilization can be introduced in order to obtain a monotone
  discrete scheme).  In fact, for the one dimensional case considered
  in detail in Section~\ref{sec:oscillations} the minimum amount of
  stabilization that provides monotone discretization can be
  calculated precisely. In general, even for two and three spatial
  dimensions, adding a stabilization term of the form $c h^2 L$ in
  case when $L$ is a Stieltjes matrix improves the monotonicity
  properties of the resulting discrete problem. This is natural to
  expect because a Stieltjes matrix is monotone. Indeed, the numerical
  results that we present later also show that adding such
  stabilizations leads to monotone schemes. However, no theoretical
  results on the monotonicity of the discrete operators for two and
  three dimensional problems are available in the literature and seem
  to be very hard to establish.
\end{remark}

\section{Error estimates for the fully discrete problem}\label{sec:fully-discrete}

In this section, we consider the error analysis of the finite element
discretization of the Biot's model.  To simplify the notation and
without loss of generality in this section we assume that
the boundary conditions for both the displacement $u$ and the pressure
$p$ are homogeneous Dirichlet boundary conditions.  Then, the weak
form of the Biot's model is as follows: Find
$\bmz u(t) \in \left[ H_0^1(\Omega) \right]^d$ and
$p(t) \in H_0^1(\Omega)$, such that
\begin{eqnarray}
&&a(\bmz u, \bmz v) - (\ddiv \bmz v, p) = (f, \bmz v), \quad \forall  \bmz v \in \bmz [H_0^1(\Omega)]^d ,  \label{eqn:biot-weak-1}  \\
&&- (\ddiv \partial_t \bmz{u}, q) - a_p(p,q) = 0, \quad \forall q \in H_0^1(\Omega) , \label{eqn:biot-weak-2}
\end{eqnarray}
with the initial data $\bmz u(0)$ and $p(0)$ given by the solution of
the following Stokes problem: Find $\bmz u(0) \in \left[ H_0^1(\Omega) \right]^d$
and $p(0) \in L^2(\Omega)$, such that,
\begin{eqnarray}
                                         && a(\bmz u(0), \bmz v) - (\ddiv \bmz v, p(0)) = (f(0), \bmz v), \quad \forall \bmz v \in \bmz [H_0^1(\Omega)]^d ,  \label{eqn:biot-weak-init-1}  \\
                                         && -(\ddiv \bmz u(0), q) = 0,
                                            \quad \forall q \in
                                            L^2(\Omega),
                                            \label{eqn:biot-weak-init-2}
                                       \end{eqnarray}

We consider the fully discretized scheme at time $t_n$, $n=1,2,\ldots$, as the following: Find $\bmz u_h^n=u_h(t_n) \in \bmz V_h \subset \left[ H^1(\Omega) \right]^d$ and $p_h^n=p_h(t_n) \in Q_h \subset H^1(\Omega)$, such that,
\begin{eqnarray}
&& a(\bmz u_h^n, \bmz v_h) - (\ddiv \bmz v_h, p_h^n) = (f(t_n), \bmz v_h), \quad \forall \bmz v_h \in \bmz V_h, \label{eqn:biot-mini-1}  \\
&& -(\ddiv \bar{\partial}_t \bmz u_h^n, q_h) - a_p(p_h^n, q_h) - \varepsilon h^2 (\nabla \bar{\partial}_t p^n_h, \nabla q_h) = 0, \quad \forall q_h \in Q_h, \label{eqn:biot-mini-2}
\end{eqnarray}

where
$\bar{\partial}_t \bmz u_h^n := (\bmz u_h^n - \bmz u_h^{n-1})/\tau$
and $\bar{\partial}_t p_h^n := (p_h^n - p_h^{n-1})/\tau$. Here we try
to analyze MINI element and stabilized P1-P1 element in a unified way,
therefore, the finite element spaces $\bmz V_h$ and $Q_h$ denote both
Stokes pairs. We also define the following norm on the finite element
spaces:
\begin{equation} \| (\bmz{u}, p ) \|_{\tau, h} := \left( \|
    \bmz u \|_a^2 + \tau \| p \|_{a_p}^2 + \varepsilon h^2 \| \nabla p
    \|^2 \right)^{1/2}.
\end{equation}
We further denote, by $\|\cdot\|_k$ and $|\cdot|_k$ the norms and
seminorms in the Sobolev space $H^{k}(\Omega)$, and without loss of
generality, by $\|\cdot\|$ the $L^2(\Omega)$ norm, i.e.
$\|\cdot\|=\|\cdot\|_0$. Below we also denote by $c$ a generic constant
independent of time step, mesh size and other important parameters.

For the initial data $\bmz u_h^0$ and $p_h^0$, we will consider two cases.  First case is that they are given by the following stabilized Stokes equation:
\begin{eqnarray}
&& a(\bmz u_h^0, \bmz v_h) - (\ddiv \bmz v_h, p_h^0) = (f(0), \bmz v_h) \quad \forall \bmz v_h \in \bmz V_h,
\label{eqn:biot-mini-init-1} \\
&& -(\ddiv \bmz u_h^0, q_h) - \varepsilon h^2 (\nabla p_h^0, \nabla q_h)= 0 \quad \forall q_h \in Q_h. \label{eqn:biot-mini-init-2}
\end{eqnarray}
Second case is that they do not satisfy \eqref{eqn:biot-mini-init-1} and \eqref{eqn:biot-mini-init-2} but are defined as following,
\begin{equation} \label{eqn:biot-mini-ini-3}
\ddiv \bmz u_h^0 = 0 \ \text{and} \ p_h^0 = 0.
\end{equation}

To derive error analysis of the fully discretized scheme
\eqref{eqn:biot-mini-1}-\eqref{eqn:biot-mini-2}, we need to define the
following elliptic projections $ \bar{\bmz u}_h$ and $\bar{p}_h$ for
$t>0$ as usual,
\begin{eqnarray}
&& a( \bar{\bmz u}_h, \bmz v_h ) - (\ddiv \bmz v_h,  \bar{p}_h) = a(\bmz u, \bmz v_h) - (\ddiv \bmz v_h, p), \quad \forall \bmz v_h \in \bmz V_h \\
&& a_p(\bar{p}_h, q_h) = a_p(p, q_h), \quad \forall
                    q_h \in Q_h
\end{eqnarray}

To estimate the error, following Thom{\'e}e,~\cite{2006ThomeeV-aa} we
split the discretization error as follows.
\begin{eqnarray}
&&\bmz u(t) - \bmz u_h(t) = (\bmz u(t) - \bar{\bmz u}_h(t) ) - (\bmz u_h(t) - \bar{\bmz u}_h(t)) =: \rho_{\bmz u} - e_{\bmz u}, \label{eqn:err_u_decomp} \\
&& p(t) - p_h(t) = (p(t) - \bar{p}_h(t) ) - (p_h(t) - \bar{p}_h(t)) =: \rho_{p} - e_p. \label{eqn:err_p_decomp}
\end{eqnarray}
For $t=t_n$ we use the short hand notation
$\rho_{\bmz u}^n=\rho_{\bmz  u}(t_n)$, and similarly $e_{\bmz u}^n$,
$\rho_{p}^n$, $e_p^n$ denote the values of $e_u$, $\rho_p$ and $e_p$
at time $t=t_n$, respectively.

For the error of the elliptic projections, because we use MINI element
or P1-P1 element, we have, for all $t$,
\begin{eqnarray}
&& \| \rho_{\bmz u} \|_a \leq c h ( | \bmz u |_2 + | p |_1 ), \label{rhou}\\
&& \| \rho_{p} \|_{1} \leq c h | p |_2, \quad\| \rho_{p} \|_{a_p} \leq
   c h | p |_2\label{rhop-1}\\
&& \| \rho_p \| \leq c h^2 | p |_2 \label{rhop-2}.
\end{eqnarray}
We refer to~\cite{MuradLoula94} for details. Since $\partial_t
\overline{p}=\overline{\partial_t p}$, we have the estimates above
also for  $\partial_t \rho_u$ and $\partial_t \rho_p$, where on the
right side of the inequalities we have norms of $\partial_t u$ and $\partial_t p$
instead of norms of $u$ and $p$ respectively.

The following lemmas estimate the error between the elliptic
projection $ \{ \bar{\bmz u}_h(t_n), \bar{p}_h(t_n) \}$ and the
numerical solutions $\{ \bmz u_h^n, p_h^n \}$.
\begin{lemma}
  \label{lem:mini-e_u-e_p} Let
  $w_{\bmz u}^j := \partial_t \bmz u(t_j) - \frac{\bar{\bmz
      u}_h(t_{j}) - \bar{\bmz u}_h(t_{j-1})}{\tau} $
  and
  $w^j_p := \partial_t p(t_j) - \frac{\bar{p}_h(t_{j}) -
    \bar{p}_h(t_{j-1})}{\tau}$,
  we have
\begin{equation}\label{ine:e_u}
 \| ( e_{\bmz u}^n, e_p^{n}) \|_{\tau,h}  \leq \| (e_{\bmz u}^0, e_p^{0}) \|_{\tau,h}  + c \tau \sum_{j=1}^n \left(  \| w_{\bmz u}^j \|_a   + \varepsilon^{1/2} h \| \nabla w_p^j \| + \varepsilon^{1/2} h \| \nabla \partial_t p(t_j) \|  \right).
\end{equation}	
If the initial data $\bmz u_h^0$ and $p_h^0$ satisfy \eqref{eqn:biot-mini-init-1} and \eqref{eqn:biot-mini-init-2}, we have,
\begin{eqnarray}
\| e_p^n \|_{a_p} & \leq &\| e_p^{0} \|_{a_p} +  c  \tau^{1/2} \left[  \left( \sum_{j=1}^n \| w_{\bmz u}^j \|_a^2\right)^{1/2}  \right. \nonumber\\
&& \quad \left. +  \left( \sum_{j=1}^n \varepsilon h^2  \| \nabla w_p^j \|^2 \right)^{1/2} + \left( \sum_{j=1}^n \varepsilon h^2  \| \nabla \partial_t p(t_j) \|^2 \right)^{1/2} \right],  \label{ine:e_p}
\end{eqnarray}
and if the initial data $\bmz u_h^0$ and $p_h^0$ are defined by \eqref{eqn:biot-mini-ini-3} and do not satisfy \eqref{eqn:biot-mini-init-1} and \eqref{eqn:biot-mini-init-2}, we have,
\begin{eqnarray}
\| e_p^n \|_{a_p} & \leq &\frac{1}{\sqrt{2 \tau}}  \| (e_{\bmz u}^0,  e_p^{0}) \|_{\tau,h}+  c  \tau^{1/2} \left[  \left( \sum_{j=1}^n \| w_{\bmz u}^j \|_a^2\right)^{1/2}  \right. \nonumber\\
& &\quad \left. +  \left( \sum_{j=1}^n \varepsilon h^2  \| \nabla w_p^j \|^2 \right)^{1/2} + \left( \sum_{j=1}^n \varepsilon h^2  \| \nabla \partial_t p(t_j) \|^2 \right)^{1/2} \right].  \label{ine:e_p_1}
\end{eqnarray}
Moreover, we also have the following estimate in the $L^2$-norm,
\begin{equation}\label{ine:e_p_L2}
  \| e_p^{n} \| \leq c \| (e_{\bmz u}^0, e_p^{0}) \|_{\tau,h}  + c \tau \sum_{j=1}^n \left(  \| w_{\bmz u}^j \|_a   + \varepsilon^{1/2} h \| \nabla w_p^j \| + \varepsilon^{1/2} h \| \nabla \partial_t p(t_j) \|  \right).
\end{equation}	
\end{lemma}

\begin{proof}
  Choosing $\bmz v = \bmz v_h \in \bmz V_h$ in \eqref{eqn:biot-weak-1}
  and $q = q_h \in Q_h$ in \eqref{eqn:biot-weak-2}, and subtracting both
  equations from \eqref{eqn:biot-mini-1} and \eqref{eqn:biot-mini-2},
  and we have for all $\bmz v_h\in \bmz V_h$ and $q_h\in Q_h$
\begin{eqnarray}
&& a(e_{\bmz u}^n, \bmz v_h) - (\ddiv \bmz v_h, e_p^n) = 0, \label{eqn:e_u-weak} \\
&& (\ddiv \bar{\partial}_t e_{\bmz u}^n, q_h) + a_p(e_p^n, q_h) + \varepsilon h^2 (\nabla \bar{\partial}_t e_p^n, \nabla q_h)   \nonumber  \\
&& \qquad \qquad = (\ddiv w^n_{\bmz u}, q_h) + \varepsilon h^2 (\nabla w^n_{p}, \nabla q_h) - \varepsilon h^2 (\nabla \partial_t p(t_n), \nabla q_h). \label{eqn:e_p-weak}
\end{eqnarray}
Choose $v_h = \bar{\partial}_t e_{\bmz u}^n$ in \eqref{eqn:e_u-weak} and $q_h = e_p^n$ in \eqref{eqn:e_p-weak} and add these two equations together, we have
\begin{eqnarray}
\| (e_{\bmz u}^n, e_p^n) \|_{\tau, h}^2
& = &a(e_{\bmz u}^n, e_{\bmz u}^{n-1}) +   \varepsilon h^2 (\nabla e_p^n, \nabla e_p^{n-1}) \nonumber \\
	&+& \tau (\ddiv w_{\bmz u}^n, e_p^n) + \tau \varepsilon h^2 (\nabla w_p^n, \nabla e_p^n) \nonumber \\
	&-& \tau \varepsilon h^2 (\nabla \partial_t p(t_n), \nabla e_p^n)   \label{eqn:error-eu}\\
	& \leq & \| e_{\bmz u}^n \|_a \| e_{\bmz u}^{n-1} \|_a + \varepsilon h^2 \| \nabla e_p^n \| \| \nabla e_p^{n-1} \| \nonumber  \\
	& +&  \tau \| \ddiv w_{\bmz u}^n \| \| e_p^n \| + \tau \varepsilon h^2  \| \nabla w_p^n \| \| \nabla e_p^n \|   + \tau \varepsilon h^2 \| \nabla \partial_t p(t_n) \| \| \nabla e_p^n \| \nonumber
\end{eqnarray}
Thanks to the inf-sup condition \eqref{ine:inf-sup}, and \eqref{eqn:e_u-weak} we have
\begin{eqnarray}
\| e_p^n \| &\leq &c \sup_{\bmz v_h \neq 0} \frac{(\ddiv \bmz v_h, e_p^n)}{\| \bmz v_h \|_a } + c_1 \varepsilon^{1/2} h \| \nabla e_p^n \| \nonumber \\
&=& c \sup_{v_h\in V_h} \frac{a(e_{\bmz u}^n, \bmz v_h)}{\| \bmz v_h \|_a} + c_1 \varepsilon^{1/2} h \| \nabla e_p^n \| = c \| e_{\bmz u}^n \|_a + c_1 \varepsilon^{1/2} h \| \nabla e_p^n \|.  \label{ine:inf-sup-1}
\end{eqnarray}
Note, for MINI element, we have $c_1 = 0$ and, for P1-P1 element, $c_1 > 0$.  Therefore,
\begin{eqnarray*}
\| (e_{\bmz u}^n, e_p^n) \|_{\tau,h}^2 & \leq &\| e_{\bmz u}^n \|_a \| e_{\bmz u}^{n-1} \|_a + \varepsilon h^2 \| \nabla e_p^n \| \| \nabla e_p^{n-1} \|  \\
& & \quad + c \tau \| w_{\bmz u}^n \|_a \left( \| e_{\bmz u}^n \|_a + c_1 \varepsilon^{1/2} h \| \nabla e_p^n \| \right)  \nonumber \\
 && \quad +  \tau \varepsilon h^2  \| \nabla w_p^n \| \| \nabla e_p^n \| +  \tau \varepsilon h^2 \| \nabla \partial_t p(t_n) \| \| \nabla e_p^n \|
\end{eqnarray*}
which implies
$$
\| (e_{\bmz u}^n, e_p^{n}) \|_{\tau, h}  \leq \| (e_{\bmz u}^{n-1},  e_p^{n-1}) \|_{\tau,h} + c \tau \left( \| w^n_{\bmz u} \|_a + \varepsilon^{1/2} h \| \nabla w_p^n \| + \varepsilon^{1/2} h \| \nabla \partial_t p(t_n) \| \right)
$$
We sum over all time steps and we have the estimate \eqref{ine:e_u}.

For the error estimate of $e_p^n$, from \eqref{eqn:e_u-weak}, we have,
\begin{equation} \label{eqn:error-ep}
a(\bar{\partial}_t e_{\bmz u}^n, \bmz v_h) - (\ddiv \bmz v_h, \bar{\partial}_t e_p^n) = 0.
\end{equation}
Note that, if the initial data $\bmz u_h^0$ and $p_h^0$ satisfy
\eqref{eqn:biot-mini-init-1} and \eqref{eqn:biot-mini-init-2},
\eqref{eqn:error-ep} holds for $n=1,2,3,\ldots$. Otherwise, for
initial data \eqref{eqn:biot-mini-ini-3}, \eqref{eqn:error-ep} only
holds for $n=2,3, \ldots$

Choosing $\bmz v_h = \bar{\partial}_t e_{\bmz u}^n$ in
\eqref{eqn:error-ep} and $q_h = \bar{\partial}_t e_p^n$ in
\eqref{eqn:e_p-weak} and adding the two equations, and we have
\begin{eqnarray*}
&  &\tau^{-1} \| e_{\bmz u}^n - e_{\bmz u}^{n-1} \|_a^2 + \| e_p^n \|_{a_p}^2 + \tau \varepsilon h^2 \| \nabla \bar{\partial}_t e_p^n \|^2 \\
&&\quad \leq \| e_p^n \|_{a_p} \| e_p^{n-1} \|_{a_p} + \| \ddiv w_{\bmz u}^n \| \| e_p^n - e_p^{n-1} \|  \\
&& \qquad + \tau \varepsilon h^2 \| \nabla w_p^n \| \| \nabla \bar{\partial}_t e_p^n\|  + \tau \varepsilon h^2 \| \nabla \partial_t p(t_n) \| \| \nabla \bar{\partial}_t e_p^n\|  \\
&&\quad \leq \| e_p^n \|_{a_p} \| e_p^{n-1} \|_{a_p} + c\| \ddiv w_{\bmz u}^n  \|\left( \| e_{\bmz u}^n - e_{\bmz u}^{n-1} \|_a + c_1 \varepsilon^{1/2} h  \| \nabla (e_p^n - e_p^{n-1}) \|  \right) \\
&& \qquad + \tau \varepsilon h^2 \| \nabla w_p^n \| \| \nabla \bar{\partial}_t e_p^n\|  + \tau \varepsilon h^2 \| \nabla \partial_t p(t_n) \| \| \nabla \bar{\partial}_t e_p^n\|  \\
&& \quad \leq \frac{1}{2} \| e_p^n \|_{a_p}^2 + \frac{1}{2} \| e_p^{n-1}\|_{a_p}^2 + c \tau \| \ddiv w_{\bmz u}^n \|^2 + \tau^{-1} \| e_{\bmz u}^n - e_{\bmz u}^{n-1} \|_a^2  + \frac{1}{3} \tau \varepsilon h^2 \| \nabla \bar{\partial}_t e_p^n \|^2 \\
&& \qquad + \frac{3}{4} \tau \varepsilon h^2  \| \nabla w_p^n \|^2 + \frac{1}{3} \tau \varepsilon h^2  \| \nabla \bar{\partial}_t e_p^n\|^2 + \frac{3}{4} \tau \varepsilon h^2 \| \nabla \partial_t p(t_n) \|^2 + \frac{1}{3} \tau \varepsilon h^2  \| \nabla \bar{\partial}_t e_p^n\|^2,
\end{eqnarray*}
where we use the inf-sup condition \eqref{ine:inf-sup-1} to estimate $\| e_p^n - e_p^{n-1} \|$.  Now we have
\begin{equation}\label{ine:err-ep-recur}
\| e_p^n \|_{a_p}^2 \leq \| e_p^{n-1} \|_{a_p}^2 + c   \left( \tau \| w_{\bmz u}^n \|_a^2  +   \tau \varepsilon h^2  \| \nabla w_p^n \|^2 + \tau \varepsilon h^2  \| \nabla \partial_t p(t_n) \|^2 \right).
\end{equation}
Now we need to consider two different cases due to the initial data.  If the initial data satisfy \eqref{eqn:biot-mini-init-1} and \eqref{eqn:biot-mini-init-2}, then above inequality \eqref{ine:err-ep-recur} holds for $n=1$ and by summing up from $1$ to $n$, we can get \eqref{ine:e_p}.

If the initial data is only defined by \eqref{eqn:biot-mini-ini-3}, \eqref{ine:err-ep-recur} does not hold for $n=1$ anymore, we need to estimate $\| e_p^1 \|$ separately.  In order to do that, we take $n=1$ in \eqref{eqn:error-eu} and then use the inf-sup condition \eqref{ine:inf-sup-1} to estimate $\| e_p^1 \|$,
\begin{eqnarray*}
\| e_{\bmz u}^1 \|_a^2 + \tau \| e_p^1 \|_{a_p}^2  + \varepsilon h^2 \| \nabla e_p^1 \|^2
& = & a(e_{\bmz u}^1, e_{\bmz u}^{0}) +   \varepsilon h^2 (\nabla e_p^1, \nabla e_p^{0}) \nonumber \\
	& +& \tau (\ddiv w_{\bmz u}^1, e_p^1) + \tau \varepsilon h^2 (\nabla w_p^1, \nabla e_p^1) \nonumber \\
	&-& \tau \varepsilon h^2 (\nabla \partial_t p(t_1), \nabla e_p^1)  \\
	& \leq & \frac{1}{2} \| e_{\bmz u}^1 \|_a^2 + \frac{1}{2} \| e_{\bmz u}^0 \|^2_a + \frac{1}{2} \varepsilon h^2 \| \nabla e_p^1 \|^2  + \frac{1}{2} \varepsilon h^2 \| \nabla e_p^0 \|^2 \\
	& + & c \tau^2 \| w_{\bmz u}^1 \|_a^2 + \frac{1}{2} \| e_{\bmz u}^1 \|_a^2 + \frac{1}{6} \varepsilon h^2 \| \nabla e_p^1 \|^2  + \frac{3}{2}\tau^2 \varepsilon h^2 \| \nabla w_p^1 \|^2  \\
	& + & \frac{1}{6} \varepsilon h^2 \| \nabla e_p^1 \|^2 + \frac{3}{2}\tau^2 \varepsilon h^2 \| \nabla \partial_tp(t_1) \|^2 + \frac{1}{6} \varepsilon h^2 \| \nabla e_p^1 \|^2.
\end{eqnarray*}
This means
\begin{eqnarray*}
\| e_p^1 \|_{a_p}^2 & \leq &\frac{1}{2 \tau}  \|( e_{\bmz u}^0, e_p^0) \|_{\tau, h} + c \left(
\tau \| w_{\bmz u}^1 \|_a^2  +   \tau \varepsilon h^2  \| \nabla w_p^1 \|^2 + \tau \varepsilon h^2  \| \nabla \partial_t p(t_1) \|^2
 \right).
\end{eqnarray*}
Now we summing up \eqref{ine:err-ep-recur} from $2$ to $n$ and use above estimate of $\| e_p^1 \|^2_{a_p}$, we can get \eqref{ine:e_p_1}.

Finally, the estimate \eqref{ine:e_p_L2} follows directly from \eqref{ine:e_u} and \eqref{ine:inf-sup-1}.
\end{proof}

Next lemma give the estimations of $w_{\bmz u}^j$ and $w_p^j$.
\begin{lemma} \label{lem:error-wj}
Let $\bmz u(t)$ and $p(t)$ be the solution of \eqref{eqn:biot-weak-1} and \eqref{eqn:biot-weak-2}, $w_{\bmz u}^j = \partial_t \bmz u(t_j) - \frac{\bar{\bmz u}_h(t_{j}) - \bar{\bmz u}_h(t_{j-1})}{\tau}$ and $\rho_{\bmz u}(t) = \bmz u(t) - \bar{\bmz u}_h(t)$. Assume $\bmz \partial_{tt}u(t) \in L^1((0,T], \bmz [H_0^1(\Omega)]^d) \cap L^2((0,T], \bmz [H_0^1(\Omega)]^d) $ and $ \partial_{tt} p(t) \in L^1((0,T], H_0^1(\Omega)) \cap L^2((0,T], H_0^1(\Omega)) $, we have,
\begin{eqnarray}
&& \sum_{j=1}^n \| w_{\bmz u}^j \|_a \leq c \left( \int_{0}^{t_n} \| \partial_{tt} \bmz u \|_1 \mathrm{d} t + \frac{1}{\tau} \int_{0}^{t_n} \| \partial_t \rho_{\bmz u} \|_1 \mathrm{d} t  \right), \label{ine:wj-l1} \\
&& \sum_{j=1}^n \| w_{\bmz u}^j \|_a^2 \leq c \left( \tau \int_{0}^{t_n} \| \partial_{tt} \bmz u \|^2_1 \mathrm{d} t + \frac{1}{\tau} \int_{0}^{t_n} \| \partial_t \rho_{\bmz u} \|^2_1 \mathrm{d} t  \right). \label{ine:wj-l2}
\end{eqnarray}
Moreover, let $w_p^j = \partial_t p(t_j) - \frac{\bar{p}_h(t_{j}) - \bar{p}_h(t_{j-1})}{\tau}$ and $\rho_{p} = p(t) - \bar{p}_h(t)$. we have
\begin{eqnarray}
&& \sum_{j=1}^n \| \nabla w_{p}^j \| \leq c \left( \int_{0}^{t_n} \| \partial_{tt} p \|_1 \mathrm{d} t + \frac{1}{\tau} \int_{0}^{t_n} \| \partial_t \rho_{p} \|_1 \mathrm{d} t  \right), \label{ine:wjp-l1} \\
&& \sum_{j=1}^n \| \nabla w_{p}^j \|^2 \leq c \left( \tau \int_{0}^{t_n} \| \partial_{tt} p \|^2_1 \mathrm{d} t + \frac{1}{\tau} \int_{0}^{t_n} \| \partial_t \rho_{p} \|^2_1 \mathrm{d} t  \right). \label{ine:wjp-l2}
\end{eqnarray}
\end{lemma}
\begin{proof}
We consider
\begin{equation*}
w_{\bmz u}^j  =\left( \partial_t \bmz{u}(t_j) - \frac{\bmz u(t_j) - \bmz u(t_{j-1})}{\tau}\right) + \left(\frac{\bmz u(t_j) - \bmz u(t_{j-1})}{\tau} - \frac{\bar{\bmz u}_h(t_{j}) - \bar{\bmz u}_h(t_{j-1})}{\tau} \right)=: w^j_{\bmz u, 1} + w^j_{\bmz u, 2}.
\end{equation*}
Note that
\begin{eqnarray*}
w^j_{\bmz u, 1} &=& \frac{1}{\tau} \int_{t_{j-1}}^{t_j} (s - t_{j-1}) \partial_{tt} \bmz u(s) \mathrm{d}s, \\
w^j_{\bmz u, 2} &= &\frac{1}{\tau} \int_{t_{j-1}}^{t_j} \partial_t \rho_{\bmz u}(s) \mathrm{d}s,
\end{eqnarray*}
then we have
\begin{eqnarray*}
\| w^j_{\bmz u} \|_a & \leq &\| w^j_{\bmz u, 1} \|_a + \| w^j_{\bmz u,2} \|_a \\
  & = & \frac{1}{\tau} \| \int_{t_{j-1}}^{t_j} (s - t_{j-1}) \partial_{tt} \bmz u(s) \mathrm{d}s \|_a + \frac{1}{\tau} \| \int_{t_{j-1}}^{t_j} \partial_t \rho_{\bmz u}(s) \mathrm{d}s \|_a \\
	& \leq & c \left( \int_{t_{j-1}}^{t_j}  \| \partial_{tt} \bmz u \|_1 \mathrm{d}s  + \frac{1}{\tau} \int_{t_{j-1}}^{t_j} \| \partial_t \rho_{\bmz u}\|_1 \mathrm{d}s \right),
\end{eqnarray*}
then \eqref{ine:wj-l1} follows directly. Moreover, we have
\begin{eqnarray*}
\| w_{\bmz u}^j \|_a^2 & \leq &c \left[ \tau^{1/2} \left( \int_{t_{j-1}}^{t_{j}} \| \partial_{tt} \bmz u \|_1^2 \mathrm{d}s \right)^{1/2}  + \tau^{-1/2} \left( \int_{t_{j-1}}^{t_j} \| \partial_t \rho_{\bmz u} \|_1^2 \mathrm{d}s  \right)^{1/2}   \right]^2 \\
	& \leq & c \left(  \tau \int_{t_{j-1}}^{t_{j}} \| \partial_{tt} \bmz u \|_1^2 \mathrm{d}s + \frac{1}{\tau}  \int_{t_{j-1}}^{t_j} \| \partial_t \rho_{\bmz u} \|_1^2 \mathrm{d}s \right),
\end{eqnarray*}
then \eqref{ine:wj-l2} follows directly.  Estimates \eqref{ine:wjp-l1} and \eqref{ine:wjp-l2} can be obtained similarly, which completes the proof
\end{proof}

Assuming extra regularities of the exact solutions
$\bmz u(t)$ and $p(t)$ as usual for convergence analysis of the finite
element method, we have the following theorem about the error
estimates for the error $(\bmz u - \bmz u_h)(t_n)$ and
$(p - p_h)(t_n)$.

We assume that $u$ and $p$ have all the regularity required by the
proof of the theorem below, which more precisely means that, for
$q=1,2,\infty$ and $s=1,2$ we have:
\begin{eqnarray*}
&&  \bmz u(t) \in L^{\infty}\left((0,T], \bmz [H_0^1(\Omega)]^d\right) \cap
  L^{\infty} \left((0, T], \bmz [H^2(\Omega)]^d \right),\\
&& \partial_t \bmz u(t) \in L^s\left( (0,T], \bmz [H^2(\Omega)]^d \right),
   \quad
\partial_{tt} \bmz u(t) \in L^s\left( (0, T], \bmz [H_0^1(\Omega)]^d
   \right),\\
&& p(t) \in L^{\infty} \left( (0,T], H^1_0(\Omega) \right) \cap L^{\infty} \left( (0,T], H^2(\Omega) \right),\\
&& \partial_t p(t) \in
L^q((0,T],H_0^1(\Omega))\cap L^s((0,T],H^2(\Omega)),\quad
 \partial_{tt} p(t) \in L^s\left( (0, T], H_0^1(\Omega) \right)
\end{eqnarray*}

\begin{theorem}
\label{thm:error-mini} Let $\bmz u(t)$ and $p(t)$ be
  the solution of \eqref{eqn:biot-weak-1} and \eqref{eqn:biot-weak-2},
  $\bmz u_h^n$ and $p_h^n$ be the solution of \eqref{eqn:biot-mini-1}
  and \eqref{eqn:biot-mini-2}.  For displacement $\bmz u(t)$, we have
\begin{eqnarray}
&& \| \left( \bmz u(t_n) - \bmz u_h^n, p(t_n) - p_h^n \right) \|_{\tau,h}  \nonumber \\
&& \leq  \|  \left( e_{\bmz u}^0, e_p^0 \right) \|_{\tau, h} +
 c \left\{  \tau \left[  \int_{0}^{t_n} \| \partial_{tt} \bmz u \|_1 \mathrm{d} t + \int_{0}^{t_n} \varepsilon^{1/2} h | \partial_{tt} p |_1 \mathrm{d}t  \right]    \right.  \nonumber \\
 && \quad   +  h  \left[ | \bmz u(t_n) |_2 + | p(t_n) |_1 + (\tau^{1/2}+\varepsilon^{1/2} h) | p(t_n)|_2 +
\int_{0}^{t_n} \left(| \partial_t \bmz u|_2 + | \partial_t p |_1 \right) \mathrm{d}t  \right. \nonumber \\
 && \quad \left.  \left.
      +   \int_{0}^{t_n} \varepsilon^{1/2} h | \partial_t p |_2 \mathrm{d}t
\right]  + t_n\max_{1\leq j\leq n} \varepsilon^{1/2} h \| \nabla \partial_t p(t_j) \|  \right\}.  \label{ine:error-mini-u}
\end{eqnarray}
For pore pressure $p(t)$, if the initial data $\bmz u_h^0$ and $p_h^0$ satisfy \eqref{eqn:biot-mini-init-1} and \eqref{eqn:biot-mini-init-2}, we have,
\begin{eqnarray}
&& \| p(t_n) - p_h^n \|_{a_p} \nonumber  \\
&&  \leq    \| e_p^0 \|_{a_p} + c \left\{ \tau \left[  \left( \int_{0}^{t_n} \| \partial_{tt} \bmz u \|^2_1 \mathrm{d} t \right)^{1/2} + \left( \int_{0}^{t_n}\varepsilon h^2 \| \partial_{tt} p \|^2_1 \mathrm{d} t \right)^{1/2} \right] \right. \nonumber \\
&&  \quad +  h \left[ |p(t_n)|_2 +  \left( \int_{0}^{t_n} \left(| \partial_t \bmz u|_2 + | \partial_t p |_1 \right) ^2 \mathrm{d} t  \right)^{1/2} +  \left( \int_{0}^{t_n} \varepsilon h^2 | \partial_t p |_2^2 \mathrm{d} t  \right)^{1/2}   \right] \nonumber \\
&& \left. \quad + \sqrt{t_n}\max_{1\leq j\leq n} \varepsilon^{1/2} h \| \nabla \partial_t p(t_j) \|   \right\}. \label{ine:error-mini-p}
\end{eqnarray}
If the initial data $\bmz u_h^0$ and $p_h^0$ are defined by \eqref{eqn:biot-mini-ini-3}, we have,
\begin{eqnarray}
&& \| p(t_n) - p_h^n \|_{a_p} \nonumber  \\
&&  \leq   \frac{1}{\sqrt{2\tau}}  \| \left( e_{\bmz u}^0, e_p^{0} \right) \|_{\tau, h}+ c \left\{ \tau \left[  \left( \int_{0}^{t_n} \| \partial_{tt} \bmz u \|^2_1 \mathrm{d} t \right)^{1/2} + \left( \int_{0}^{t_n}\varepsilon h^2 \| \partial_{tt} p \|^2_1 \mathrm{d} t \right)^{1/2} \right] \right. \nonumber \\
&&  \quad +  h \left[ |p(t_n)|_2 +  \left( \int_{0}^{t_n} \left(| \partial_t \bmz u|_2 + | \partial_t p |_1 \right) ^2 \mathrm{d} t  \right)^{1/2} +  \left( \int_{0}^{t_n} \varepsilon h^2 | \partial_t p |_2^2 \mathrm{d} t  \right)^{1/2}   \right] \nonumber \\
&& \left. \quad + \sqrt{t_n}\max_{1\leq j\leq n} \varepsilon^{1/2} h \| \nabla \partial_t p(t_j) \|   \right\}. \label{ine:error-mini-p-1}
\end{eqnarray}
Moreover, for pore pressure,   we also have the following error estimate in $L^2$-norm,
\begin{eqnarray}
&& \| p(t_n) -  p_h^n \|  \nonumber \\
&& \leq c \|  \left( e_{\bmz u}^0, e_p^0 \right) \|_{\tau, h} +
 c \left\{  \tau \left[  \int_{0}^{t_n} \| \partial_{tt} \bmz u \|_1 \mathrm{d} t + \int_{0}^{t_n} \varepsilon^{1/2} h | \partial_{tt} p |_1 \mathrm{d}t  \right]    \right.  \nonumber \\
 && \quad   +  h^2 | p(t_n)|_2  +  h  \left[  \int_{0}^{t_n} \left(| \partial_t \bmz u|_2 + | \partial_t p |_1 \right) \mathrm{d}t   + \int_{0}^{t_n} \varepsilon^{1/2} h | \partial_t p |_2 \mathrm{d}t
\right]  \nonumber \\
 && \quad \left.
 + t_n\max_{1\leq j\leq n} \varepsilon^{1/2} h \| \nabla \partial_t p(t_j) \|  \right\}.  \label{ine:error-mini-p-L2}
\end{eqnarray}
\end{theorem}

\begin{proof}
  The estimate \eqref{ine:error-mini-u} follows directly from
  \eqref{eqn:err_u_decomp}, \eqref{eqn:err_p_decomp}, \eqref{rhou}, \eqref{rhop-1}, \eqref{ine:e_u},
  \eqref{ine:wj-l1}, \eqref{ine:wjp-l1},
  and triangle inequality. Note that we used \eqref{rhou} and
  \eqref{rhop-1} not only for $u$, $p$, but also their counterparts
  for $\partial_t \rho_u$ and $\partial_t \rho_p$.

Similarly,   \eqref{ine:error-mini-p} follows from \eqref{eqn:err_p_decomp},
  \eqref{ine:e_p}, \eqref{ine:wj-l2}, \eqref{ine:wjp-l2},
  \eqref{rhou}, \eqref{rhop-1}, and their versions for the time
  derivatives of the error and the triangle inequality.

  Next, for the second set of initial conditions,
  \eqref{ine:error-mini-p-1} follows from \eqref{eqn:err_p_decomp},
  \eqref{ine:e_p_1}, \eqref{ine:wj-l2}, \eqref{ine:wjp-l2},
  \eqref{rhou}, \eqref{rhop-1} (applied also for time derivatives of
  the error), and the triangle inequality.

  Finally, \eqref{ine:error-mini-p-L2} follows from
  \eqref{eqn:err_p_decomp}, \eqref{rhop-2}, \eqref{ine:e_p_L2}, \eqref{ine:wj-l1},
  \eqref{ine:wjp-l1} and the triangle inequality.
\end{proof}

\begin{remark}
  All the error estimates in Theorem~\ref{thm:error-mini} consist of two
  parts.  One part is the error for $t>0$ which, in all cases, gives
  optimal convergence order.  The other part is the error in the
  approximation of the
  initial data, i.e., $ \| (e_{\bmz u}^0, e_p^{0} )\|_{\tau,h}$ and
  $\| e_p^0 \|_{a_p}$.  From the  triangle inequality, we have
\begin{eqnarray*}
\| ( e_{\bmz u}^0, e_p^{0}) \|_{\tau, h}& \leq &
c \left[ \| ( \rho_{\bmz u}^0, \rho_p^{0} )\|_{\tau,h}  + \| \left( \bmz{u}(0) - \bmz u_h^0, p(0) - p_h^0 \right) \|_{\tau,h} \right] , \\
\| e_p^0 \|_{a_p} &\leq &\| \rho_p^0 \|_{a_p} + \| p(0) - p_h^0 \|_{a_p},
\end{eqnarray*}
where $\rho_{\bmz u}^0$ and $\rho_p^0$ are the errors due to the
elliptic projection and $(\bmz u(0) - \bmz u_h^0)$ and
$(p(0) - p_h^0)$ are the errors due to the choice of initial
conditions, either satisfying Stokes equation
\eqref{eqn:biot-mini-init-1} and \eqref{eqn:biot-mini-init-2} or the
simpler
given in \eqref{eqn:biot-mini-ini-3}.

If the initial data satisfies the stabilized Stokes equation
\eqref{eqn:biot-mini-init-1} and \eqref{eqn:biot-mini-init-2}, the
initial errors strongly depend on the regularity of the initial data.
A crucial role is played by the assumptions on the regularity of the
pore pressure $p(0)$.  If we assume $p(0) \in H_0^1(\Omega)$, then the
standard error estimates for the elliptic projection and stabilized
Stokes equation show that the initial data errors are appropriately
bounded, and, hence, we have optimal order of convergence for the
discrete scheme. Therefore, the overall convergence rate of the
stabilized MINI element is optimal.  However, if we assume that $p(0)$
is merely in $L^2(\Omega)$, then we cannot expect that the errors in
the initial data are of optimal order, and, therefore, the overall
convergence rate of the stabilized MINI element is not optimal as
well.

If we just use the simple practical choice
\eqref{eqn:biot-mini-ini-3}, we cannot expect that $\bmz u_h^0$
$p_h^0$ approximate $\bmz u(0)$ and $p(0)$ in general.  Therefore,
regardless of the regularity assumption of the initial data, the
overall convergence rate of the stabilized MINI element will not be as
desired. However, in some cases, even when the initial errors are
large, they decay with respect to time
(see~\cite{MuradLoula94}).  As a consequence, the
discretization error when using stabilized mini element is still
optimal for sufficiently large time (long time).
\end{remark}

\section{Numerical Experiments}\label{sec:numerics}
In this section, we present several numerical experiments in order to
illustrate the performance of the proposed stabilized methods. We will
choose well-known benchmark problems in order to deal with different
aspects as variable permeability, different boundary conditions, the
accuracy of the approximations, etc.

\subsection{Layered porous medium with variable
  permeability} \label{test:Variable_k}

In the first experiment we want to illustrate non-monotone pressure
behavior when we have a low permeability in a sub-domain. We consider
a test proposed in~\cite{NAG:NAG1062} which models a porous
material on which a low--permeable layer ($K = 10^{-8}$) is placed
between two layers with unit permeability ($K = 1$), as shown in
Figure~\ref{problem_variable_k}.
\begin{figure}[htb]
\begin{center}
\includegraphics*[width = 0.4\textwidth]{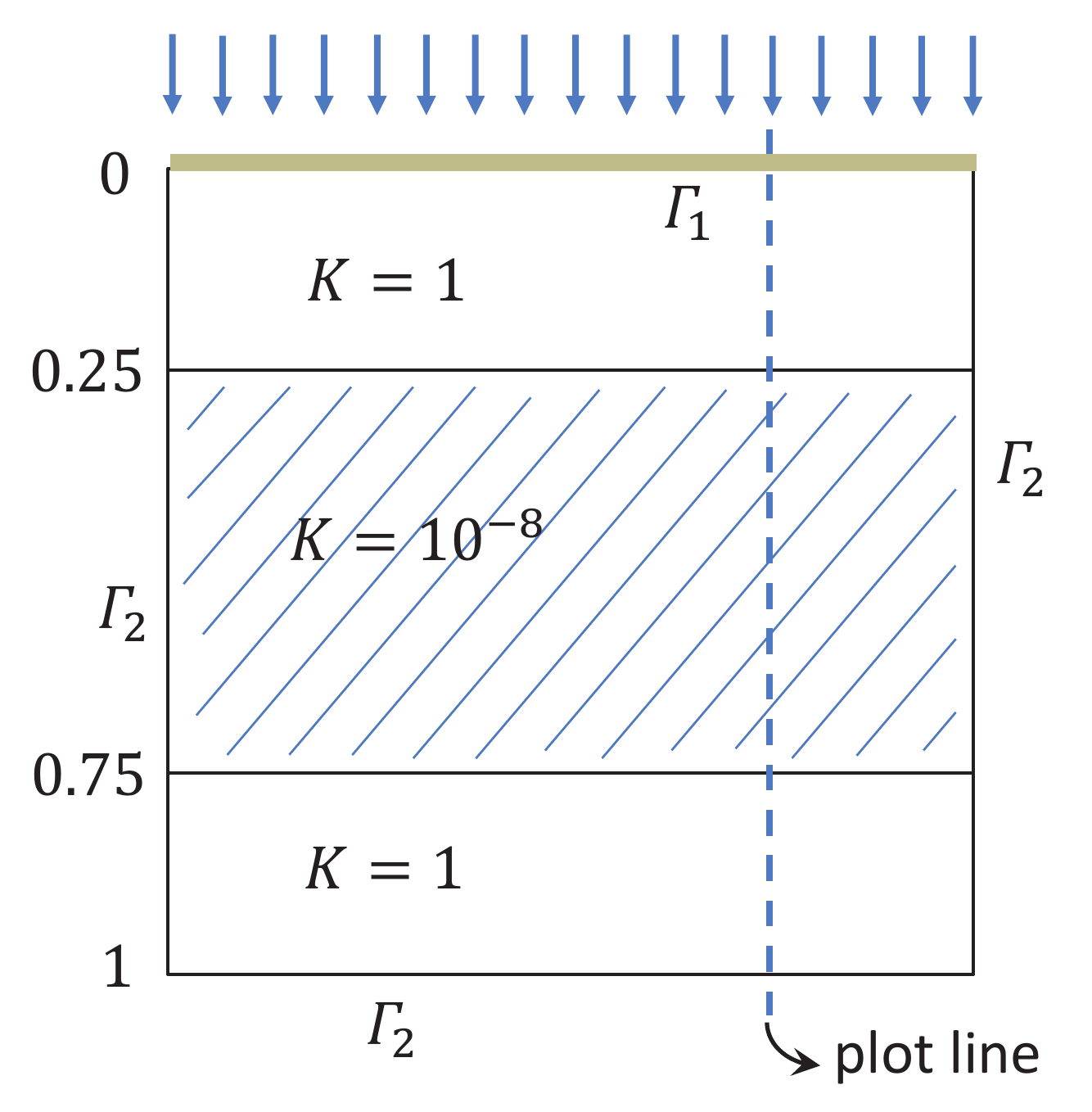}
\caption{Domain representing a square of layered porous material with different permeability.}
\label{problem_variable_k}
\end{center}
\end{figure}
The boundary of the square domain is split in two disjoint subsets
$\Gamma_1$ and $\Gamma_2$ on which we assume the following boundary
conditions: on the top, which is free to drain, a uniform load is
applied, that is,
\begin{equation}\label{bc_gamma1}
p = 0, \qquad \sigma \cdot n = g, \; \hbox{with} \; g = (0,-1)^t, \; \hbox{on} \; \Gamma_1,
\end{equation}
whereas at the sides and bottom that are rigid the boundary is considered to be impermeable , that is,
\begin{equation}\label{bc_gamma2}
\nabla p \cdot n = 0, \qquad u = 0, \; \hbox{on} \; \Gamma_2.
\end{equation}
Zero initial conditions are considered for both variables, and the time step is chosen as $\tau=1$.
Notice that this test can be reduced to a one-dimensional problem. Therefore, in the following simulations we will show the numerical solutions corresponding to one vertical line in the domain as displayed in Figure~\ref{problem_variable_k}. \\
\begin{figure}[htb]
\begin{center}
\begin{tabular}{cc}
\includegraphics*[width = 0.4\textwidth]{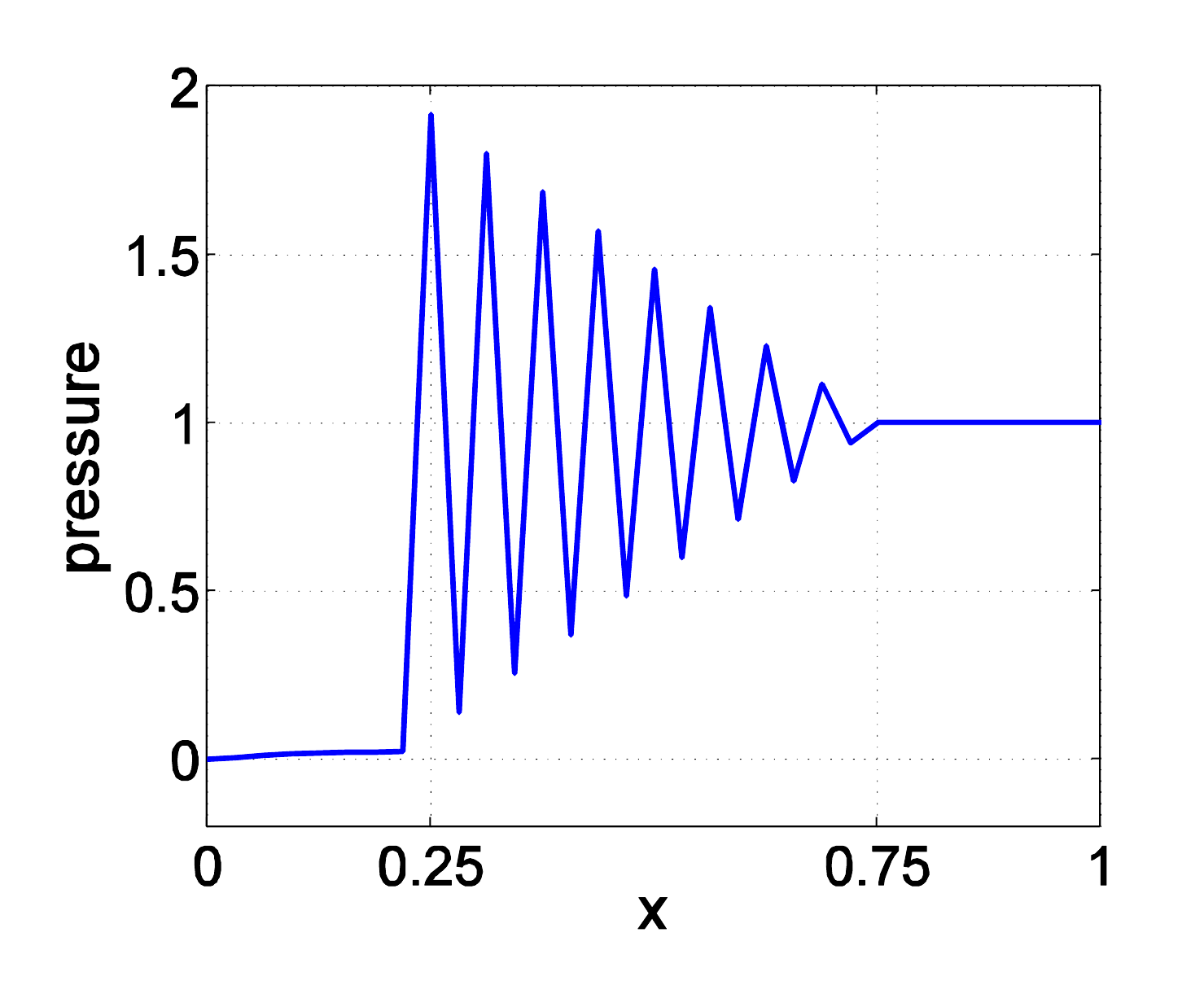}
&
\includegraphics*[width = 0.4\textwidth]{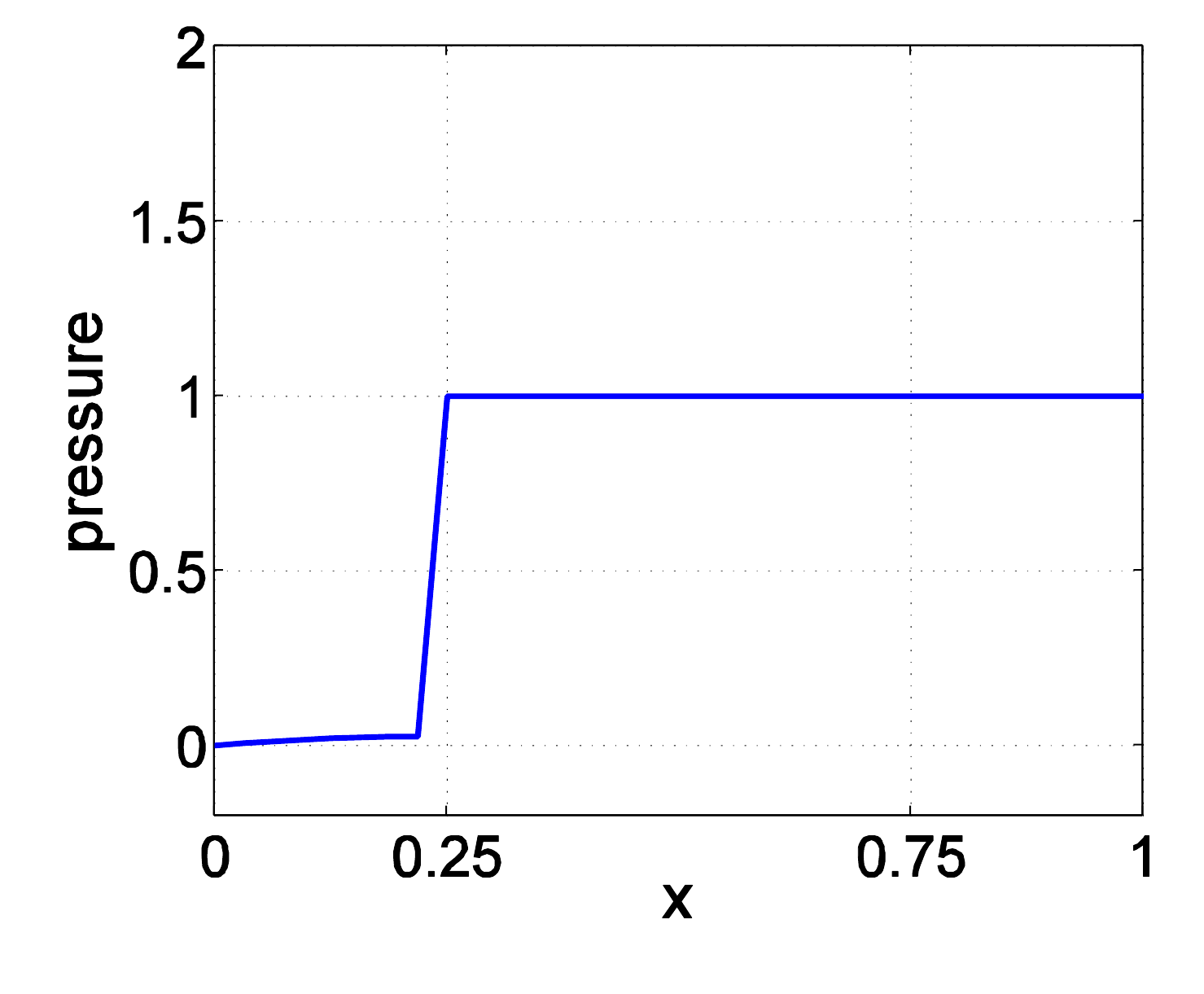}
\\
(a)
&
(b)
\end{tabular}
\caption{Numerical solution by P1--P1 for the pressure to the two-material problem (a) without stabilization term and (b) with stabilization term.}
\label{p1_p1_variable_k}
\end{center}
\end{figure}
First we approximate using linear finite
elements for displacements and pressure. If no stabilization term is
added to the discrete formulation, the approximation for the pressure
field that is obtained by using $32$ elements on the grid is shown in
Figure~\ref{p1_p1_variable_k} (a). We observe that strong spurious
oscillations appear in the part corresponding to the low-permeable
layer. However, if the stabilized scheme is used for the
simulation with the same number of nodes, the oscillations are
completely eliminated and the method gives rise to the monotone
solution for the pressure, as we see in Figure~\ref{p1_p1_variable_k} (b).

 \begin{figure}[htb]
 \begin{center}
 \begin{tabular}{cc}
 \includegraphics*[width = 0.4\textwidth]{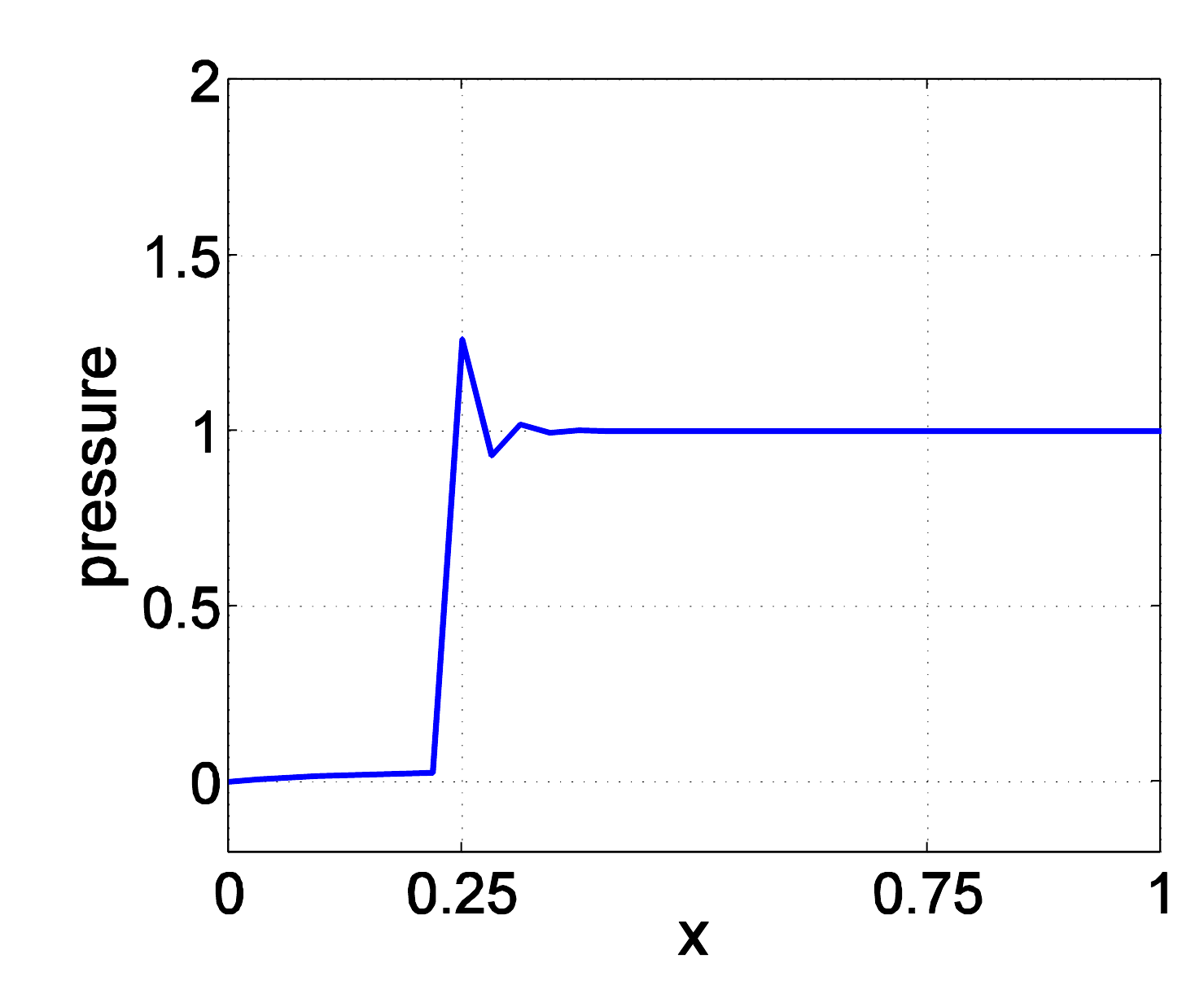}
 &
 \includegraphics*[width = 0.4\textwidth]{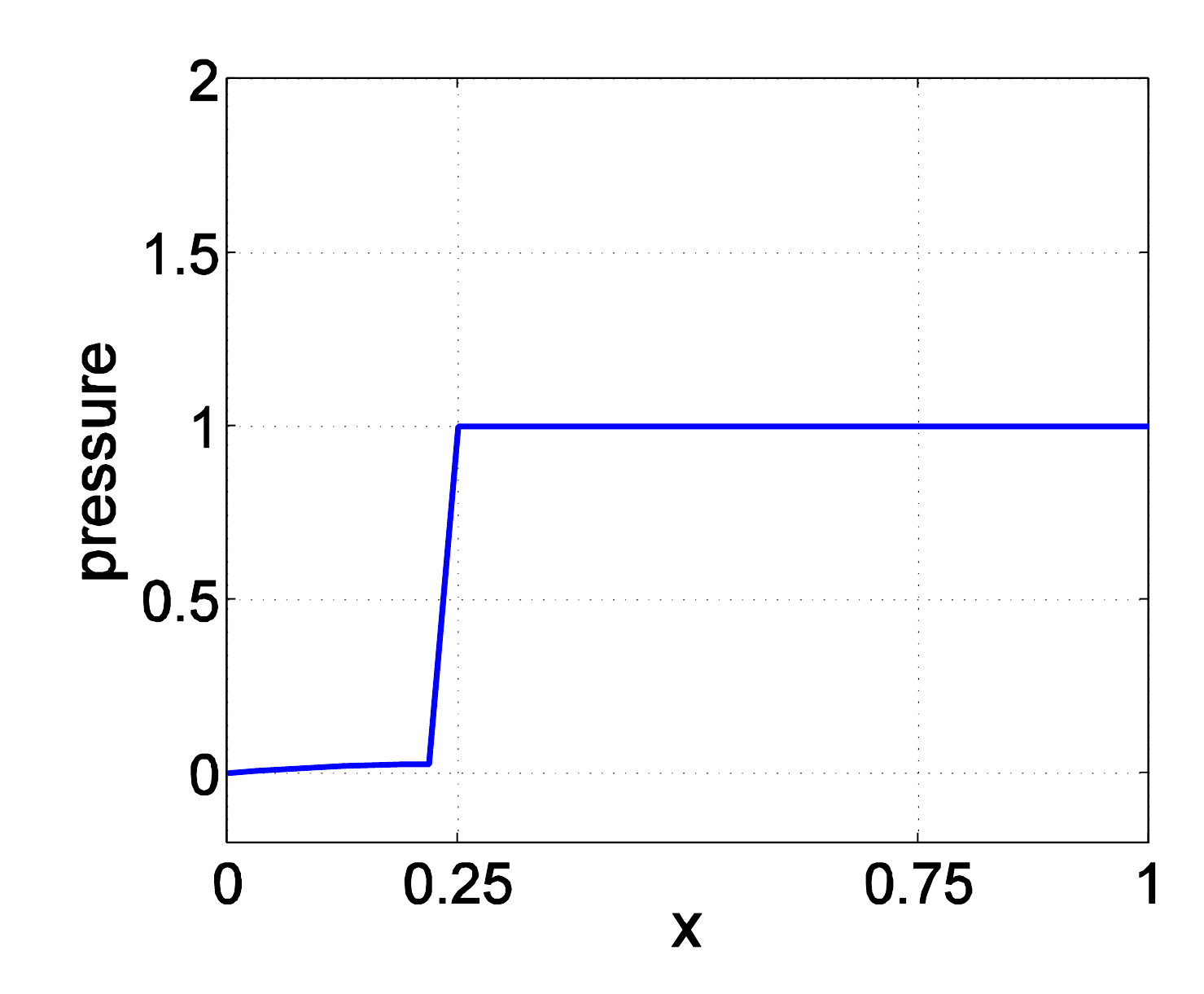}
 \\
 (a)
 &
 (b)
\end{tabular}
\caption{Numerical solution by P2--P1 for the pressure to the two-material problem (a) without stabilization term and (b) with stabilization term.}
\label{p2_p1_variable_k}
\end{center}
\end{figure}

Next, we use approximation by MINI element with the same number of
elements. Similarly to the previous case, when no
stabilization parameter is included in the formulation, the
oscillatory behaviour of the pressure approximation is evident, as shown
in Figure~\ref{p2_p1_variable_k} (a). Notice that the oscillations are
much smaller than in the case of P1--P1 elements, but are still not
eliminated by using this Stokes stable pair of spaces.  Again, a
perturbation stabilizes the method and
we obtain oscillation-free approximation for the pressure field (see Figure~\ref{p2_p1_variable_k} (b)).

\subsection{Mandel's problem}\label{test:Mandel}
Mandel's problem (see~\cite{Mandel}) is an important benchmark
problem because the analytical solution in two dimensions on
a finite domain is known. It is an excellent model that can be used
to verify the accuracy of a discretization. Mandel's problem models an
infinitely long poroelastic slab sandwiched at the top and the
bottom by two rigid frictionless and impermeable plates. The material
is assumed incompressible and saturated with a single-phase
incompressible fluid.
 \begin{figure}[htb]
 \begin{center}
 \includegraphics*[width = 0.4\textwidth]{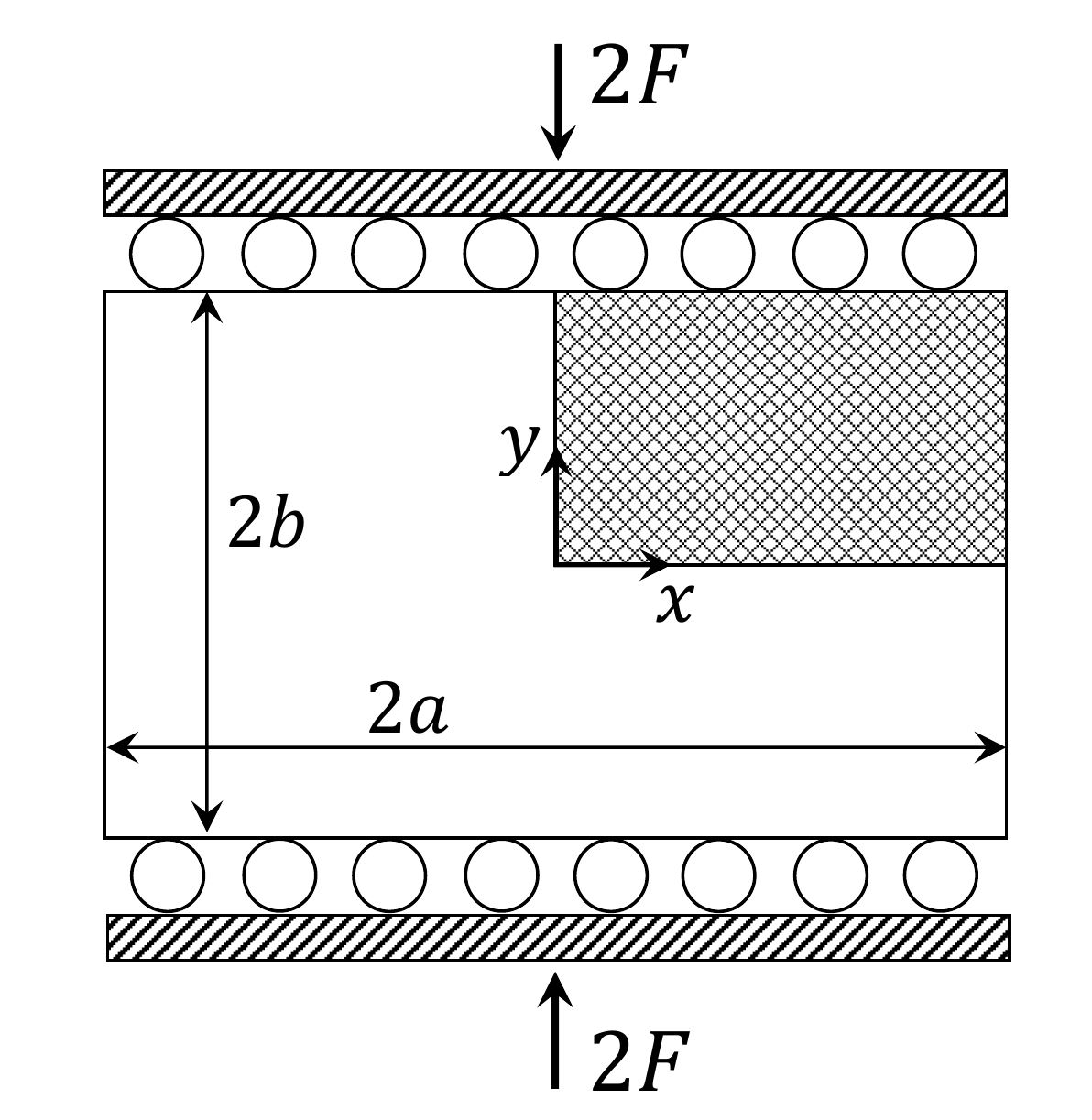}
\caption{2D physical and computational domains for Mandel's problem.}
\label{Mandel_Problem}
\end{center}
\end{figure}
Both plates are loaded by a constant vertical force as shown in
Figure~\ref{Mandel_Problem}, where a $2a\times 2b$ wide cross-section
is displayed. The force of magnitude $2F$ per unit length is suddenly
applied at $t = 0$, generating an instantaneous overpressure by the
Skempton effect~\cite{Skempton}, which will dissipate near the
side edges as time progresses due to the drainage effect, since the
side surfaces ($x = \pm a$) are drained and traction-free.  In this
problem, it turns out that the horizontal displacement $u$ is
independent of the vertical direction $y$, whereas the vertical
displacement $v$ is independent of the horizontal coordinate $x$.  The
analytical solution for the pore pressure can be found in~\cite{abousleiman1996mandel} and is given as follows
\begin{equation}\label{analytical_p}
p(x,y,t) = 2*p_0\sum_{n=1}^{\infty}\frac{\sin \alpha_n}{\alpha_n-\sin \alpha_n \cos \alpha_n}\left(\cos\frac{\alpha_n x}{a}-\cos \alpha_n\right)\exp\left(\frac{-\alpha_n^2 c t}{a^2}\right),
\end{equation}
where $p_0 = \frac{1}{3a}B(1+\nu_u)F$, being $B$ the Skempton's coefficient that for our problem is $B=1$ and $\nu_u = \frac{3\nu+B(1-2\nu)}{3-B(1-2\nu)}$ the undrained Poisson's ratio, $c$ is the consolidation coefficient given by $c = K (\lambda + 2\mu)$, and $\alpha_n$ are the positive roots of the nonlinear equation
$$\displaystyle \tan \alpha_n = \frac{1-\nu}{\nu_u-\nu}\alpha_n.$$
As can be observed in (\ref{analytical_p}), also the pressure is independent of the vertical direction. In fact, Coussy (see~\cite{coussy1995}) shows that the normalized pressure is the solution of the following equation
\begin{equation}\label{eq_only_pressure}
\displaystyle \frac{\partial \hat{p}}{\partial \hat{t}} - \frac{\partial^2 \hat{p}}{\partial \hat{x}^2} = 2 \sum_{n=1}^{\infty} \frac{\alpha_n^2 \sin \alpha_n \cos \alpha_n}{\alpha_n-\sin \alpha_n \cos \alpha_n}\exp(-\alpha_n^2 \hat{t}).
\end{equation}
Note that the right-hand side is constant in space and it can become
large at the beginning of the process.

For the finite element solution, the symmetry of the problem allows us to
choose only a quarter of the physical domain as a computational
domain, as shown in Figure~\ref{Mandel_Problem}. Moreover, the rigid
plate condition is enforced by adding constrained equations such that
vertical displacements on the top are equal to an unknown constant
value. The triangulation of the computational domain is obtained from
a uniform rectangular grid $n_x \times n_y$ by splitting each
element in half.  The dimension of the porous slab is specified by
$a = b = 1,$ and the material properties are given by $K = 10^{-6}$,
$E=10^4$, $\nu = 0$, and therefore $\nu_u = 0.5$. The Lam\`{e}
coefficients are computed in terms of the Young modulus and the
Poisson ratio as follows,
\begin{equation*}
\lambda = \displaystyle\frac{E\nu}{(1-2\nu)(1+\nu)},\quad \mu = \displaystyle\frac{E}{2(1+\nu)}.
\end{equation*}
Finally, the applied force has a magnitude of $F = 1 \,M\,Pa\,m$.

\noindent
The first test with Mandel's problem will illustrate the
need of stabilizing the P1-P1 discretization, as well as the MINI
element discretization,
in order to remove the spurious oscillations in the pressure field.
We choose a final time $T= 10^{-4}$
for the computations with only one time-step, and a spatial grid with
$n_x = n_y =32$. Since the pressure unknown is independent of the
vertical coordinate, we will present the results on a representative
horizontal line.
 \begin{figure}[htb]
 \begin{center}
 \begin{tabular}{cc}
 \includegraphics*[width = 0.4\textwidth]{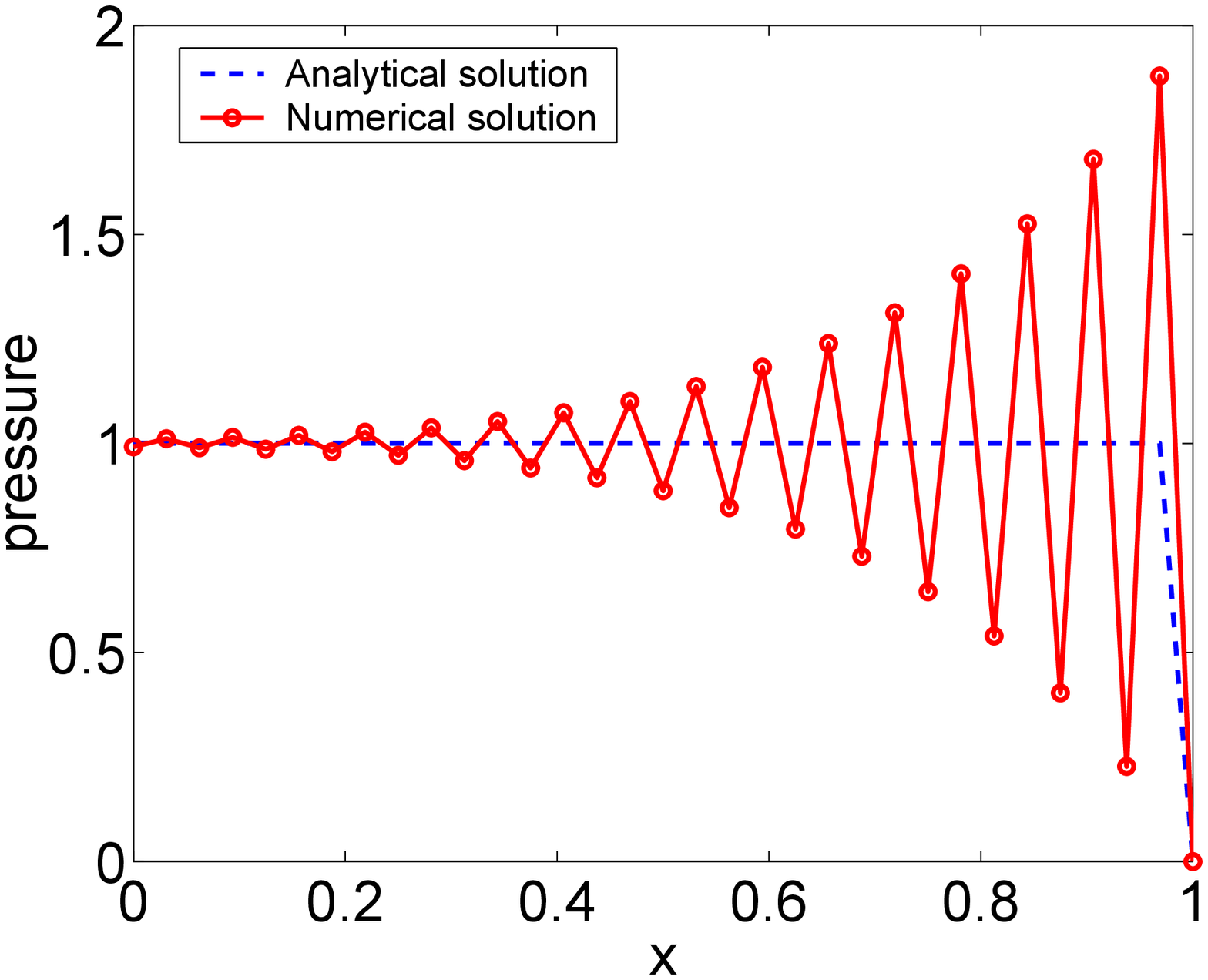}
 &
 \includegraphics*[width = 0.4\textwidth]{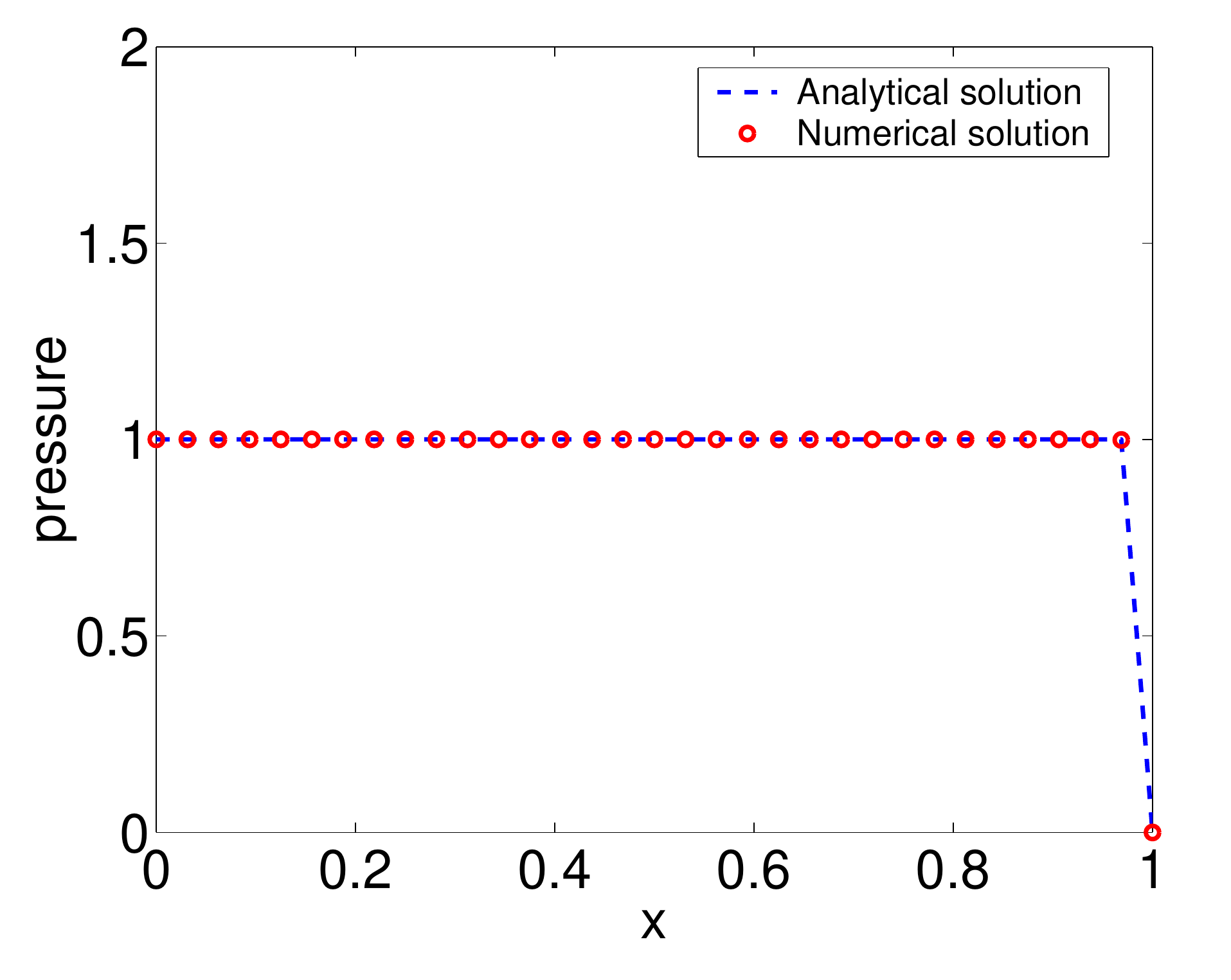}
 \\
 (a)
 &
 (b)
\end{tabular}
\caption{Numerical solution by P1-P1 of the pressure for Mandel's problem (a) without stabilization term and (b) with stabilization term.}
\label{p1_p1_Mandel}
\end{center}
\end{figure}
In Figures~\ref{p1_p1_Mandel} (a) and~\ref{p1_p1_Mandel} (b), we show
the numerical solution for the pressure (plotted in circular symbols)
obtained by using P1--P1 finite element methods without and with
stabilization, respectively. The numerical solution is plotted against
the analytical solution that is displayed by a dashed line. The same
comparison is shown in Figures~\ref{MINI_Mandel} (a)
and~\ref{MINI_Mandel} (b) for the MINI element scheme. For the latter,
the inf-sup condition is satisfied, but we observe that nonphysical
oscillations appear in the pressure field, albeit smaller than in the
P1--P1 case. By adding in both methods stabilization terms,
oscillation-free solutions are obtained, as seen in
Figures~\ref{p1_p1_Mandel}(b) and~\ref{MINI_Mandel}(b).
 \begin{figure}[htb]
 \begin{center}
 \begin{tabular}{cc}
 \includegraphics*[width = 0.4\textwidth]{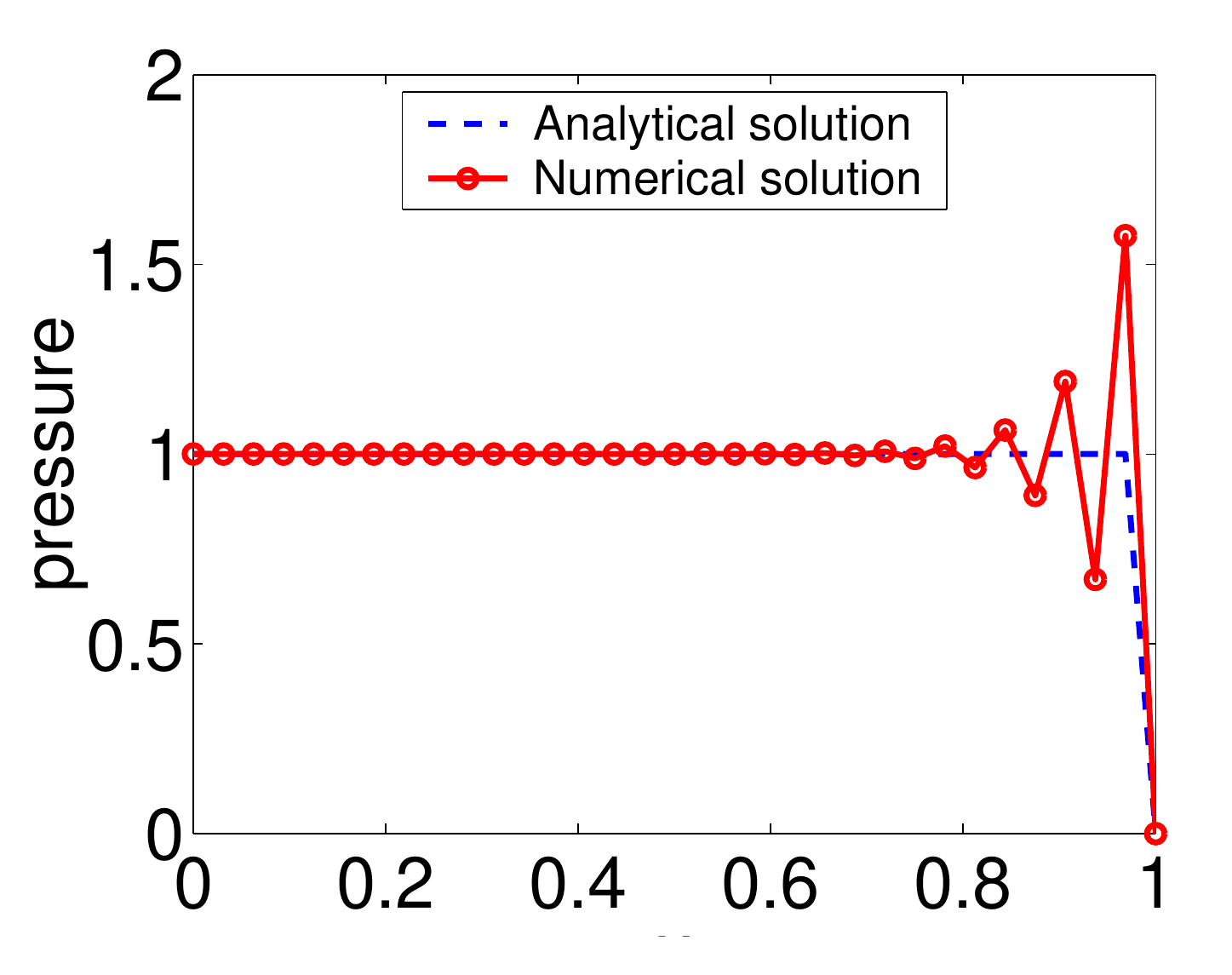}
 &
 \includegraphics*[width = 0.4\textwidth]{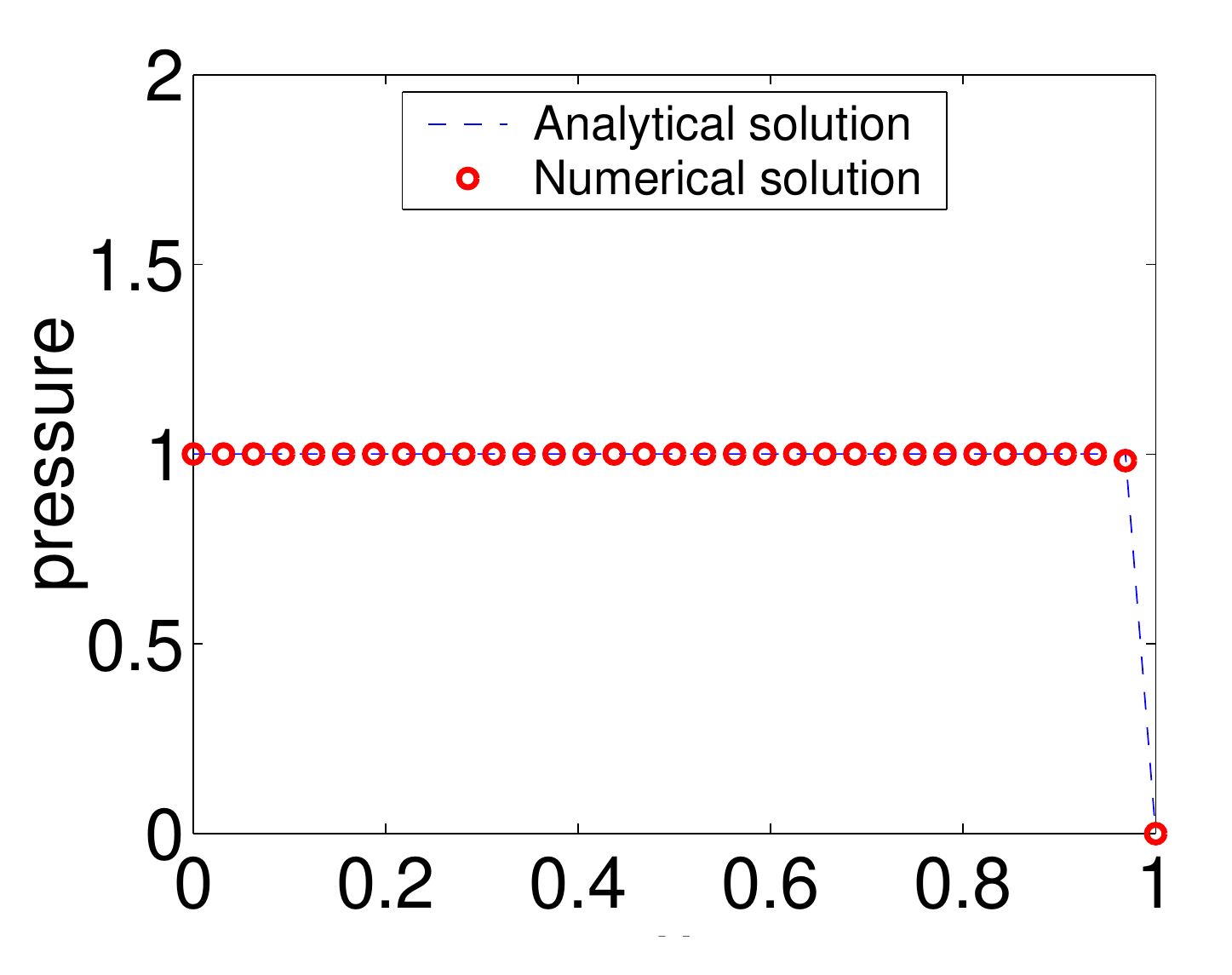}
 \\
 (a)
 &
 (b)
\end{tabular}
\caption{Numerical solution by MINI element of the pressure for Mandel's problem (a) without stabilization term and (b) with stabilization term.}
\label{MINI_Mandel}
\end{center}
\end{figure}
Next, we analyze the behavior of the pressure in different times. For
this purpose, in Figure~\ref{Mandel_times} the solution of the
pressure obtained by stabilized P1--P1 finite elements on a grid with
$n_x = n_y = 32,$ together with the corresponding analytical solution
are shown in different times. We can observe a good agreement between
both solutions for all the cases. A very interesting behavior of the
solution of Mandel's problem is that it can achieve values greater
than one at some time instants. In the literature, this is known as
the Mandel-Cryer effect and usually is associated to a lack of
monotonicity. However, it is clear that this phenomenon is due to the
source term that appears in equation~(\ref{eq_only_pressure}), and is
fully in agreement with the maximum principle for the heat equation.
\begin{figure}[htb]
\begin{center}
\includegraphics*[width = 0.5\textwidth]{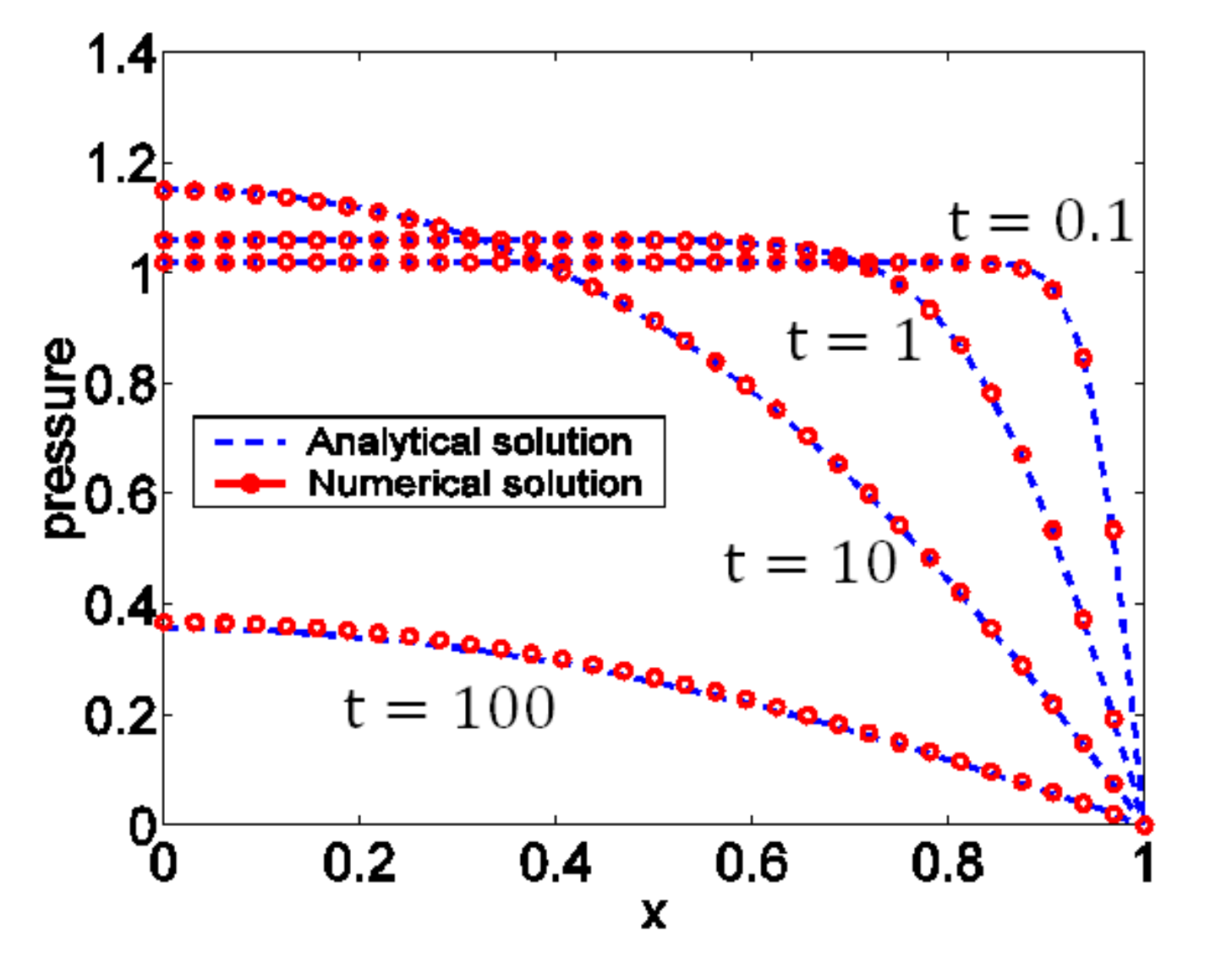}
\caption{Comparison of numerical and analytical solutions of the pore pressure for Mandel's problem at various times.}
\label{Mandel_times}
\end{center}
\end{figure}

Finally, we investigate the convergence properties of the proposed
stabilized schemes by comparing the analytical solution, given
in~\eqref{analytical_p}, with the numerical solution obtained on
progressively refined computational grids with $n_x = n_y$ ranging
from $10$ to $80$ and with time-steps ($\tau = T/n_t$) from $0.5$ to
$0.0625$. In Table~\ref{Mandel_error_p}, for each mesh and a final
time of $T = 1$, we display the error for the pressure in the norm
\begin{equation*}
|| p(t_n) - p_h^n||^2 =  || p(t_n) - p_h^n||^2 + K\tau|| \nabla(p(t_n) - p_h^n)||^2.
\end{equation*}
\begin{table}[!htb]
{\begin{tabular}{ccccc} \hline
$n_x \times n_y \times n_t$ & $10\times10\times 2$ & $20\times20\times 4$ & $40\times40\times 8$ & $80\times80\times 16$ \\ \hline
P1--P1 & 0.0163 & 0.0110 & 0.0058 & 0.0029 \\
MINI & 0.0162 & 0.0110 & 0.0058 & 0.0030 \\ \hline
\end{tabular}}
\caption{Energy norm of the error for the pore pressure by using \\ stabilized P1--P1 and MINI element for different spatial-temporal grids. } \label{Mandel_error_p}
\end{table}
From Table~\ref{Mandel_error_p} we observe first order convergence,
according to the error estimate obtained in
Theorem~\ref{thm:error-mini}.  A very interesting insight rising from
these results is that similar errors for both finite element methods
are obtained. This is due to the fact that very similar stabilization
parameters have to be added to both methods to avoid the nonphysical
oscillations, since the addition of the bubble plays a positive role
but with a very small contribution. This point could be a reason to
support the use of the stabilized P1--P1 scheme against the MINI
element that also has to be stabilized.

\subsection{Barry \& Mercer's problem}\label{test:BarryMercer}

Another well-known benchmark problem on a finite
two-dimensional domain is Barry \& Mercer's model, see~\cite{BarryMercer}. It models the behavior of a rectangular uniform porous
material with a pulsating point
source, drained on all sides, and on which zero tangential displacements are
assumed on the whole boundary. 
\begin{figure}[!htb]
\begin{center}
\includegraphics*[width = 0.55\textwidth]{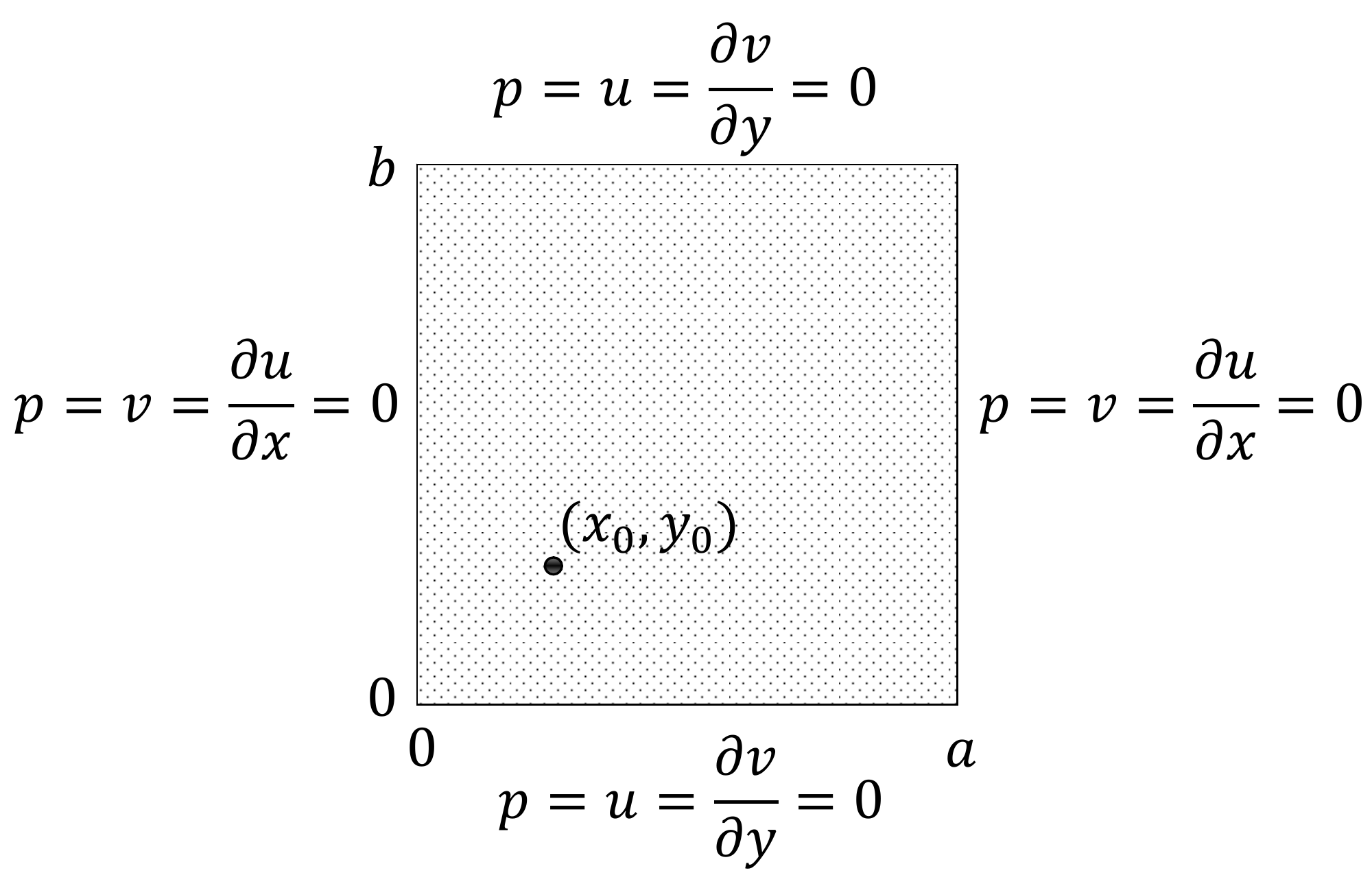}
\caption{Computational domain and boundary conditions for the Barry and Mercer's source problem.}
\label{BarryMercer}
\end{center}
\end{figure}
The point-source corresponds to a sine
wave on the rectangular domain $[0,a]\times[0,b]$ and is given as follows
\begin{equation}\label{BM_source}
f(t) = 2\beta \,\delta_{(x_0,y_0)}\sin(\beta \,t),
\end{equation}
where $\beta = \displaystyle\frac{(\lambda + 2\mu)K}{a\,b}$ and $\delta_{(x_0,y_0)}$ is the Dirac delta at the point $(x_0,y_0)$.
In Figure~\ref{BarryMercer} the computational domain together with the boundary conditions are depicted.
The boundary conditions do not correspond to a realistic physical
situation, but they admit an analytical solution making this model a
suitable test for numerical codes.
Here we use this model to assess the monotone behavior of the
approximations of the pressure.
\begin{figure}[!htb]
\centering
\subfloat[$T = \pi/2$]{
\includegraphics*[width = 0.4\textwidth]{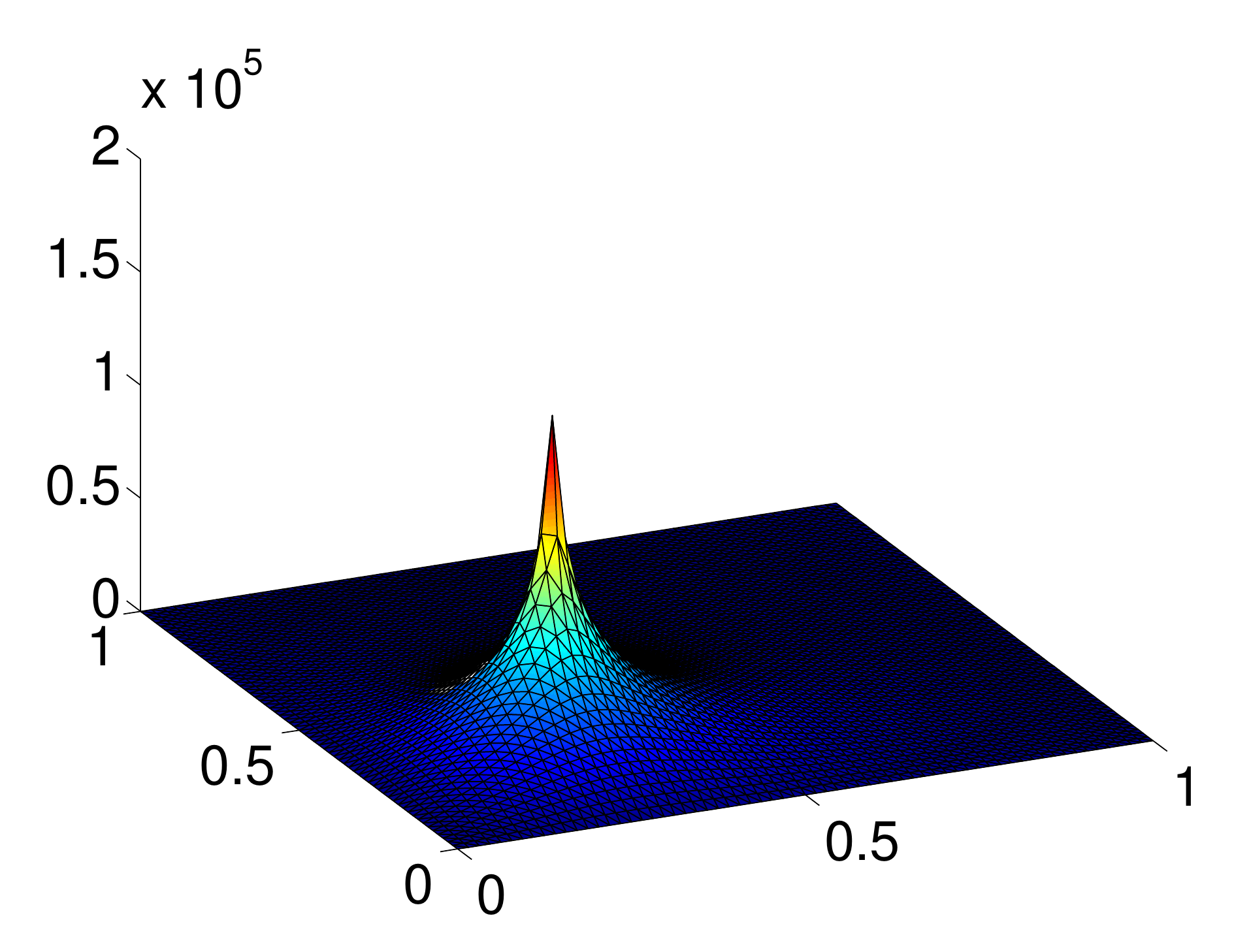}
\hfill
\includegraphics*[width = 0.35\textwidth]{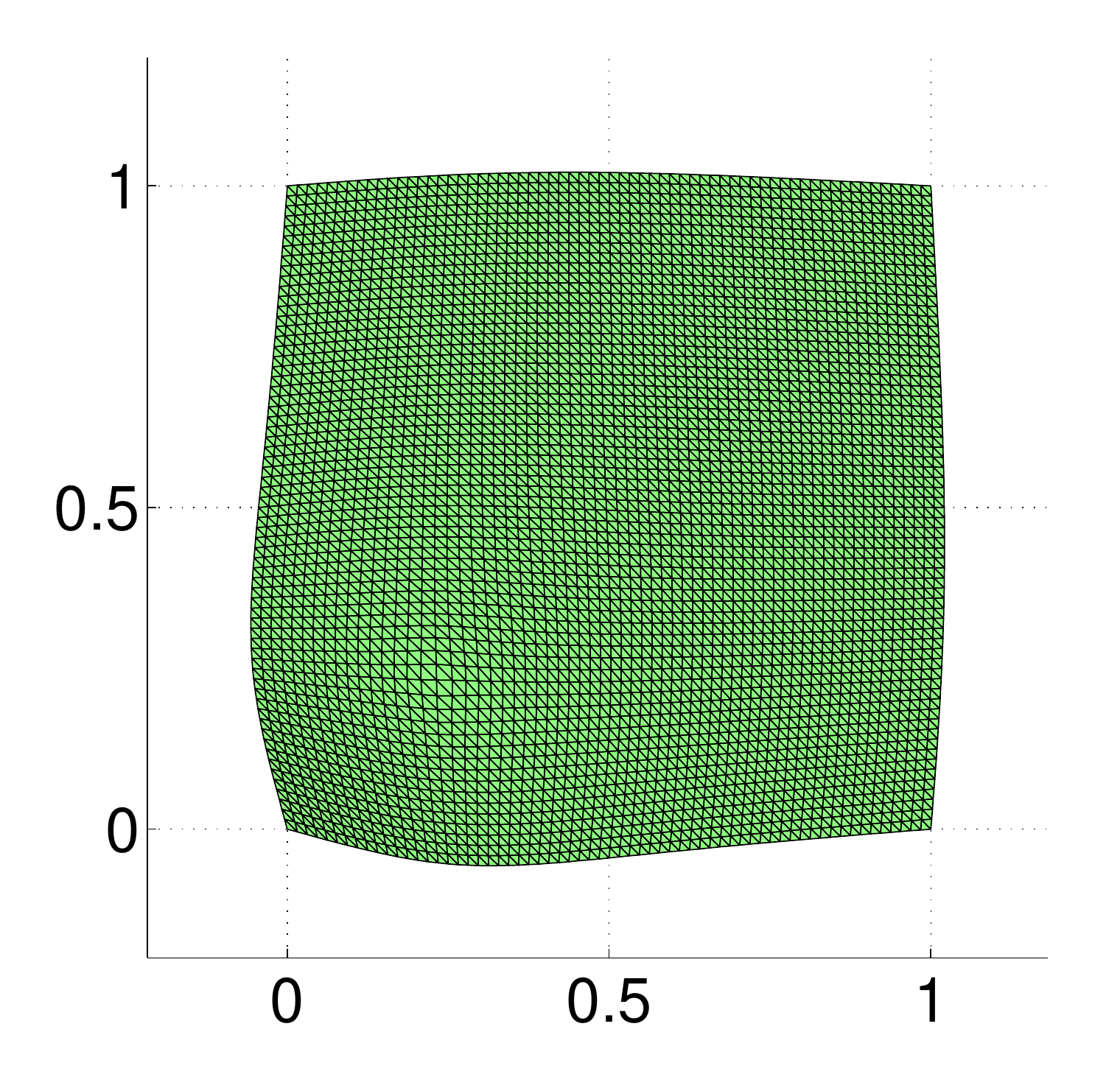}
}
\\
\subfloat[$T=3\pi/2$]
{
\includegraphics*[width = 0.4\textwidth]{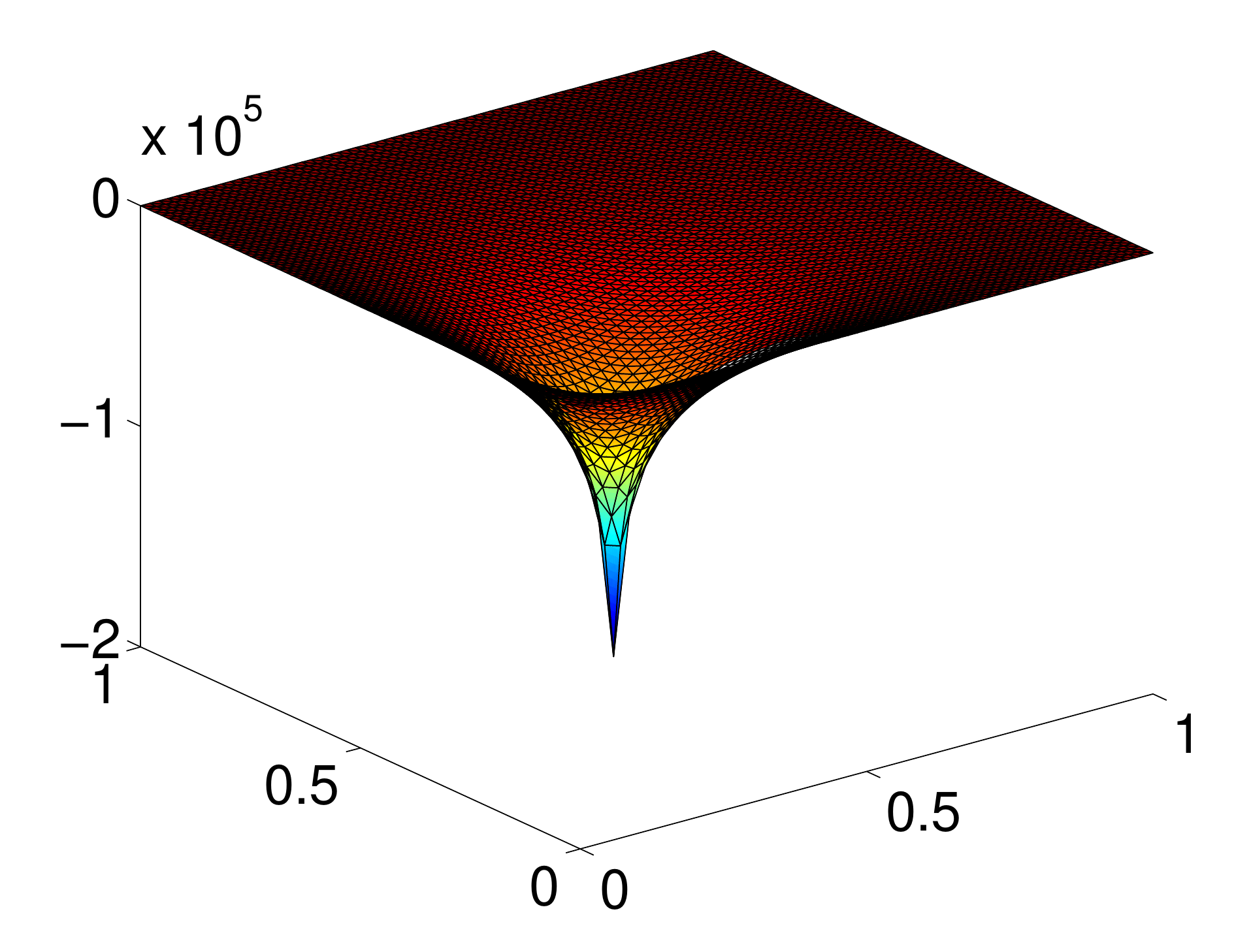}
\hfill
\includegraphics*[width = 0.35\textwidth]{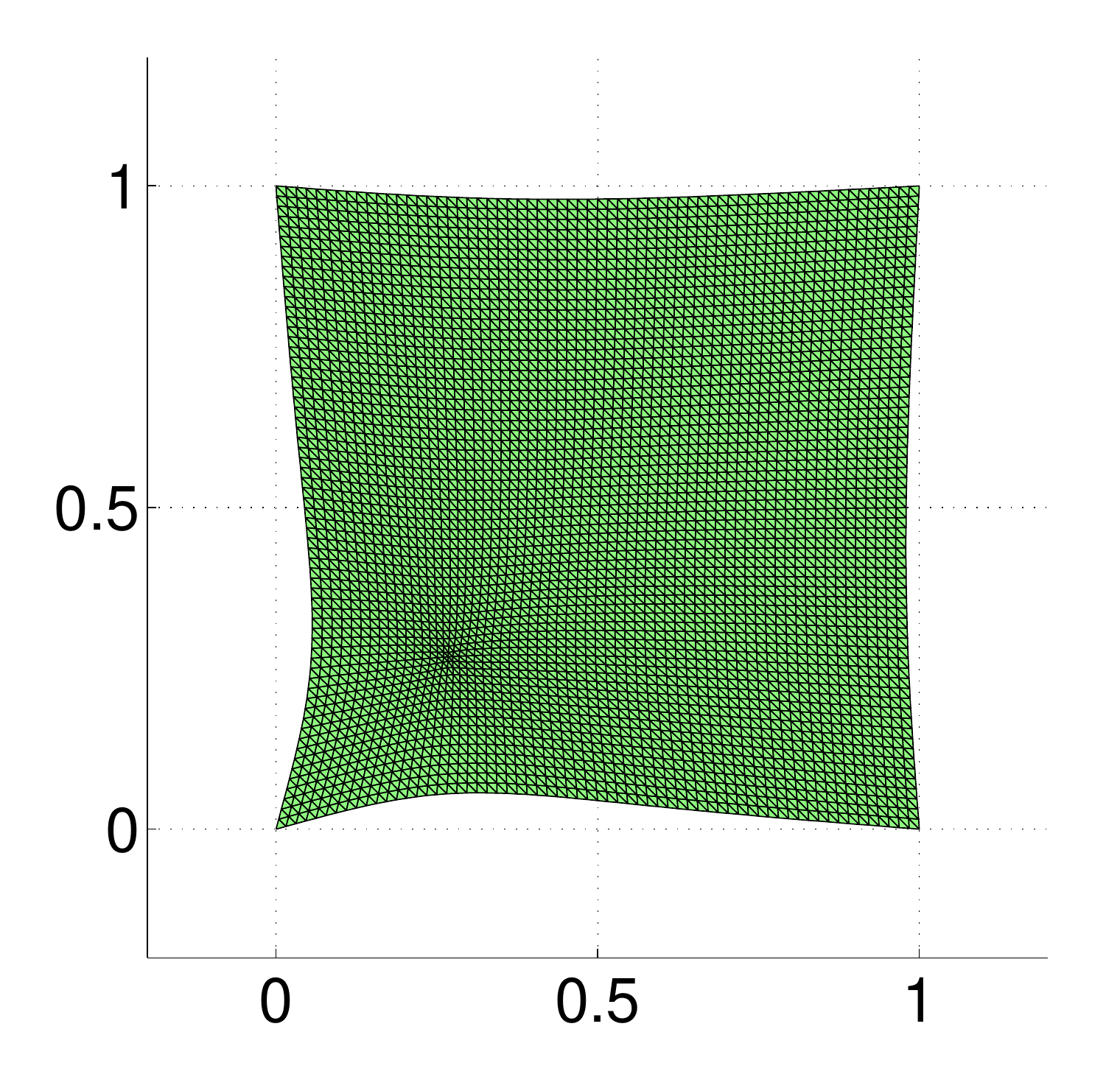}
}
\caption{Numerical solution for the pressure by P1-P1 and deformation of the grid after applying the pulsating pressure point source, for two different values of $T$.}\label{BM_P1_P1}
\end{figure}

We consider the rectangular domain $(0,1)\times(0,1)$, and
the following values of the material parameters are considered
$E = 10^5$, $\nu = 0.1$ and $K = 10^{-2}$.  The source is positioned
at the point $(1/4,1/4)$ and a right triangular grid with
$n_x = n_y = 64$ is used for the simulations. The solution for the
pressure produced by the stabilized P1-P1 scheme is plotted in
Figure~\ref{BM_P1_P1} for two different ``normalized times''
$\hat{t} = \beta \, t$ of values $\hat{t} = \pi/2$ and
$\hat{t} = 3\pi/2$. Also we display the
deformation of the considered triangular grid, according to the
results obtained for the displacements. We can observe that depending
on the sign of the source term (positive for $\hat{t} = \pi/2$ and
negative for $\hat{t} = 3\pi/2$) the resultant displacements cause an
expansion or a contraction of the medium.

The analytical solution of this problem is given by an infinite
series, and can be found in~\cite{BarryMercer}. It has been
observed that solutions displayed in Figure~\ref{BM_P1_P1} resemble
the exact solution very precisely.

Fluid pressure oscillations for the Barry and Mercer's problem can be
demonstrated by considering the standard schemes given by a P1-P1 or
MINI element discretizations. In order to see this characteristic
non-physical oscillatory behavior, a small permeability and/or a short
time intervals are considered. Therefore, in the previous test,
we have changed the value of $K$ to $10^{-6}$ and $T$ to
$10^{-4}$. For these parameters, in Figure~\ref{BM_oscillations} we
show the numerical solutions obtained for the pressure field, by using
P1-P1 scheme (on the top) and the MINI element (on the bottom). We can
observe that if no stabilization term is added to any of the discrete
schemes (left pictures), then non-physical oscillations appear in the
surroundings of the source-point. However, by adding the proposed
artificial stabilizations, we can see (right pictures) that these
oscillations are completely eliminated.
\begin{figure}[htb]
\begin{center}
\begin{tabular}{cc}
\includegraphics*[width = 0.35\textwidth]{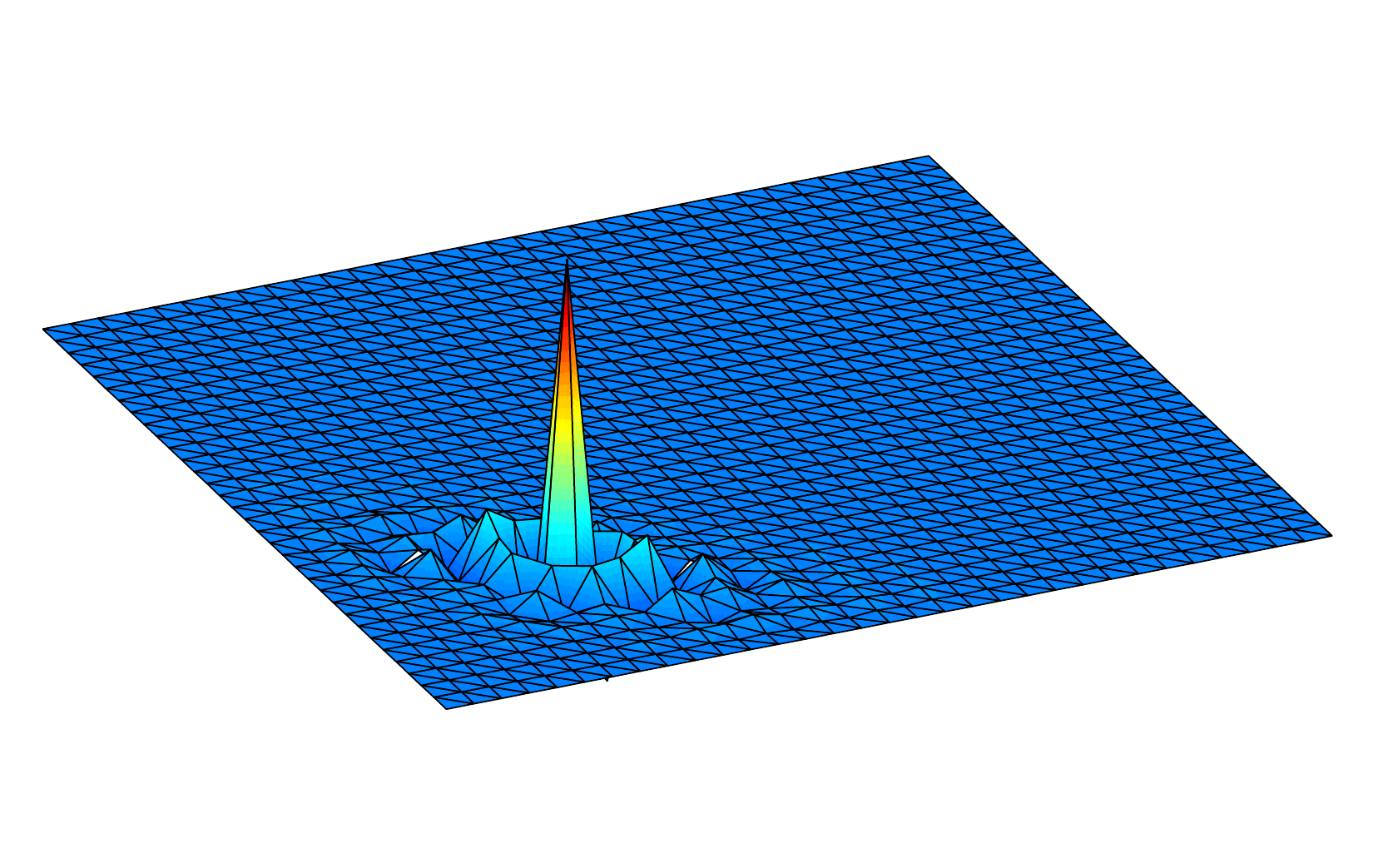}
&
\includegraphics*[width = 0.35\textwidth]{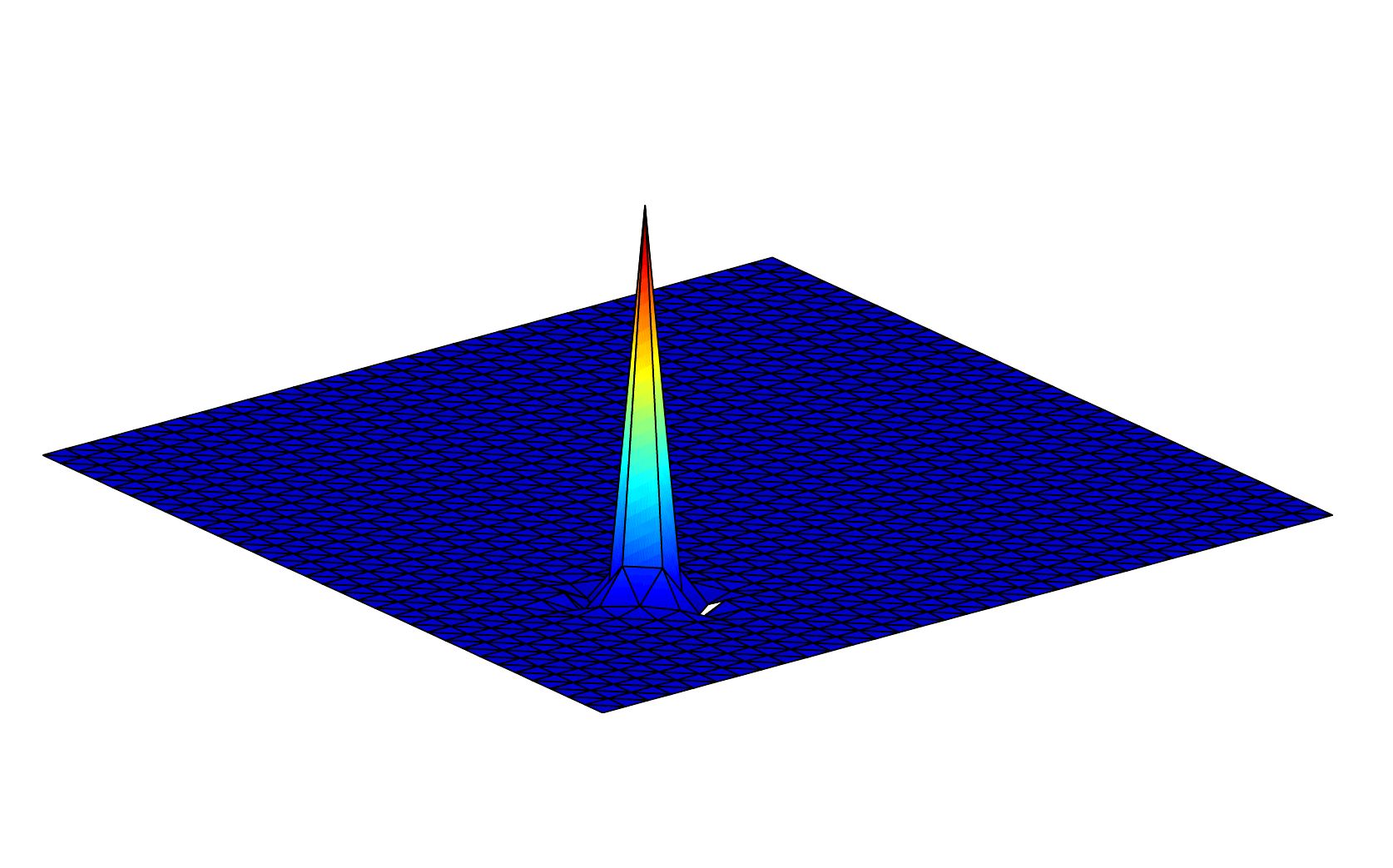}
\\
\includegraphics*[width = 0.35\textwidth]{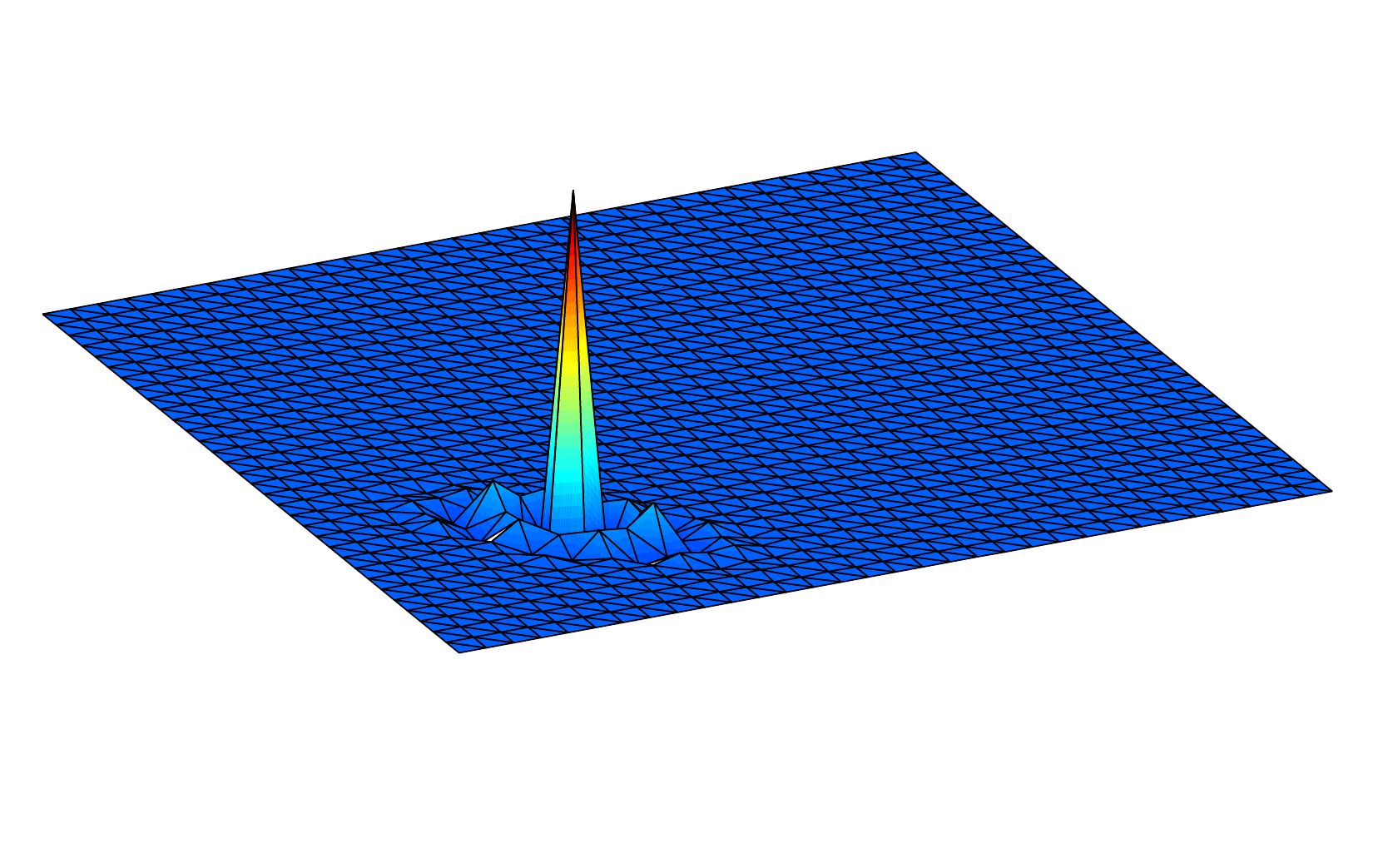}
&
\includegraphics*[width = 0.35\textwidth]{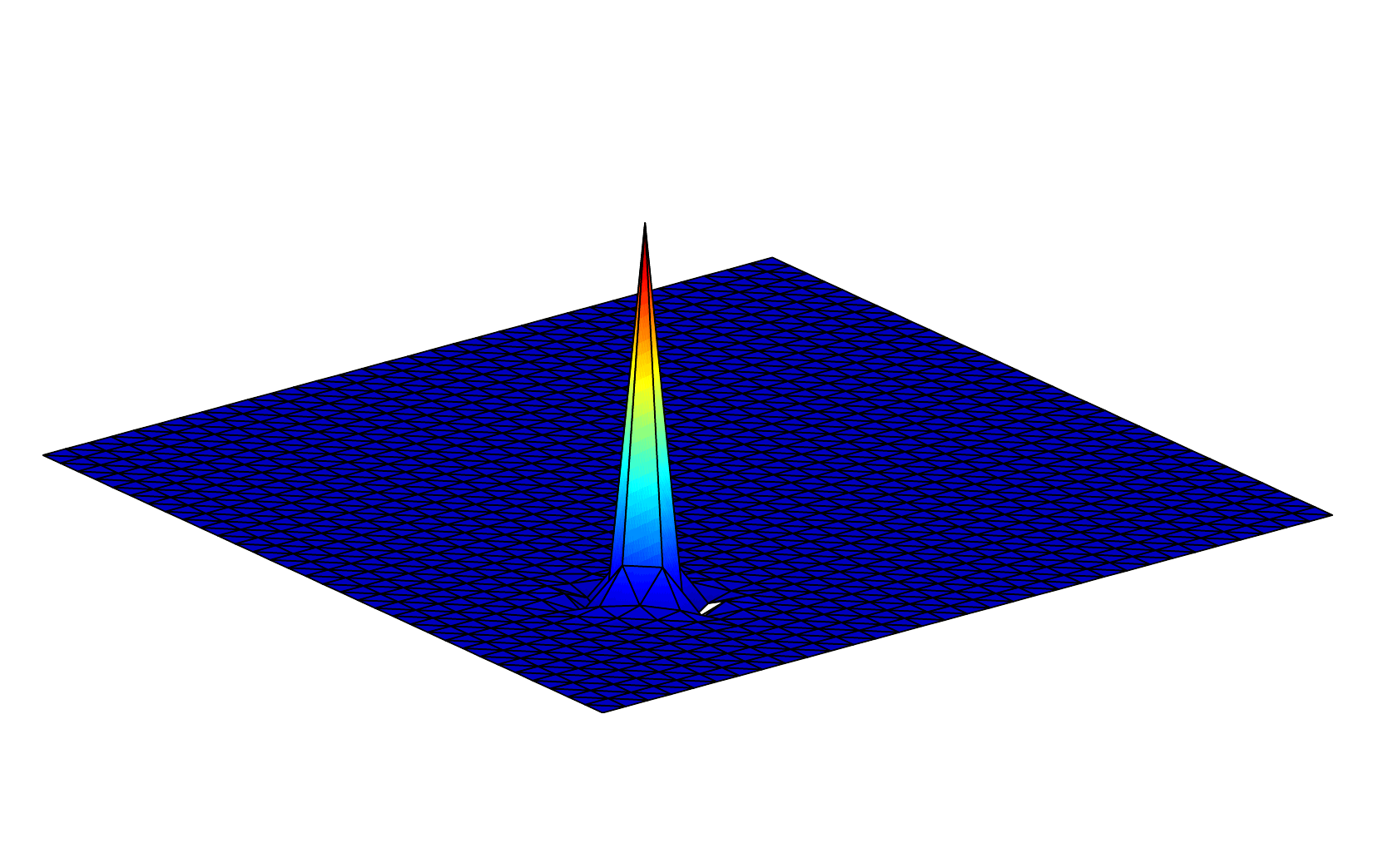}
\end{tabular}
\caption{Numerical solution for the pressure field by P1-P1 (top) and MINI element (bottom), without and with stabilization term, at a final time of $10^{-4}$ and a permeability of $K = 10^{-6}$.}
\label{BM_oscillations}
\end{center}
\end{figure}

\section{Conclusions}\label{sec:conclusions}
In this paper we have analyzed the convergence and the monotonicity
properties of low order discretizations of the Biot's consolidation
model in poromechanics. While the convergence results are complete in
some sense, there are still several open theoretical questions
regarding the monotonicity of the resulting discretizations. Clearly,
our numerical results show that choosing the stabilization parameters
correctly lead to oscillation-free solutions, but justifying this
rigorously is difficult and a topic of ongoing research. We have to
say though that as a rule of thumb, one can choose stabilizations that
are optimal in 1D, and, the resulting approximations in higher spatial
dimensions will be oscillation-free.

\section*{Acknowledgements}
The work of Francisco J. Gaspar and Carmen Rodrigo is supported in
part by the Spanish project FEDER /MCYT MTM2013-40842-P and the DGA
(Grupo consolidado PDIE). The research of Ludmil Zikatanov is supported in
part by NSF DMS-1217142 and NSF DMS-1418843. Ludmil Zikatanov gratefully
acknowledges the support for this work from the Institute of Mathematics
and Applications at University of Zaragoza and Campus Iberus, Spain.



\appendix

\section{Local elimination of bubbles}\label{sect:local-Schur}

In this appendix we compute the contribution of bubble stabilization
in the MINI element.  We show that $G_b^T A_{b}^{-1} G_b$ is
spectrally equivalent to the stiffness matrix corresponding to the
discretization of the Laplace with continuous piece-wise linear finite
elements.

To begin, we fix $T\in \mathcal{T}_h$ and we prove several simple
identities. When the dependence on $T$ neds to be emphasized we
indicate this by indexing the corresponding quantities with $T$, but
most of the time, this is not needed and we set
\begin{eqnarray*}
&&\lambda_k=\lambda_{k,T},\quad k=1,\ldots, (d+1), \\
&&\alpha=\alpha_T, \quad\mbox{and}\quad
\varphi=\varphi_{b,T} = \alpha \lambda_1\ldots\lambda_{d+1}.
\end{eqnarray*}
Here $\lambda_{k,T}(x)$ are the standard barycentric coordinates on
$T$ and $\alpha_T$ is a constant chosen so that $\varphi_{b,T}$ has a
value $1$ at the barycenter of $T$. To integrate polynomials over a
$d$-dimensional simplex $T$ we use the well known formula for
integrating powers of the barycentric coordinates~(see~\cite{1978CiarletP-aa}):
\begin{equation}\label{int-on-T}
\int_T \lambda_1^{\beta_1}\ldots\lambda_{d+1}^{\beta_{d+1}}\;dx =
|T|\frac{\beta_{1}!\ldots\beta_{d+1}! d! }{(\beta_{1}+\ldots+\beta_{d+1}+d)!}.
\end{equation}
Further, we introduce the matrix $\Lambda \in \mathbb{R}^{d\times(d+1)}$ whose columns are the
appropriately scaled gradients of
$\lambda_k$, $k=1,\ldots, (d+1)$ i.e.
\[
\Lambda =
\sqrt{|T|}
 (\nabla \lambda_1 , \ldots, \nabla \lambda_{d+1}) =
\sqrt{|T|}
\begin{pmatrix}
\nabla
\lambda_1\cdot\bm e_1,& \ldots &\nabla \lambda_{d+1}\cdot\bm e_{1}\\
\cdots& \cdots &\cdots\\
\nabla
\lambda_1\cdot\bm e_d,& \ldots &\nabla \lambda_{d+1}\cdot\bm e_{d}.
\end{pmatrix}
\]
We note that $\Lambda^T \Lambda$ equals the local stiffness matrix for the
Laplace equation on $T$, namely
\[
(L_T)_{jk}=(\Lambda^T \Lambda)_{jk} = \int_T \nabla \lambda_k\cdot \nabla \lambda_j.
\]
Note that  we have
\begin{eqnarray*}
\nabla \varphi
&=& \alpha \sum_{k=1}^{d+1}\chi_k\nabla \lambda_k,
\quad \chi_{k} = \prod_{j=1;j\neq k}^{d+1} \lambda_j,\quad k=1,\ldots,(d+1).
\end{eqnarray*}
With this notation in hand, we now prove two auxiliary identities.
\begin{lemma}\label{lemma:ids}
 For $\nabla \varphi$ we have
\begin{itemize}
\item[(i)] $\displaystyle
\int_T(\nabla \varphi\nabla \varphi^T) = \alpha^2\eta_d\Lambda\Lambda^T$,
\item[(ii)] $\int_T |\nabla \varphi|^2  = \alpha^2\eta_d\operatorname{tr}(L_T)$.
\end{itemize}
Here $\displaystyle \eta_d =  \frac{2^{d-1}d!}{(3d)!}$.
\end{lemma}
\begin{proof}
To prove (i) we observe that
$\nabla \varphi = \frac{\alpha}{\sqrt{|T|}}
\Lambda\bm{\chi}$,  $\bm{\chi}=(\chi_1,\ldots,\chi_{d+1})^T$.
Since $\Lambda$ is a constant matrix
(independent of $x$) we have that
\[
\int_T\nabla \varphi\nabla \varphi^T =
\frac{\alpha^2}{|T|}\Lambda \left(\int_T\bm{\chi} \bm{\chi^T}\right)\Lambda^T.
\]
The formula given in~\eqref{int-on-T} gives that
\begin{equation*}
\left(\int_T\bm \chi \bm \chi^T\right)_{jk} =   \int_T\chi_j\chi_k
= \eta_d|T| \left\{
\begin{matrix}
2,\quad j=k,\\
1,\quad j\neq k.
\end{matrix}
\right.
\end{equation*}
Hence, $\int_T\bm \chi \bm \chi^T = \eta_d|T| (I +\bm{1}\bm{1}^T)$,
where $\bm 1 = (\underbrace{1,\ldots,1}_{d+1})^T$.  As
$\sum_{k=1}^{d+1} \lambda_k = 1$, we have that,
$\sum_{k=1}^{d+1} \nabla \lambda_k = 0$, or, equivalently,
$\Lambda \bm 1 = 0$.  These identities show that
\[
\int_T\nabla \varphi\nabla \varphi^T =
\alpha^2\eta_d\Lambda (I + \bm{1}\bm{1}^T)\Lambda^T =
\alpha^2\eta_d\Lambda \Lambda^T,
\]
and the proof of (i) is complete.

To show that (ii) holds we observe that
$\int_T |\nabla\varphi|^2=\int_T \operatorname{tr}(\nabla \varphi \nabla\varphi^T)$,
and we can use (i) to compute that
\begin{eqnarray*}
\int_T|\nabla \varphi|^2=
\int_T \operatorname{tr}(\nabla \varphi \nabla\varphi^T) =
 \operatorname{tr}\left(\int_T\nabla \varphi \nabla\varphi^T\right)=
\alpha^2\eta_d\operatorname{tr}(L_T).
\end{eqnarray*}
In the last step we used that
$\operatorname{tr}(\Lambda\Lambda^T)=\operatorname{tr}(\Lambda^T\Lambda)$.
\end{proof}
Using this lemma we now calculate the local stiffness matrices
for $A_{b,T}$ and $G_{b,T}$.
\begin{lemma}\label{lemma:matrices} For $A_{b,T}$ and $G_{b,T}$ we
  have
\begin{enumerate}
\item[(i)] $A_{b,T} =
\alpha^2\eta_d(\mu\operatorname{tr}(L_T)I + (\lambda+\mu)\Lambda\Lambda^T)$.
\item[(ii)] $\displaystyle G_{b,T} = \frac{\alpha\sqrt{|T|}d!}{(2d+1)!}\Lambda$.
\end{enumerate}
\end{lemma}
\begin{proof}
To show the identity for $A_{b,T}$ recall that
\begin{eqnarray*}
(A_{b,T})_{jk}&=&a((\varphi\bm e_k),(\varphi\bm{e}_j))
\\
& = & 2\mu \int_T
\varepsilon((\varphi\bm e_k)):\varepsilon((\varphi\bm e_j))
+ \lambda \int_T
\operatorname{div}(\varphi\bm e_k)\operatorname{div}(\varphi\bm e_j).
\end{eqnarray*}
A straightforward calculation shows that
\[
\varepsilon((\varphi\bm e_k))=
\frac12(\nabla\varphi\bm{e}_k^T +\bm e_k(\nabla\varphi)^T),
\]
and hence
\[
\int_T\varepsilon((\varphi\bm e_k)):\varepsilon((\varphi\bm e_j))=
\frac{\delta_{jk}}{2}\int_T|\nabla \varphi|^2 +
\frac12 \left(\int_T\nabla \varphi \nabla \varphi^T\right)_{jk}.
\]
We also have
\[
\int_T
\operatorname{div} (\varphi\bm e_k)
\operatorname{div} (\varphi\bm e_j)  = \left(\int_T\nabla \varphi \nabla \varphi^T\right)_{jk}.
\]
Finally, using Lemma~\ref{lemma:ids} for $A_{b,T}$ we get
\begin{equation}\label{final-matrix}
A_{b,T} =
\alpha^2\eta_d(\mu\operatorname{tr}(L_T)I + (\lambda+\mu)\Lambda\Lambda^T).
\end{equation}

To show (ii), we have, for  $k=1,\ldots, (d+1)$ and $j=1,\ldots, d$,
\begin{eqnarray*}
(G_{b,T})_{jk} &=&
 \int_T (\varphi \bm{e}_j\cdot \nabla \lambda_{k})
=(\nabla\lambda_k\cdot \bm{e_j}) \int_T \varphi
=\frac{\Lambda_{jk}}{\sqrt{|T|}}
\int_T \varphi.
\end{eqnarray*}
Computing $\int_T \varphi_T$ concludes the proof of (ii).
\end{proof}

From this, for the local Schur complement $S_{b,T}$ we get
\begin{eqnarray*}
S_{b,T} & = &  G_{b,T}^T A_{b,T}^{-1} G_{b,T} =
c_d|T|\Lambda^T( \mu\operatorname{tr}(L_T) I +
(\lambda+\mu)\Lambda\Lambda^T)^{-1} \Lambda\\
& = &
\sigma\Lambda^T( I + \beta\Lambda\Lambda^T)^{-1} \Lambda\\
\end{eqnarray*}
where
\[
\sigma =\frac{c_d|T|}{\mu\operatorname{tr}(L_T)},\quad
\beta =
\frac{\lambda+\mu}{\operatorname{\mu tr}(L_T)},
\quad\mbox{and}\quad
c_d =\frac{d!(3d)!}{2^{d-1}((2d+1)!)^2}.
\]
We apply the Sherman-Morrison-Woodbury to obtain that
\begin{equation*}
(I+\beta \Lambda^T \Lambda)^{-1} = I -
\beta \Lambda^T( I + \beta \Lambda\Lambda^T)^{-1}\Lambda
\end{equation*}
This then shows that
\begin{equation*}
\sigma\Lambda^T( I + \beta \Lambda\Lambda^T)^{-1}\Lambda =
\frac{\sigma}{\beta}\left[I - (I+\beta L_T)^{-1} \right].
\end{equation*}
Observing that
\[
I - (I+\beta L_T)^{-1} = I - ((I + \beta L_T) - \beta L_T)(I+\beta L_T)^{-1} =
\beta L_T (I+\beta L_T)^{-1},
\]
we obtain that
\begin{equation}\label{eq:sb}
S_{b,T}  = \sigma L_T (I+\beta L_T)^{-1}.
\end{equation}

We next show that $S_{b,T}$ behaves as a scaling of the local stiffness
matrix corresponding to the Laplace operator.
\begin{lemma}\label{lemma:schur-local} We have the following spectral equivalence result
\[
S_{b,T}  \eqsim h_T^2 L_T,
\]
with constants independent  of the mesh size.
\end{lemma}
\begin{proof}
  Let $\mu_1\ge \mu_2\ge\ldots\ge \mu_{d}> \mu_{d+1}=0$ be the
  eigenvalues of the scaled matrix
  $\widetilde L_T=\frac{1}{\operatorname{tr}(L_T)}L_T$, and
  $\psi_1,\ldots,\psi_{d+1}$ be the corresponding eigenvectors. Note
  that because of the scaling, we have that $\mu_k$ can be bounded
  independently of the mesh size $h_T$.  We set
  $\widetilde{\beta}=\beta \operatorname{tr}(L_T)
  =\frac{(\lambda+\mu)}{\mu}$
and we obtain the following representation of $\widetilde{L}_T$:
\[
\widetilde{L}_T = \sum_{j=1}^d \mu_j \psi_j\psi_j^T,
\quad
(I+\beta L_T)^{-1} =
(I+\widetilde \beta \widetilde{L}_T)^{-1} = \sum_{j=1}^d \frac{1}{1+\widetilde{\beta}\mu_j} \psi_j\psi_j^T.
\]
Obviously, similar relation holds for $S_{b,T}$ because the
eigenvectors of $L_T$ ($\widetilde L_T$) and $S_{b,T}$ are the same
(this is easily seen from~\eqref{eq:sb}). We then have that for any
$x\in \mathbb{R}^d$ the following inequalities hold
\begin{eqnarray*}
  \langle S_{b,T} x,x\rangle_{\ell^2}
&= &\frac{c_d|T|}{\mu} \sum_{j=1}^{d}
                \frac{\mu_j}{1+\widetilde{\beta}\mu_j}\langle\psi_j,x\rangle_{\ell^2}^2.
\end{eqnarray*}
Hence,
\begin{eqnarray*}
 \frac{c_d|T|}{\mu(1+\widetilde{\beta}\mu_1)} \langle \widetilde{L}_T x,x\rangle_{\ell^2}
\le   \langle S_{b,T} x,x\rangle_{\ell^2}
\le \frac{c_d|T|}{\mu(1+\widetilde{\beta}\mu_d)} \langle \widetilde{L}_T x,x\rangle_{\ell^2}.
\end{eqnarray*}
We write everything in terms of $L_T$, and from the obvious relations
$\operatorname{tr}(L_T)\approx h_T^{d-2}$, $|T|\approx h_T^d$ we
conclude the proof of the lemma.
\end{proof}
\begin{remark}
As is easily seen, for $d=1$ we have that both bounds coincide, and in
fact, we have that
\begin{equation}\label{eq:1d-schur}
  S_{b,T}  =  \frac{h_T}{6\mu(1+\widetilde{\beta})}\widetilde{L}_T =
 \frac{h^2_T}{12(2\mu+\lambda)}L_T,
\end{equation}
where we have used that $\mu_1 = 1$ and $|T| = h_T$ in 1d.
\end{remark}



\section*{References}
\bibliographystyle{elsarticle-num}
\bibliography{bib_PoroE}

\begin{thebibliography}{10}
\expandafter\ifx\csname url\endcsname\relax
  \def\url#1{\texttt{#1}}\fi
\expandafter\ifx\csname urlprefix\endcsname\relax\def\urlprefix{URL }\fi
\expandafter\ifx\csname href\endcsname\relax
  \def\href#1#2{#2} \def\path#1{#1}\fi

\bibitem{terzaghi}
K.~Terzaghi, Theoretical Soil Mechanics, Wiley: New York, 1943.

\bibitem{biot1}
M.~A. Biot, General theory of three‐dimensional consolidation, Journal of
  Applied Physics 12~(2) (1941) 155--164.

\bibitem{biot2}
M.~A. Biot, Theory of elasticity and consolidation for a porous anisotropic
  solid, Journal of Applied Physics 26~(2) (1955) 182--185.

\bibitem{showalter}
R.~Showalter, Diffusion in poro-elastic media, Journal of Mathematical Analysis
  and Applications 251~(1) (2000) 310 -- 340.

\bibitem{zenisek}
A.~{\v{Z}}en{\'{\i}}{\v{s}}ek, The existence and uniqueness theorem in {B}iot's
  consolidation theory, Apl. Mat. 29~(3) (1984) 194--211.

\bibitem{DDL1997}
R.~Z. Dautov, M.~I. Drobotenko, A.~D. Lyashko, Study on well-posedness of the
  generalized solution of the problem of filtration consolidation, Differents.
  Uravnenia 33 (1997) 515 -- 521.

\bibitem{Roose2003204}
T.~Roose, P.~A. Netti, L.~L. Munn, Y.~Boucher, R.~K. Jain, Solid stress
  generated by spheroid growth estimated using a linear poroelasticity model,
  Microvascular Research 66~(3) (2003) 204 -- 212.

\bibitem{Swan2003}
C.~Swan, R.~Lakers, R.~Brand, K.~Stewart, Micromechanically based poroelastic
  modeling of fluid flow in haversian bone, J. Biomech. Eng. 125 (2003) 25 --
  37.

\bibitem{Datta2010}
A.~D. Amit~Halder, A.~K. Datta, Modeling transport in porous media wiith phase
  change: Applications to food processing, J. Heat Transfer 133 (2010)
  031010--1 -- 031010--13.

\bibitem{Coussy2004}
O.~Coussy, Poromechanics, John Wiley \& Sons, Ltd, 2004.

\bibitem{BarryMercer}
S.~I. {Barry}, G.~N. {Mercer}, {Exact Solutions for Two-Dimensional
  Time-Dependent Flow and Deformation Within a Poroelastic Medium}, Journal of
  Applied Mechanics 66 (1999) 536.
\newblock \href {http://dx.doi.org/10.1115/1.2791080}
  {\path{doi:10.1115/1.2791080}}.

\bibitem{LewisSchrefler}
R.~Lewis, B.~Schrefler, The Finite Element Method in the Static and Dynamic
  Deformation and Consolidation of Porous Media, Wiley: New York, 1998.

\bibitem{LewisSchrefler78}
R.~W. Lewis, B.~Schrefler, A fully coupled consolidation model of the
  subsidence of venice, Water Resources Research 14~(2) (1978) 223--230.

\bibitem{LewisTran89}
R.~W. Lewis, D.~V. Tran, Numerical simulation of secondary consolidation of
  soil: Finite element application, International Journal for Numerical and
  Analytical Methods in Geomechanics 13~(1) (1989) 1--18.

\bibitem{LewisSchreflerSimoni}
R.~W. Lewis, B.~A. Schrefler, L.~Simoni, Coupling versus uncoupling in soil
  consolidation, International Journal for Numerical and Analytical Methods in
  Geomechanics 15~(8) (1991) 533--548.

\bibitem{MastersPaoLewis}
I.~Masters, W.~K.~S. Pao, R.~W. Lewis, Coupling temperature to a
  double-porosity model of deformable porous media, International Journal for
  Numerical Methods in Engineering 49~(3) (2000) 421--438.

\bibitem{CMAM_2002}
F.~J. Gaspar, F.~J. Lisbona, P.~N. Vabishchevich,
  \href{http://dx.doi.org/10.2478/cmam-2002-0008}{Finite difference schemes for
  poro-elastic problems}, Comput. Methods Appl. Math. 2~(2) (2002) 132--142.
\newblock \href {http://dx.doi.org/10.2478/cmam-2002-0008}
  {\path{doi:10.2478/cmam-2002-0008}}.
\newline\urlprefix\url{http://dx.doi.org/10.2478/cmam-2002-0008}

\bibitem{Lazarov2007}
R.~E. Ewing, O.~P. Iliev, R.~D. Lazarov, A.~Naumovich, On convergence of
  certain finite volume difference discretizations for 1d poroelasticity
  interface problems, Numerical Methods for Partial Differential Equations
  23~(3) (2007) 652--671.

\bibitem{PhDNaumovich}
A.~Naumovich, Efficient numerical methods for the {B}iot poroelasticity system
  in multilayered domains, Kaiserslautern, Techn. Univ., Diss., 2007.

\bibitem{Gaspar2003}
F.~J. Gaspar, F.~J. Lisbona, P.~N. Vabishchevich,
  \href{http://dx.doi.org/10.1016/S0168-9274(02)00190-3}{A finite difference
  analysis of {B}iot's consolidation model}, Appl. Numer. Math. 44~(4) (2003)
  487--506.
\newblock \href {http://dx.doi.org/10.1016/S0168-9274(02)00190-3}
  {\path{doi:10.1016/S0168-9274(02)00190-3}}.
\newline\urlprefix\url{http://dx.doi.org/10.1016/S0168-9274(02)00190-3}

\bibitem{Ferronato2010}
M.~Ferronato, N.~Castelletto, G.~Gambolati, A fully coupled 3-d mixed finite
  element model of {B}iot consolidation, J. Comput. Phys. 229~(12) (2010)
  4813--4830.

\bibitem{Langtangen2012}
J.~B. Haga, H.~Osnes, H.~P. Langtangen, On the causes of pressure oscillations
  in low-permeable and low-compressible porous media, International Journal for
  Numerical and Analytical Methods in Geomechanics 36~(12) (2012) 1507--1522.

\bibitem{Favino2013}
M.~Favino, A.~Grillo, R.~Krause,
  \href{http://ascelibrary.org/doi/abs/10.1061/9780784412992.110}{A stability
  condition for the numerical simulation of poroelastic systems}, in:
  C.~Hellmich, B.~Pichler, D.~Adam (Eds.), Poromechanics V: Proceedings of the
  Fifth {B}iot Conference on Poromechanics, 2013, pp. 919--928.
\newline\urlprefix\url{http://ascelibrary.org/doi/abs/10.1061/9780784412992.110}

\bibitem{WheelerPhillip}
P.~Phillips, M.~Wheeler, Overcoming the problem of locking in linear elasticity
  and poroelasticity: an heuristic approach, Computational Geosciences 13~(1)
  (2009) 5--12.

\bibitem{1974BrezziF-aa}
F.~Brezzi, On the existence, uniqueness and approximation of saddle-point
  problems arising from {L}agrangian multipliers, Rev. Fran\c caise Automat.
  Informat. Recherche Op\'erationnelle S\'er. Rouge 8~(R-2) (1974) 129--151.

\bibitem{MuradLoula92}
M.~A. Murad, A.~F.~D. Loula,
  \href{http://dx.doi.org/10.1016/0045-7825(92)90193-N}{Improved accuracy in
  finite element analysis of {B}iot's consolidation problem}, Comput. Methods
  Appl. Mech. Engrg. 95~(3) (1992) 359--382.
\newblock \href {http://dx.doi.org/10.1016/0045-7825(92)90193-N}
  {\path{doi:10.1016/0045-7825(92)90193-N}}.
\newline\urlprefix\url{http://dx.doi.org/10.1016/0045-7825(92)90193-N}

\bibitem{MuradLoula94}
M.~A. Murad, A.~F.~D. Loula, \href{http://dx.doi.org/10.1002/nme.1620370407}{On
  stability and convergence of finite element approximations of {B}iot's
  consolidation problem}, Internat. J. Numer. Methods Engrg. 37~(4) (1994)
  645--667.
\newblock \href {http://dx.doi.org/10.1002/nme.1620370407}
  {\path{doi:10.1002/nme.1620370407}}.
\newline\urlprefix\url{http://dx.doi.org/10.1002/nme.1620370407}

\bibitem{MuradLoulaThome}
M.~A. Murad, V.~Thom{\'e}e, A.~F.~D. Loula,
  \href{http://dx.doi.org/10.1137/0733052}{Asymptotic behavior of semidiscrete
  finite-element approximations of {B}iot's consolidation problem}, SIAM J.
  Numer. Anal. 33~(3) (1996) 1065--1083.
\newblock \href {http://dx.doi.org/10.1137/0733052}
  {\path{doi:10.1137/0733052}}.
\newline\urlprefix\url{http://dx.doi.org/10.1137/0733052}

\bibitem{Aguilar2008}
G.~Aguilar, F.~Gaspar, F.~Lisbona, C.~Rodrigo,
  \href{http://dx.doi.org/10.1002/nme.2295}{Numerical stabilization of {B}iot's
  consolidation model by a perturbation on the flow equation}, Internat. J.
  Numer. Methods Engrg. 75~(11) (2008) 1282--1300.
\newblock \href {http://dx.doi.org/10.1002/nme.2295}
  {\path{doi:10.1002/nme.2295}}.
\newline\urlprefix\url{http://dx.doi.org/10.1002/nme.2295}

\bibitem{1973TaylorC_HoodP-aa}
C.~Taylor, P.~Hood, A numerical solution of the {N}avier-{S}tokes equations
  using the finite element technique, Internat. J. Comput. \& Fluids 1~(1)
  (1973) 73--100.

\bibitem{1984ArnoldD_BrezziF_FortinM-aa}
D.~N. Arnold, F.~Brezzi, M.~Fortin,
  \href{http://dx.doi.org/10.1007/BF02576171}{A stable finite element for the
  {S}tokes equations}, Calcolo 21~(4) (1984) 337--344 (1985).
\newblock \href {http://dx.doi.org/10.1007/BF02576171}
  {\path{doi:10.1007/BF02576171}}.
\newline\urlprefix\url{http://dx.doi.org/10.1007/BF02576171}

\bibitem{1991BrezziF_FortinM-aa}
F.~Brezzi, M.~Fortin, Mixed and hybrid finite element methods, Springer, New
  York, 1991.

\bibitem{2013BoffiD_BrezziF_FortinM-aa}
D.~Boffi, F.~Brezzi, M.~Fortin,
  \href{http://dx.doi.org/10.1007/978-3-642-36519-5}{Mixed finite element
  methods and applications}, Vol.~44 of Springer Series in Computational
  Mathematics, Springer, Heidelberg, 2013.
\newblock \href {http://dx.doi.org/10.1007/978-3-642-36519-5}
  {\path{doi:10.1007/978-3-642-36519-5}}.
\newline\urlprefix\url{http://dx.doi.org/10.1007/978-3-642-36519-5}

\bibitem{1984BrezziF_PitkarantaJ-aa}
F.~Brezzi, J.~Pitk{{\"a}}ranta, On the stabilization of finite element
  approximations of the {S}tokes equations, in: Efficient solutions of elliptic
  systems ({K}iel, 1984), Vol.~10 of Notes Numer. Fluid Mech., Friedr. Vieweg,
  Braunschweig, 1984, pp. 11--19.

\bibitem{1993BaiocchiC_BrezziF_FrancaL-aa}
C.~Baiocchi, F.~Brezzi, L.~P. Franca,
  \href{http://dx.doi.org/10.1016/0045-7825(93)90119-I}{Virtual bubbles and
  {G}alerkin-least-squares type methods ({G}a.{L}.{S}.)}, Comput. Methods Appl.
  Mech. Engrg. 105~(1) (1993) 125--141.
\newblock \href {http://dx.doi.org/10.1016/0045-7825(93)90119-I}
  {\path{doi:10.1016/0045-7825(93)90119-I}}.
\newline\urlprefix\url{http://dx.doi.org/10.1016/0045-7825(93)90119-I}

\bibitem{1984VerfurthR-aa}
R.~Verf{\"u}rth, Error estimates for a mixed finite element approximation of
  the {S}tokes equations, RAIRO Anal. Num\'er. 18~(2) (1984) 175--182.

\bibitem{1990BankR_WelfertB-aa}
R.~E. Bank, B.~D. Welfert,
  \href{http://dx.doi.org/10.1016/0045-7825(90)90124-5}{A comparison between
  the mini-element and the {P}etrov-{G}alerkin formulations for the generalized
  {S}tokes problem}, Comput. Methods Appl. Mech. Engrg. 83~(1) (1990) 61--68.
\newblock \href {http://dx.doi.org/10.1016/0045-7825(90)90124-5}
  {\path{doi:10.1016/0045-7825(90)90124-5}}.
\newline\urlprefix\url{http://dx.doi.org/10.1016/0045-7825(90)90124-5}

\bibitem{1986HughesT_FrancaL_BalestraM-aa}
T.~J.~R. Hughes, L.~P. Franca, M.~Balestra,
  \href{http://dx.doi.org/10.1016/0045-7825(86)90025-3}{A new finite element
  formulation for computational fluid dynamics. {V}. {C}ircumventing the
  {B}abu\v ska-{B}rezzi condition: a stable {P}etrov-{G}alerkin formulation of
  the {S}tokes problem accommodating equal-order interpolations}, Comput.
  Methods Appl. Mech. Engrg. 59~(1) (1986) 85--99.
\newblock \href {http://dx.doi.org/10.1016/0045-7825(86)90025-3}
  {\path{doi:10.1016/0045-7825(86)90025-3}}.
\newline\urlprefix\url{http://dx.doi.org/10.1016/0045-7825(86)90025-3}

\bibitem{1988BrezziF_DouglasJ-aa}
F.~Brezzi, J.~Douglas, Jr.,
  \href{http://dx.doi.org/10.1007/BF01395886}{Stabilized mixed methods for the
  {S}tokes problem}, Numer. Math. 53~(1-2) (1988) 225--235.
\newblock \href {http://dx.doi.org/10.1007/BF01395886}
  {\path{doi:10.1007/BF01395886}}.
\newline\urlprefix\url{http://dx.doi.org/10.1007/BF01395886}

\bibitem{2006ThomeeV-aa}
V.~Thom{\'e}e, Galerkin finite element methods for parabolic problems, 2nd
  Edition, Vol.~25 of Springer Series in Computational Mathematics,
  Springer-Verlag, Berlin, 2006.

\bibitem{NAG:NAG1062}
J.~B. Haga, H.~Osnes, H.~P. Langtangen, On the causes of pressure oscillations
  in low-permeable and low-compressible porous media, International Journal for
  Numerical and Analytical Methods in Geomechanics 36~(12) (2012) 1507--1522.
\newblock \href {http://dx.doi.org/10.1002/nag.1062}
  {\path{doi:10.1002/nag.1062}}.

\bibitem{Mandel}
J.~Mandel, Consolidation des sols (\'{e}tude de math\'{e}matique),
  G\'{e}otechnique 3 (1953) 287--299.

\bibitem{Skempton}
A.~W. Skempton, The pore-pressure coefficients {A} and {B}, G\'{e}otechnique 4
  (1954) 143--147(4).

\bibitem{abousleiman1996mandel}
Y.~Abousleiman, A.-D. Cheng, L.~Cui, E.~Detournay, J.-C. Roegiers, Mandel's
  problem revisited, Geotechnique 46~(2) (1996) 187--195.

\bibitem{coussy1995}
O.~Coussy, Mechanics of Porous Continua, Wiley, 1995.

\bibitem{1978CiarletP-aa}
P.~G. Ciarlet, The finite element method for elliptic problems, North-Holland
  Publishing Co., Amsterdam-New York-Oxford, 1978, studies in Mathematics and
  its Applications, Vol. 4.

\end{thebibliography}





\end{document}